\documentclass[12pt]{amsart}
\usepackage{amsmath}
\usepackage{amssymb}
\usepackage{amsfonts}
\usepackage{latexsym}
\usepackage{amscd}
\usepackage[mathscr]{euscript}
\usepackage{enumitem}
\usepackage{xy} \xyoption{all}
\usepackage{pdfsync}
\usepackage[active]{srcltx}
\usepackage{amssymb} 
\def\acts{\curvearrowright}

\newcommand{\Breg}{{\mathcal{B}_{reg}}}

\newcommand{\tot}{{\mathbf{1}}}
\newcommand{\N}{{\mathbb{N}}}

\newcommand{\Z}{{\mathbb{Z}}}

\newcommand{\C}{{\mathbb{C}}}
\newcommand{\Cu}{{\xi}}
\newcommand{\Sa}{{\mathcal{S}}}
\newcommand{\Ri}{{\mathcal{R}}}
\newcommand{\D}{{\mathcal{D}}}
\newcommand{\Do}{{\mathcal{D}_\alpha}}
\newcommand{\Ra}{{\mathcal{R}_\alpha}}

\newcommand{\Gti}{{\mathcal{G}_{\text{tight}}}}
\newcommand{\F}{{\mathcal{F}}}
\newcommand{\He}{{\mathcal{H}}}
\newcommand{\E}{{\mathcal{E}}}
\newcommand{\ET}{\mathcal{E}(T)}

\newcommand{\Ne}{{\mathcal{N}}}
\newcommand{\Me}{{\mathcal{M}}}
\newcommand{\Labe}{{\mathcal{L}}}
\newcommand{\Clab}{C^*(\B,\Labe,\theta)}

\newcommand{\B}{{\mathcal{B}}}
\newcommand{\Id}{{\mathcal{I}}}

\newcommand{\uloopr}[1]{\ar@'{@+{[0,0]+(-4,5)}@+{[0,0]+(0,10)}@+{[0,0] +(4,5)}}^{#1}}
\newcommand{\uloopd}[1]{\ar@'{@+{[0,0]+(5,4)}@+{[0,0]+(10,0)}@+{[0,0]+ (5,-4)}}^{#1}}
\newcommand{\dloopr}[1]{\ar@'{@+{[0,0]+(-4,-5)}@+{[0,0]+(0,-10)}@+{[0, 0]+(4,-5)}}_{#1}}
\newcommand{\dloopd}[1]{\ar@'{@+{[0,0]+(-5,4)}@+{[0,0]+(-10,0)}@+{[0,0 ]+(-5,-4)}}_{#1}}

\newcommand{\luloop}[1]{\ar@'{@+{[0,0]+(-8,2)}@+{[0,0]+(-10,10)}@+{[0, 0]+(2,2)}}^{#1}}

\newtheorem{lem}{Lemma}[section]
\newtheorem{corol}[lem]{Corollary}
\newtheorem{theor}[lem]{Theorem}
\newtheorem{prop}[lem]{Proposition}
\theoremstyle{definition}
\newtheorem{defi}[lem]{Definition}

\newtheorem{exem}[lem]{Example}

\newtheorem{nota}[lem]{Notation}
\newtheorem{rema}[lem]{Remark}

\newtheorem{noname}[lem]{}

\begin{document}

\title[$C^*$-algebras associated to Boolean dynamical systems]{$C^*$-algebras associated to Boolean dynamical systems}%

\author{Toke Meier Carlsen}
\address{Department of Science and Technology, University of the Faroe Islands\\
N\'{o}at\'{u}n 3\\
F0-100 T\'{o}rshavn\\
The Faroe Islands} \email{tokemc@setur.fo}

\author{Eduard Ortega}
\address{Department of Mathematical Sciences\\
NTNU\\
NO-7491 Trondheim\\
Norway } \email{Eduardo.Ortega@math.ntnu.no}

\author{Enrique Pardo}
\address{Departamento de Matem\'aticas, Facultad de Ciencias\\ Universidad de C\'adiz, Campus de
Puerto Real\\ 11510 Puerto Real (C\'adiz)\\ Spain.}
\email{enrique.pardo@uca.es}\urladdr{https://sites.google.com/a/gm.uca.es/enrique-pardo-s-home-page/}

\dedicatory{Dedicated to the memory of Uffe Haagerup}

\thanks{The third-named author was partially supported by PAI III grants FQM-298 and P11-FQM-7156 of the Junta de Andaluc\'{\i}a, and by the DGI-MINECO and European Regional Development Fund, jointly, through Project MTM2014-53644-P}



\date{\today}

\begin{abstract}
The goal of these notes is to present the $C^*$-algebra $\Clab$ of a Boolean dynamical system $(\mathcal{B}, \mathcal{L}, \theta)$, that generalizes the $C^*$-algebra associated to Labelled graphs introduced by Bates and Pask, and to determine its simplicity, its gauge invariant ideals, as well as compute its K-Theory.

\end{abstract}

\maketitle

\section{Introduction}

In 1980 Cuntz and Krieger \cite{CuKr} associated a $C^*$-algebra $\mathcal{O}_A$ to a shift of finite type with transition matrix $A$. Various authors --including Bates, Fowler, Kumjian, Laca, Pask and  Raeburn-- extended the original construction to more general subshifts associated with directed graphs, giving origin to the graph $C^*$-algebra $C^*(E)$ of a directed graph $E$ (see e.g. \cite{FoLaRa,KPRR}). Using a different approach, Exel and Laca \cite{ExLa} generalize Cuntz-Krieger algebras, by associating a $C^*$-algebra to an infinite matrix which $0$ and $1$ entries.  Later, Tomforde \cite{Tom} introduced the class of ultragraph algebras in order to unify Exel-Laca algebras and graph $C^*$-algebras. Also, motivated by Cuntz-Krieger construction, Matsumoto \cite{Mat} introduced a $C^*$-algebra associated with a general two-sided subshift over a finite alphabet. Later, the first named author \cite{Car} extended Matsumoto's construction, by constructing the $C^*$-algebra $\mathcal{O}_\Lambda$ associated with a general one-sided subshift $\Lambda$ over a finite alphabet. 

One of the the underlying ideas of associating a $C^*$-algebra to a dynamical system comes from the Franks classification of irreducible shifts of finite type up to flow equivalence \cite{Franks}. This classification use the Bowen-Franks group of the shift space, that turns out to be the $K_0$ group of the associated Cuntz-Krieger algebra \cite{CuKr}. Therefore, the idea was to study the connection between classification of shift spaces and classification of $C^*$-algebras. Following this point of view, the recent results of Matsumoto and Matui \cite{MaMa} characterize continuous orbit equivalence of shifts of finite type by using $K$-theoretical invariants of the associated $C^*$-algebra. It is natural to try to extend the scope of this strategy to classify shift space over a countable alphabet. By adapting the left-Krieger cover construction given in \cite{Kri}, any shift space over a countable alphabet may be presented by a left-resolving labelled graph. Thus, in the same spirit of the previous constructions, labelled graph algebras, introduced by Bates and Pask in \cite{BPI}, provided a method for associating a $C^*$-algebra to a shift space over a countable alphabet. The class of labelled graph $C^*$-algebras contains, in particular, all the above mentioned classes of $C^*$-algebra. Properties like simplicity, ideal structure and purely infinity was studied in \cite{BPII,JeKiPa}  and the computation of the $K$-theory was achieved in \cite{BPIII}.\vspace{.2truecm}

The original goal of the present paper was to continue the study of the labelled graph $C^*$-algebras, by characterizing them as $0$-dimensional topological graphs \cite{KatsuraI}. However, the topological graph $E$ associated to the data of the labelled graph is just a realization of a Boolean algebra of a family of subsets of vertices of $E$, plus some partial actions given by the arrows of $E$. Thus, we adapt the labelled graph $C^*$-algebra construction, as well as our topological graph characterization, to the context of a $C^*$-algebra associated to a general family of partial actions over a fixed Boolean algebra (we call it a Boolean dynamical system). This class of $C^*$-algebras, that we call Boolean Cuntz-Krieger algebras associated with a Boolean dynamical systems, includes labelled graph $C^*$-algebras, homeomorphism $C^*$-algebras over $0$-dimensional compact spaces, and graph $C^*$-algebras, among others. Essentially, it is not a new class of $C^*$-algebras, since they are ($0$-dimensional) algebras over topological graphs, a class deeply studied by Katsura \cite{KatsuraI,KatsuraIII}. However, the advantage of our approach is that we can skip to deal with the topology of the graph, and instead can concentrate only in combinatorial properties of actions over a Boolean algebra. In particular, we can use a different picture when studying $C^*$-algebras associated to combinatorial objects, by using groupoid $C^*$-algebras. This is a classical approach, used by Kumjian, Pask, Raeburn and Renault \cite{KPRR} when studying graph $C^*$-algebras. This approach attained a new level of efficiency when Exel \cite{Exel} developed a huge machinery that helps to represent any ``combinatorial'' $C^*$-algebra as a full groupoid $C^*$-algebra. The strategy is to  associate to the $C^*$-algebra an $\ast$-inverse semigroup (see e.g. \cite{Lawson}) and a ``tight'' representation (i.e. a representations preserving additive identities on pairwise orthogonal idempotents). When this is possible, there is a standard way of producing a \'etale, second countable topological groupoid which full $C^*$-algebra is isomorphic to the original $C^*$-algebra under consideration. In the case of Boolean Cuntz-Krieger algebras associated to Boolean dynamical system this strategy works, and so we can use all the machinery developed by Exel \cite{Exel, ExelTight} for analyze the structure of the algebras under study. Recent examples of application of such an strategy are \cite {ExPa, ExPa2}.\vspace{.2truecm}

The contents of this paper can be summarized as follows: In Section 2 we recall some Boolean algebra Theory. In particular, we summarize some well-known results about the topology of the space of characters (the Stone's spectrum) of a Boolean algebra. In Section 3 we define Boolean dynamical systems, that are families of partial actions on a Boolean algebra, and their representations in a $C^*$-algebra; the $C^*$-algebra associated to the universal  representation will be the Boolean Cuntz-Krieger algebra. We state the existence of a universal representation and the gauge uniqueness theorem, that will be proved later. In Section 4 we recall the definition of Katsura's topological graph. When $E$ is a $0$-dimensional space, i.e. both the vertex and edge spaces are $0$-dimensional, we construct a Boolean dynamical system that can be represented in the associated topological graph $C^*$-algebra $\mathcal{O}(E)$. In Section 5 we focus on finding a universal representation of a given Boolean dynamical system. This is achieved by constructing a compactly supported $0$-dimensional  topological graph with the data of the Boolean dynamical system, and defining a representation of the Boolean dynamical system in the topological graph $C^*$-algebra. We conclude proving that the Boolean Cuntz-Krieger algebras are isomorphic to a $0$-dimensional topological graph $C^*$-algebra, and using this characterization to compute its $K$-Theory. In Sections 6,7 and 8 we apply Exel's machinery to Boolean Cuntz-Krieger algebras. To this end, we first define an $\ast$-inverse semigroup associated to a Boolean dynamical system, and then we prove that the $C^*$-algebra associated to the universal tight representation of this $\ast$-inverse semigroup is isomorphic to the unitization of our Boolean Cuntz-Krieger algebra. Finally, we define the groupoid of germs of the partial actions of the $\ast$-inverse semigroup on the space of tight filters defined over its semilattice of idempotents. Thus, by using Exel's results \cite{Exel, Reconst}, we can see that the Boolean Cuntz-Krieger algebra is the full $C^*$-algebra of this groupoid. This allows us to work in the realm of groupoid $C^*$-algebra, and to use the known results on this class to characterize properties of Boolean Cuntz-Krieger algebras. In particular, we use the groupoid characterization of the Boolean Cuntz-Krieger algebras in Section 9 to characterize its simplicity in terms of intrinsic properties of the associated Boolean dynamical system. A similar approach was used by Marrero and Muhly for ultragraph $C^*$-algebras \cite{MaMu}, although the way they constructed the groupoid is quite different to ours; also, after the final version of the present paper was ready, we were aware of Boava, de Castro and Mortari's work for labelled graph $C^*$-algebras \cite{BCM}, were they constructed an inverse semigroup in a similar (although abstract) way as our inverse semigroup $T$ (see Section 6), but they concentrated their attention in understanding the nature of the tight spectra, and do not work out either an associated groupoid or a groupoid picture of labelled $C^*$-algebras associated to it. In Section 10 we define an admissible pair for a Boolean dynamical system, and we state an order lattice bijection between the admissible pairs and the gauge invariant ideals of the Boolean Cuntz-Krieger algebras. Finally, we realize the quotient of a Boolean Cuntz-Krieger algebra modulo a gauge invariant ideal as the Boolean Cuntz-Krieger algebra of another induced Boolean dynamical system.

\section{Boolean $C^*$-algebras}
The main objects of this paper is a boolean algebra and its associated $C^*$-algebras. We will first introduce basic definitions and results, mostly well-known, and then we will focus on finding a representation of a boolean algebra as the set of clopen subsets of a topological space (Stone's representation). It turns out that the points of this topological space are the set of the ultrafilters of the elements of the boolean algebra.

\begin{defi}\label{def_Boolean} A \emph{Boolean algebra} is a quadruple $(\B,\cap,\cup,\setminus)$, where $\B$ is a set with a distinguished element $\emptyset\in \B$, that we called \emph{empty}, and maps
$\cup:\B\times \B\rightarrow \B$,  $\cap:\B\times \B\rightarrow \B$ and $\setminus:\B\times \B\rightarrow\B$ that we call the \emph{union}, \emph{intersection} and \emph{relative complement} maps, satisfying the standard axioms (see \cite[Chapter 2]{GiHa}).
%\begin{description}
%\item[B1] $A\setminus\emptyset=A$ and $\emptyset \setminus A=\emptyset$,
%\item[B2] $A\cap \emptyset=\emptyset$ and $A\cup \emptyset=A$,
%\item[B3] $(A\setminus B)\cap C=(A\cap C)\setminus B$,
%\item[B4] $(A\setminus B)\setminus C=A\setminus (B\cup C)$,
%\item[B5] $(A\cup B)\setminus C=(A\setminus C)\cup(B\setminus C)$,
%\item[B6] $A\setminus (B\setminus C)=(A\setminus B)\cup (A\cap C)$,
%\item[B7] $A\setminus (B\cap C)=(A\setminus B)\cup (A\setminus C)$,
%\item[B8] $A\cap A=A\cup A=A$,
%\item[B9] $A\cap B=B\cap A$ and $A\cup B=B\cup A$,
%\item[B10] $A\cap (B\cap C)=(A\cap B)\cap C$ and $A\cup (B\cup C)=(A\cup B)\cup C$,
%\item[B11] $A\cap (B\cup C)=(A\cap B)\cup (A\cap C)$ and $A\cup (B\cap C)=(A\cup B)\cap (A\cup C)$.
%\end{description} 
The Boolean algebra $\B$ is \emph{unital} if does exist $\tot\in \B$ such that $\tot\cup A=\tot$ and $\tot\cap A=A$ for every $A\in \B$.  A \emph{boolean homomorphism} is a map $\phi$ from one boolean algebra $\B_1$ to another boolean algebra $\B_2$ such that $\phi(A\cap B)=\phi(A)\cap\phi(B)$, $\phi(A\cup B)=\phi(A)\cup\phi(B)$, and $\phi(A\setminus B)=\phi(A)\setminus\phi(B)$ for all $A,B\in\B_1$.
\end{defi}
\begin{rema} What we call a Boolean algebra is sometimes called a Boolean ring, and  what we call a unital Boolean algebra is sometimes simple called a Boolean algebra. The theories of Boolean algebras and Boolean rings are very closely related;
in fact, they are just different ways of looking at the same subject. See \cite{GiHa} for further explanation.
\end{rema}
A subset $\B'\subseteq \B$ is called a \emph{Boolean subalgebra} if $\B'$ is closed by the union, intersection and the relative complement operations.

Given a Boolean algebra $\B$, we can define the following partial order: given $A,B\in \B$ 
$$A\subseteq B\qquad \text{if and only if}\qquad A\cap B=A\,.$$
Then $(\B,\subseteq)$ is a partially ordered set.

\begin{defi} An element $B\in\B$ is called a  \emph{least upper-bound}  for  $\{A_\lambda\}_{\lambda\in \Lambda}$ with $A_\lambda\in \B$ if it is the least element of $\B$ satisfying $A_\lambda \subseteq B$ for every $\lambda\in \Lambda$. We will write the unique least upper-bound as $\bigcup\limits_{\lambda\in \Lambda} A_\lambda$. 
\end{defi}

Observe that least upper-bound do not necessarily exist, but  if $|\Lambda|<\infty$ then the least upper-bound of $\{A_\lambda\}_{\lambda\in \Lambda}$ is $\bigcup\limits_{\lambda\in \Lambda} A_\lambda$.

\begin{defi}
Let $\B$ be a Boolean algebra. We say that a subset $\Id
$ of $\B$ is an \emph{ideal} if given $A,B\in \B$, then:
\begin{enumerate}
\item if $A,B\in \Id$ then $A\cup B\in \Id$,
\item if $A\in \Id$ then $A\cap B\in \Id$.
\end{enumerate}  
\end{defi}

Observe that in particular an ideal $\Id$ of a Boolean algebra $\B$ is a Boolean subalgebra.

Given $A\in \B$ we define $\Id_A:=\{B\in\B: B\subseteq A\}$, that is the  ideal generated by $A$.

\begin{defi} Let $\B$ be the Boolean algebra  and let $\Id$ be an ideal of $\B$. Given $A,B\in \B$, we define the following equivalent relation: $A\sim B$ if and only if there exists $A',B'\in \Id$ such that $A\cup A'=B\cup B'$. We define by $[A]$ the set of all the elements of $\B$ equivalent to $A$, and we denote by $\B/\Id$ the set of all equivalent classes of $\B$. Moreover, we say that $[A]\subseteq [B]$ if and only if there exists $H\in \Id$ such that $A\subseteq B\cup H$.
\end{defi}

\begin{defi} Let $\B$ be a Boolean algebra. A subset $\Cu\subseteq \B$ is called a \emph{filter of $\B$} if it has the following properties:
\begin{description}
\item[F0] $\emptyset\notin \Cu$,
\item[F1] given $B\in \B$ and $A\in \Cu$ with  $A\subseteq B$ then $B\in \Cu$,
\item[F2] given $A,B\in \Cu$ then $A\cap B\in \Cu$.
\end{description}
If moreover $\Cu$ satisfies:
\begin{description}
\item[F3] given $A\in \Cu$ and $B,B'\in \B$ with $A=B\cup B'$ then either $B\in \Cu$ or $B'\in \Cu$,
\end{description}
then it is called an \emph{ultrafilter of $\B$}.
\end{defi}

Given two filters $\Cu_1$ and $\Cu_2$ of $\B$, we say that $\Cu_1\subseteq \Cu_2$ if every $A_1\in \Cu_1$ is also in $\Cu_2$. This defines a partial order on the set of filters of $\B$. Then, an easy application of the Zorn's Lemma shows that an ultrafilter as a maximal filter.\vspace{.2truecm}

We will denote by $\widehat{\B}$ the set of ultrafilters of $\B$. Given any $A\in \B$, we define  the \emph{cylinder set of $A$} as $Z(A):=\{\Cu\in \widehat{\B}: A\in \Cu\}$. It is an easy exercise to show that the family $\{Z(A):A\in \B\}$ defines a topology of $\widehat{\B}$, in which the sets $Z(A)$ are clopen and compact  (see for example \cite[Chapter 34]{GiHa}). We will call $\widehat{\B}$ equiped with this topology the \emph{Stone's spectrum of $\B$}.

\begin{exem}\label{example_1}  Let $X=\N$ and let $\B:=\{F\subseteq \N: F \text{ finite }\}\cup \{\N\setminus F : F \text{ finite }\}$. Clearly, $\B$ is a Boolean algebra.  We will now describe the Stone spectrum for $\widehat{\B}$ of $\B$.

For $i\in\N$, let $$\Cu_i=\{A\in \B:i\in A\}\,,$$ and let 
$$\Cu_\infty:=\{A\in \B: \exists N\in \N\text{ such that }k\in A \text{ }\forall k\geq N \}\,.$$
It is easy to check that $\Cu_\infty$ and each $\Cu_i$ are ultrafilters of $\B$.

We claim that $\widehat{\B}= \{\Cu_i:i\in \N\cup\{\Cu_\infty\}\}$. To see this, let $\Cu$ be an ultrafilter of $\B$ such that $\bigcap\limits_{A\in \Cu} A=\emptyset$. We will show that $\Cu=\Cu_\infty$. Given $k\in \N$, let us denote by $[k,\infty)$ the set $\N\setminus \{1,\ldots, k-1\}\in \B$. Observe that, since $\bigcap\limits_{A\in \Cu} A=\emptyset$, given any $k\in\N$ there exists $n_k\in\N$ and $A_1,\ldots,A_{n_k}\in \Cu$ such that $A_1\cap \cdots\cap A_{n_k}\subseteq [k,\infty)$. Therefore, by $\textbf{F1}$, $[k,\infty)\in \Cu$ for every $k\in \N$. Now, given any $A\in \Cu_\infty$, there exists $k\in \N$ such that $[k,\infty)\subseteq A$, whence $A\in \Cu$ by $\textbf{F1}$. On the other side, given any $A\in \Cu$, we claim that $|A|=\infty$. Otherwise, if $|A|=n<\infty$, then there exist $A_1,\ldots,A_n\in \Cu$ such that $A\cap A_1\cap\cdots\cap A_n=\emptyset$, contradicting condition $\textbf{F2}$. Thus, $|A|=\infty$. Therefore, since $A\in \B$, we have that $A=\N\setminus F$ for some finite set $F$ of $\N$. Then, there exists $k\in \N$ such that $[k,\infty)\subseteq A$. So, since $[k,\infty)\in \Cu_\infty$, condition $\textbf{F1}$ says that $A\in \Cu_\infty$ too. Thus $\Cu=\Cu_\infty$.

Therefore, we have that $\widehat{\B}= \{\Cu_i:i\in \N\cup\{\infty\}\}$. Finally observe that, with the induced  topology, we have that $\widehat{\B}$ is the one point compactification of $\N$.
\end{exem}

Let $\B$ be a Boolean algebra, and let $\Id$ be an ideal of $\B$. Then, the map $\iota:\widehat{\Id}\longrightarrow \widehat{\B}$ defined by $\iota(\Cu)=\{A\in \B: B\subseteq A\text{ for some }B\in \Cu\}$ is injective. So, given $A\in\B$, we have that $Z(A)=\iota(\widehat{\Id_A})$. Therefore, we will identify $\widehat{\Id_A}$ with $Z(A)$, so $\widehat{\Id_A}\subseteq \widehat{\B}$ for every $A\in\B$.

Moreover, there exists a bijection between the ultrafilters of $\B/\Id$ and the ultrafilters of $\B$ that do not contain any element of $\Id$. Therefore, the natural map $\pi:\B\rightarrow \B/\Id$ is surjective, and it induces an injective map $\widehat{\pi}: \widehat{\B/\Id }\longrightarrow \widehat{\B}$ given by $[\Cu]\rightarrow\pi^{-1}([\Cu])=\{A\in \B: [A]\in [\Cu]\}$ for every $[\Cu]\in \widehat{\B/\Id}$. Therefore, we will identify $\widehat{\B/\Id}$ with $\widehat{\pi}(\widehat{\B/\Id})$, so $\widehat{\B/\Id}\subseteq \widehat{\B}$.

\begin{rema}\label{rema_sub}
Let $\Id$ be an ideal of $\B$, then $\widehat{\Id}\cap \widehat{\B/\Id}=\emptyset$ and $\widehat{\B}=\widehat{\Id}\cup \widehat{\B/\Id}$.
\end{rema}

\begin{lem}\label{induced_map}
Let $\B_1$ and $\B_2$ be two Boolean algebras, and let  $\varphi:\B_1\rightarrow \B_2$ be a Boolean algebras homomorphism with  $\varphi(\emptyset)=\emptyset$ such that for every $A\in\B_2$ there exists $B\in\B_1$ such that $A\subseteq \varphi(B)$.

 Then this map induces continuous  map $\widehat{\varphi}:\widehat{\B_2}\rightarrow \widehat{\B_1}$ defined as $$\widehat{\varphi}(\xi)=\{A\in \B_1:\varphi(A)\in\xi \}$$ for every $\xi\in \B_2$. 
\end{lem}
\begin{proof}
Let $\varphi:\widehat{\B_2}\longrightarrow\widehat{\B_2}$ be the map given by $\varphi(\Cu)=\{A\in\widehat{\B_1}: \varphi(A)\in \Cu\}$. It is routine to check that $\{A\in\widehat{\B_1}: \varphi(A)\in \Cu\}$ is an ultrafilter of $\widehat{\B_1}$. Thus, $\widehat{\varphi}$ is a well-defined map. If $A\in \B_2$, then we claim that $\widehat{\varphi}^{-1}(Z(A))=\{\Cu\in\widehat{\B_2}: \varphi(A)\in \Cu\}$. 
Indeed, the inclusion $\subseteq$ is clear. For the inclusion $\supseteq$, let $\xi\in\widehat{\B_2}$ with $\varphi(A)\in \xi$, and let us define the set $\F=\{B\in \B_1: \varphi(B)\in \Cu\}$. By hypothesis, we have that $A\in\F$, so $\textbf{F0}$ is satisfied. $\textbf{F1}$ and $\textbf{F2}$ follows because of  conditions $\textbf{F1}$ and $\textbf{F2}$ of $\Cu$, and the fact that $\varphi$ preserves intersections. Thus, $\F$ is a filter. Then by an easy application of the Zorn's Lemma we can find a maximal filter $\zeta$ containing $\F$. Thus, $\zeta\in\widehat{\B_1}$ such that $\varphi(B)\in\xi$ for every $B\in \zeta$, so $\widehat{\varphi}(\zeta)=\Cu$ with $\zeta\in Z(A)$, as desired.

Then $\widehat{\varphi}^{-1}(Z(A))=\{\Cu\in\widehat{\B_2}: \varphi(A)\in \Cu\}=Z(\varphi(A))$ that is an open subset. Thus, $\widehat{\varphi}$ is a continuous map.
\end{proof}

Given a Boolean algebra $\B$ and given $A\in \B$ we let $\chi_A:\B\rightarrow \mathbb{C}$ denote the function defined on $\B$ by

$$\chi_A(B)=\left\lbrace \begin{array}{ll} 1 & \text{if } A\cap B\neq \emptyset \\ 0 & \text{otherwise }\end{array}\right. .$$

We will regard $\chi_A$ as an element of the $C^*$-algebra of bounded operators on $\ell^2(\B)$. 

\begin{defi}
Let $\B$ be a Boolean algebra. Then we define the \emph{Boolean $C^*$-algebra of $\B$} as the sub-$C^*$-algebra of the $B(\ell^2(\B))$ generated by $\{\chi_A : A \in  \B\}$.  We denote it as $C^*(\B)$.
\end{defi}

$C^*(\B)$ is a commutative $C^*$-algebra, and given $A,B\in \B$ we have that 
$$\chi_A\cdot \chi_B=\chi_{A\cap B}\qquad \text{and}\qquad \chi_{A\cup B}=\chi_A+\chi_B-\chi_{A\cap B}\,,$$
where $\chi_{\emptyset}=0$. Thus, $C^*(\B)=\overline{\text{span}}\{\chi_A:A\in\B\}$.

First, recall that the spectrum of $C^*(\B)$, denoted by $\widehat{C^{*}(\B)}$,  is the set of characters of $C^*(\B)$. Observe that an additive map $\eta:C^*(\B)\longrightarrow \C$ is a $*$-homomorphism if and only if given $A,B\in \B$
$$\begin{array}{ll} (C1) &  \eta(\chi_A)\eta(\chi_B)=\eta(\chi_{A\cap B}) \\
(C2) & \eta(\chi_{A\cup B})=\eta(\chi_A)+\eta(\chi_B)-\eta(\chi_{A\cap B})\,.\end{array}$$

If $\eta$ is a character of $C^*(\B)$, then we define 
$$\Cu_\eta:=\{A\in \B: \eta(\chi_A)=1\}\,.$$
Recall that, since $\chi_A$ is a projection for every $A\in \B$ and $\eta$ is a $*$-homomorphism, $\eta(\chi_A)$ is either $0$ or $1$. Then the following lemma is straightforward.
\begin{lem}\label{lemma_1} If  $\eta$ is a character of $C^*(\B)$, then $\Cu_\eta$ is an ultrafilter of $\B$.
\end{lem}

Given an ultrafilter $\Cu$ of $\B$, we define the unique additive map $\eta_\Cu:C^*(\B)\longrightarrow \C$ such that 
$$\eta_\Cu(\chi_A)=\left\lbrace \begin{array}{ll} 1 & \text{if }A\in \Cu \\ 0 & \text{if }A\notin \Cu\end{array}\right. $$

\begin{lem}\label{lemma_2} $\eta_\Cu$ is a character of $C^*(\B)$.
\end{lem}
\begin{proof} We must check that $\eta_\Cu$ satisfies $C1$ and $C2$. For $C1$, let $A,B\in \B$, and recall that $\chi_A\cdot \chi_B=\chi_{A\cap B}$. First, suppose that $\eta_\Cu(\chi_{A\cap B})=0$. Therefore, $A\cap B\notin \Cu$ and hence, by $\textbf{F2}$, either $A$ or $B$ are not in $\Cu$. Thus, $\eta_\Cu(A)\eta_\Cu(A)=0=\eta_\Cu(\chi_{A\cap B})$, as desired. Now, suppose that $\eta_\Cu(\chi_{A\cap B})=1$, so $A\cap B\in \Cu$. Therefore, by $\textbf{F1}$, it follows that $A,B\in \Cu$ too, and hence $\eta_\Cu(A)\eta_\Cu(A)=1=\eta_\Cu(\chi_{A\cap B})$, as desired. Thus, $C1$ is verified.

For $C2$, let $A,B\in \B$. First, suppose that $\eta_\Cu(\chi_{A\cup B})=0$. So, $A\cup B\notin \Cu$, and since $A,B,A\cap B\subseteq A\cup B$, it follows from $\textbf{F1}$  that $A,B,A\cap B\notin \Cu$. Therefore, 
$$\eta_\Cu(\chi_{A\cup B})=0=\eta_\Cu(\chi_A)+\eta_\Cu(\chi_B)-\eta(\chi_{A\cap B})\,.$$
Finally, suppose that $A\cup B\in \Cu$. Hence, by $\textbf{F3}$, either $A$ or $B$ belongs to $\Cu$. First suppose that $A,B\in \Cu$. Then, by $\textbf{F2}$ so does $A\cap B$. Therefore, 
$$\eta_\Cu(\chi_{A\cup B})=1+1-1=\eta_\Cu(\chi_A)+\eta_\Cu(\chi_B)-\eta(\chi_{A\cap B})\,,$$
as desired.
Now, suppose that  $A\in \Cu$ but $B\notin \Cu$. By $\textbf{F2}$, we have that $A\cap B\notin \Cu$, so 
$$\eta_\Cu(\chi_{A\cup B})=1+0-0=\eta_\Cu(\chi_A)+\eta_\Cu(\chi_B)-\eta(\chi_{A\cap B})\,,$$
as desired.
\end{proof}

The following result follows directly from the definitions.

\begin{prop}\label{proposition_1} Let $\Cu$ be an ultrafilter of $\B$ and let $\eta$ a character of $\mathcal{A}$. Then $\Cu_{\eta_\Cu}=\Cu$ and $\eta_{\Cu_\eta}=\eta$. Therefore, there is a bijection between the ultrafilters of $\B$ and the characters of $\mathcal{A}$.
\end{prop}

By Proposition \ref{proposition_1} there is a bijection between $\widehat{\B}$ and the set of characters of $C^*(\B)$. 
Recall that by the Gelfand-Naimark Theorem $C^*(\B)\cong C_0(\widehat{C^{*}(\B)})$, where $\widehat{C^{*}(\B)}$ has the Jacobson topology. Recall that, given a subset of $Y$ of $\widehat{C^{*}(\B)}$, we define the closure of $Y$ as $\{\eta\in \widehat{C^{*}(\B)}: \text{Ker }\eta\supseteq \bigcap\limits_{\rho\in Y}\text{Ker }\rho\}$.

\begin{prop}[Stone's Representation Theorem]\label{proposition_2}  Let $\B$ be a Boolean algebra and let $\widehat{\B}$ be the Stone's spectrum of $\B$. Then $\widehat{C^{*}(\B)}$ and $\widehat{\B}$ are homeomorphic topological spaces. Therefore, $C^{*}(\B)\cong C_0( \widehat{\B})$. 
\end{prop}
\begin{proof} First recall that, using Proposition \ref{proposition_1}, we identify a character $\eta$ of $C^{*}(\B)$ with its associated ultrafilter $\Cu_\eta$. Observe that, given $\Cu\in \widehat{\B}$, we have $\text{Ker }\eta_\Cu=\{ \chi_B : B\notin \Cu\}$. Then, given a set $Y\subseteq \widehat{\B}$, we define 
$$I_Y:=\bigcap\limits_{\Cu\in Y} \text{Ker }\eta_\Cu=\overline{\text{span}}\{\chi_B:B\notin \Cu,\,\forall \Cu\in Y\}\,.$$
Using the definitions, it is straightforward to check that $I_Y=\overline{\text{span}}\{\chi_B:B\in \B,\,Y\cap Z(B)=\emptyset\}$.

Let $\{A_\lambda\}_{\lambda\in \Lambda}$ be a family of elements of $\B$ and let us consider $V:=\bigcup\limits_{\lambda\in\Lambda}Z(A_\lambda)$.  We will prove that $Y:=\widehat{\B}\setminus V$ is closed in the Jacobson topology, whence every closed subset of $\widehat{\B}$  is also closed with respect to the Jacobson topology. Hence, $I_Y=\overline{\text{span}}\{\chi_B:B\in \B,\,Z(B)\subseteq V\}$. Then, the closure of $Y$ with respect the Jacobson topology is the set
$$\{\Cu\in \widehat{\B}: \text{Ker }\eta_\Cu\supseteq I_Y\}=\{\Cu\in \widehat{\B}: \text{ if }B\in \Cu\text{ then }Z(B)\nsubseteq V\}\,.$$
Let $\Cu\notin Y$ but in the closure of $Y$ with respect to the Jacobson topology. Then, $\Cu\in V=\bigcup\limits_{\lambda\in\Lambda}Z(A_\lambda)$. So, there exists $\lambda'\in \Lambda$ such that $\Cu\in Z(A_{\lambda'})$. But since $Z(A_{\lambda'})\subseteq V$, this contradicts that $A_{\lambda'}\in \Cu$. Therefore, $Y$ is closed with respect to the  Jacobson topology, as desired.  So, every closed subset of $\widehat{\B}$ is also closed with the Jacobson topology.

Now, let $Y$ be a closed subset of $\widehat{\B}$ with respect the Jacobson topology, and let $\Cu$ be an ultrafilter that does not belong to $Y$. Therefore, we have that $\text{Ker }\eta_\Cu\nsupseteq I_Y$. This is equivalent to say that there exists $B_\Cu\in \Cu$ such that $Z(B_\Cu)\cap Y = \emptyset$. Thus, for every $\Cu\in \widehat{\B}\setminus Y$ we can find $B_\Cu\in \B$ such that $Z(B_\Cu)\cap Y = \emptyset$. Then, we have that $\widehat{\B}\setminus Y=\bigcup\limits_{\Cu\in\widehat{\B}\setminus Y} Z(B_\Cu)$.  Hence, $\widehat{\B}\setminus Y$ is an open set because it is a union of open subsets. Therefore, $Y$ is a closed subset of $\widehat{\B}$.\end{proof}

%\begin{lem}\label{lemma_7}  Let $A,A_1,A_2,\ldots \in \Bs$ such that $\overline{A}=\cup_{n\in \N} \overline{A_n}$. Then $A=\cup_{n\in \N} A_n$.
%\end{lem}
%\begin{proof} 
%Let $v\in A$ and consider the ultrafilter $\Cu_v$ that belongs to $\overline{A}$, by the definition of $\Cu_v$. Then there exists %$n\in\N$ such that $\Cu_v\in \overline{A_n}$, so $v\in A_n$. Conversely, given any $n\in \N$ and $v\in A_n$ we have that $\Cu_v\in %\overline{A_n}$. Then $\Cu_v\in \overline{A}$, so $A\in \Cu_v$ and hence $v\in A$.
%\end{proof}

\begin{corol}\label{corollary_1}   Let $\B$ be a Boolean algebra and let $\widehat{\B}$ be the Stone's spectrum of $\B$. Then, given any $A\in \B$, we have that $\widehat{\Id_A}$ is a compact subspace of $\widehat{\B}$.
\end{corol}

\section{Actions on Boolean spaces and crossed products}\label{section_2}

By the previous results, it is possible to define a partial action on the Boolean  $C^*$-algebra by describing a partial action on the Boolean algebra. This gives a more intuitive way to understand the actions at the level of the $C^*$-algebra, and to extract information of this action by understanding the dynamics of the elements of the Boolean algebra. In this section, we will introduce dynamical systems on a Boolean algebra, and define what is a Cuntz-Krieger representation of this dynamical system on a $C^*$-algebra. Essentially, this is a generalization of a Cuntz-Krieger representation of directed graphs, considering the set of vertices the Boolean algebra, and the set of edges the partially defined actions on the vertices. 

\begin{defi} Let $\B$ be a Boolean algebra, we say that a map $\theta:\B\longrightarrow \B$ is an \emph{action on $\B$} if $\theta$ is a Boolean algebras homomorphism with $\theta(\emptyset)=\emptyset$. 
We say that the action has \emph{compact range} if $\{\theta(A)\}_{A\in \B}$ has least upper-bound, that we will denote $\mathcal{R}_\theta$. Moreover, we say that the action has \emph{closed domain}  if there exists $\mathcal{D}_\theta\in \B$ such that $\theta(\mathcal{D}_\theta)=\mathcal{R}_\theta$.
\end{defi}

\begin{rema} Observe that given an action $\theta$ with compact range and closed domain, there is not necessarily a unique $\D_\theta$ with $\theta(\D_\alpha)=\mathcal{R}_\theta$, but we will assume that in the definition there is a fixed one.
\end{rema}

Given a set $\Labe$, and given any $n\in\N$, we define $\Labe^n=\{(\alpha_1,\ldots,\alpha_n):\alpha_i\in \Labe)\}$, and $\Labe^*=\bigcup\limits_{n=0}^\infty\Labe^n$, where $\Labe^0=\{\emptyset\}$.  Given $\alpha\in \Labe^n$ for $n\geq 1$, we will write it as $\alpha=\alpha_1\cdots\alpha_n$ where $\alpha_i\in \Labe$. Given $1\leq l\leq k\leq n$, we define $\alpha_{[l,k]}:=\alpha_l\cdots\alpha_k$. We can also endow an order on $\Labe^*$ as follows: given $\alpha\in \Labe^n$ and $\beta\in \Labe^m$,  
$$\alpha\leq \beta \qquad\text{if and only if }\qquad n\leq m\text{ and }\alpha=\beta_{[1,n]}\,.$$
In case that $\alpha\leq \beta$, we define $\beta\setminus \alpha:=\beta_{[n+1,m]}$ if $n<m$ and $\emptyset$ otherwise.

\begin{defi} 
A \emph{Boolean dynamical system} on a Boolean algebra $\B$ is a triple $(\B,\Labe,\theta)$ such that $\Labe$ is a set, and $\{\theta_\alpha\}_{\alpha\in \Labe}$ is a set of actions on $\B$. Moreover, given $\alpha=(\alpha_1,\ldots,\alpha_n)\in \Labe^{\geq 1}$ the action $\theta_\alpha:\B\longrightarrow \B$ defined as $\theta_\alpha=\theta_{\alpha_n}\circ\cdots\circ\theta_{\alpha_1}$ has compact range and closed domain. 
%such that given any $n\in \N$ and $\alpha=(\alpha_1,\ldots,\alpha_n)\in \Labe^n$ we have that $A_\alpha:=\bigcup_{A\in \B}\theta_{\alpha_n}\circ \cdots\circ \theta_{\alpha_1}(A)\in \B$.
\end{defi}

\begin{nota}
Given any $\alpha\in \Labe^*$, we will write $\mathcal{D}_\alpha:=\mathcal{D}_{\theta_\alpha}$ and $\mathcal{R}_\alpha:=\mathcal{R}_{\theta_\alpha}$. Also, when $\alpha=\emptyset$, we will define $\theta_{\emptyset}=\text{Id}$, and we will formally assume that $\mathcal{R}_{\emptyset}=\mathcal{D}_{\emptyset}:=\bigcup\limits_{A\in \mathcal{B}}A$, in order to guarantee that $A\subseteq \mathcal{R}_{\emptyset}$ for every $A\in \mathcal{B}$.
\end{nota}

\begin{defi}  Let $(\B,\Labe,\theta)$ be a  Boolean dynamical system. Given $B\in \B$ we define 
$$\Delta_{B}:=\{\alpha\in \Labe: \theta_\alpha(B)\neq\emptyset\}\qquad\text{and}\qquad \lambda_{B}:=|\Delta_{B}|\,.$$  We say that $A\in \B$ is a \emph{regular set} if given any $\emptyset\neq B\in \B$ with $B\subseteq A$ we have that $0<\lambda_B<\infty $, otherwise is called a \emph{singular set}. We denote by $\Breg$ the set of all regular sets where we will include $\emptyset$.
\end{defi}

\begin{defi}\label{loc_finite}
A Boolean dynamical system $(\B,\Labe,\theta)$  is \emph{locally finite} if given $\Cu\in\widehat{\B}$ there exists $A\in \Cu$ such that for every $B\in\xi$ the set 
$$\{\alpha\in\Labe: \theta_{\alpha}(A\cap B)\neq \emptyset\}$$
is finite.
%does not exist a sequence infinite $\{\alpha_j\}_{j=1}^\infty\subseteq\Labe$ such that $\theta_{\alpha_j}(A)\neq \emptyset$ for every $A\in\Cu$. 
\end{defi}

Observe that if $|\Labe|<\infty$ then $(\B,\Labe,\theta)$ is locally finite.
 
\begin{defi}\label{DefAlgebra}
A \emph{Cuntz-Krieger representation of the Boolean dynamical system} $(\B,\Labe, \theta)$ in a $C^*$-algebra $\mathcal{A}$ consists of a family of projections $\{P_A : A\in\B\}$ and partial isometries $\{S_\alpha : \alpha\in \Labe\}$ in $\mathcal{A}$, with the following properties:
\begin{enumerate}
\item If $A,B\in\B$, then $P_A\cdot P_B = P_{A\cap B}$ and $P_{A\cup B} = P_A + P_B-P_{A\cap B}$, where $P_{\emptyset}=0$.
\item If $\alpha\in \Labe$ and $A\in\B$, then $P_A\cdot  S_\alpha = S_\alpha  \cdot  P_{\theta_\alpha(A)}$. 
\item If $\alpha, \beta\in \Labe$ then $S_\alpha^* \cdot S_\beta = \delta_{\alpha,\beta} \cdot  P_{\Ra}$.
\item Given $A\in \Breg$ we have that
$$P_A =\sum_{\alpha\in\Delta_A} S_\alpha\cdot  P_{\theta_\alpha(A)}\cdot  S_\alpha^*\,.$$
\end{enumerate}
A representation is called \emph{faithful} if $P_A\neq 0$ for every $A\in\B$.
\end{defi}

Given a representation $\{P_A,S_\alpha\}$ of a Boolean dynamical system $(\B,\Labe, \theta)$ in a $C^*$-algebra $\mathcal{A}$, we define $C^*(P_A,S_\alpha)$ to be the sub-$C^*$-algebra of $\mathcal{A}$ generated by $\{P_A,S_\alpha:A\in \B,\,\alpha\in \Labe\}$.

A \emph{universal representation} $\{p_A,s_\alpha\}$ of a Boolean dynamical system $(\B,\Labe, \theta)$ is a representation satisfying the following universal property: given a representation $\{P_A,S_\alpha\}$ of $(\B,\Labe,\theta)$ in a $C^*$-algebra $\mathcal{A}$, there exists a non-degenerate $*$-homomorphism $\pi_{S,P}:C^*(p_A,s_\alpha)\longrightarrow \mathcal{A}$ such that $\pi_{S,P}(p_A)=P_A$ and $\pi_{S,P}(s_\alpha)=S_\alpha$ for $A\in \B$ and $\alpha\in \Labe$.  We will set $C^*(\B,\Labe,\theta):=C^*(p_A,s_\alpha)$. The existence of the universal representation can be found in \cite{BPII}, but we will show it in a different way  in Section \ref{section_4}: given a Boolean dynamical system $(\B,\Labe, \theta)$, we will construct a topological graph  $E$ \cite{KatsuraI}, and we will prove that there exists a one to one correspondence between Cuntz-Krieger representations of $(\B,\Labe, \theta)$ and Cuntz-Krieger representations of $E$. Hence, the universal $C^*$-algebra $C^*(\B,\Labe,\theta)$ is isomorphic to the universal  $C^*$-algebra $\mathcal{O}(E)$ associated to the topological graph $E$.

\begin{theor}[Existence of a Universal representation]\label{theorem_1} Given a Boolean dynamical system $(\B,\Labe,\theta)$ there exists a unique universal representation of $(\B,\Labe,\theta)$. If $C^*(\B,\Labe,\theta)$ is the associated $C^*$-algebra, we will call $C^*(\B,\Labe,\theta)$ the Cuntz-Krieger Boolean algebra of the Boolean dynamical system $(\B,\Labe,\theta)$.
\end{theor}

%\begin{rema}{\textbf{(The unital case)}} $C^*(\B)$ is unital is equivalent to say that there exists $X\in \B$ such that $A\subseteq X$ for every $A\in \B$. Then given any representation $\{P_A,S_\alpha\}$ of $(\B,\Labe, \theta)$ we have that $P_X$ is the unit of $C^*(P_A,S_\alpha)$. Conversely, suppose that $e$ is a unit of $C^*(P_A,S_\alpha)$. Then observe that $\{P_A\}_{A\in \B}$ is an approximate unit of projections of $C^*(P_A,S_\alpha)$, so $P_A\longrightarrow e$ in the strict topology, and hence for every $\varepsilon>0$ there exists $X\in \B$ such that $\|e-eP_X\|=\|e-P_X\|<\varepsilon$. Therefore since  $e-P_X$ is a projection we have that $\|e-P_X\|$ is either $0$ or $1$, so if  $\varepsilon<1$ we have that $e=P_X$. Thus, given $A\in \B$ we have that $P_A=P_XP_A$, so $P_X$ is a unit of $C^*(P_A,S_\alpha)$.\end{rema}

By the universality of $C^*(\B,\Labe,\theta)$, there exists a strongly continuous action  $\beta:\mathbb{T}\curvearrowright \text{Aut }(C^*(\B,\Labe,\theta))$ such that $\beta_z(p_A)=p_A$ and $\beta_z(s_\alpha)=zs_\alpha$ for every $A\in \B$, $\alpha\in \Labe$ and $z\in \mathbb{T}$. The action $\beta$ is called the \emph{gauge action}

Therefore, we can use the representation of $C^*(\B,\Labe,\theta)$ as a topological graph $C^*$-algebra to obtain a gauge uniqueness theorem \cite[Theorem 4.5]{KatsuraI}.

\begin{theor}[Gauge Uniqueness Theorem] Let $(\B,\Labe,\theta)$ be a Boolean dynamical system and let $\{P_A,S_\alpha\}$ be a representation of $(\B,\Labe,\theta)$ in $\mathcal{A}$. Suppose that $P_A\neq 0$ whenever $A\neq\emptyset$, and that there is a strongly continuous action $\gamma$ of $\mathbb{T}$ on $C^*(P_A,S_\alpha)\subseteq \mathcal{A}$, such that for all $z\in \mathbb{T}$ we have that $\gamma_z\circ \pi_{S,P}=\pi_{S,P}\circ \beta_z$. Then, $\pi_{S,T}$ is injective.
\end{theor}

\section{$0$-dimensional topological graphs}

Our goal in this section is to use a topological graph $E=(E^0,E^1,d,r)$ with $E^0$ and $E^1$ being second countable, locally compact $0$-dimensional spaces (i.e., Hausdorff, totally disconnected and having a basis consisting of clopen sets) to construct a Boolean dynamical system.

First, we should recall the definition of topological graph given in \cite{KatsuraI}.

\begin{defi}
Let $E^0$ and $E^1$ be locally compact spaces, let $d:E^1\rightarrow E^0$ be a local homeomorphism, and let $r:E^1\rightarrow E^0$ be a continuous map. Then, the quadruple $E=(E^0,E^1,d,r)$ is called a \emph{topological graph}. We will call $E$ a $0$-dimensional graph if $E^0$ and $E^1$ have $0$ covering dimension.
\end{defi}
Let us denote $C_d(E^1)$ the set of continuous functions on $E^1$ such that 
$$\langle \xi|\xi\rangle(v):=\sum_{e\in d^{-1}(v)}|\xi(e)|^2<\infty$$ 
for any $v\in E^0$ and $\langle\xi|\xi\rangle\in C_0(E^0)$. For $\xi,\zeta\in C_d(E^1)$ and $f\in C_0(E^0)$, we define $\xi f\in C_d(E^1)$ and $\langle\xi|\zeta\rangle\in C_0(E^1)$ by
$$(\xi f)(e)=\xi(e)f(d(e))\qquad \text{for }e\in E^1$$
$$\langle\xi|\zeta\rangle(v)=\sum_{e\in d^{-1}(v)}\overline{\xi(e)}\zeta(e) \qquad\text{for }v\in \overline{E}^0\,.$$
With these operations, $C_d(E^1)$ is a right Hilbert $C_0(E^0)$-module. We define a left action $\pi_r$ of $C_0(E^0)$ on $C_d(E^1)$ by $(\pi_r(f)\xi)(e)=f(r(e))\xi(e)$ for $e\in E^1$, $\xi\in C_d(E^1)$ and $f\in C_0(E^0)$. In this way, we define a $C^*$-correspondence $C_d(E^1)$ over $C_0(E^0)$.

\begin{defi} A \emph{Toeplitz $E$-pair} on a $C^*$-algebra $\mathcal{A}$ is a pair of maps $T=(T^0,T^1)$, where $T^0:C_0(E^0)\longrightarrow \mathcal{A}$ is a $*$-homomorphism and $T^1:C_d(E^1)\longrightarrow\mathcal{A}$ is a linear map, satisfying:
\begin{enumerate}
\item $T^1(\xi)^*T^1(\zeta)=T^0(\langle \xi|\zeta\rangle)$ for $\xi,\zeta\in C_d(E^1)$,
\item $T^0(f)T^1(\xi)=T^1(\pi_r(f)\xi)$ for $f\in C_0(E^0)$ and $\xi\in C_d(E^1)$.
\end{enumerate}
We will denote by  $C^*(T^0,T^1)$  the sub-$C^*$-algebra of $\mathcal{A}$ generated by the Toeplitz $E$-pair $(T0, T^1)$.
\end{defi}

Given a topological graph $E$, we define the following $3$ open subsets of $E^0$:
$$E_{sce}:=E^0\setminus \overline{r(E^0)}\,,$$
$$E^0_{fin}:=\{v\in E^0:\exists V \text{neighborhood of }v \text{ such that }r^{-1}(V)\text{ is compact}\} \text{, and }$$
$$E^0_{rg}:=E^0_{fin}\setminus \overline{E^0_{sce}}\,.$$
We have that $\pi_r^{-1}(\mathcal{K}(C_d(E^1)))=C_0(E^0_{fin})$ and $\text{Ker }\pi_r=C_0(E^0_{sce})$. For a Toeplitz $E$-pair $T=(T^0,T^1)$, we define a $*$-homomorphism $\Phi:\mathcal{K}(C_d(E^1))\longrightarrow \mathcal{A}$ by $\Phi(\theta_{\xi,\zeta})=T^1(\xi)T^1(\zeta)^*$ for $\xi,\zeta\in C_d(E^1)$.

\begin{defi} A Toeplitz $E$-pair $T=(T^0,T^1)$ is called a \emph{Cuntz-Krieger $E$-pair} if $T^0(f)=\Phi(\pi_r(f))$ for any $f\in C_0(E^0_{rg})$. We denote by $\mathcal{O}(E)$ the $C^*$-algebra  is generated by the universal Cuntz-Krieger $E$-pair $t=(t^0,t^1)$.
\end{defi}

Therefore, $\mathcal{O}(E)$ is generated by $\{t^0(f):f\in C_0(E^0)\}$ and $\{t^1(\xi):\xi\in C_d(E^1)\}$, where $(t^0,t^1)$ is a universal Cuntz-Krieger pair of $E$.

\begin{defi}\label{compactsup}
Let $E$ be a topological graph, then a family $\{V_\alpha\}_{\alpha\in \Labe}$ of subsets of $E^1$ \emph{compactly supports} $E$ if it satisfies the  conditions:
\begin{enumerate}
\item $V_\alpha$ is a compact clopen set of $E^1$ for every $\alpha\in\Labe$,
\item $E=\bigcup\limits_{\alpha\in \Labe} V_\alpha$,
\item $V_\alpha\cap V_\beta=\emptyset$ when $\alpha\neq \beta$,
\item the restriction $d_{| V_\alpha}$ is a homeomorphism for every $\alpha\in \Labe$,
\item there exists a compact clopen $\Do$ with $r(V_\alpha)\subseteq \Do$ for every $\alpha\in\Labe$.
\end{enumerate}
\end{defi}

\begin{rema} If $E$ is a topological graph with $E^0$ and $E^1$ being second countable and locally compact $0$-dimensional spaces, then it always exists  $\{V_\alpha\}_{\alpha\in \Labe}$ that compactly supports $E$.
\end{rema}

Then, we can trivially define a Boolean dynamical system.

\begin{lem}\label{boolean_dynamical}
Let $E$ be a $0$-dimensional topological graph that has a family of subsets $\{V_\alpha\}_{\alpha\in \Labe}$ of $E^1$ that compactly supports $E$. Then if $\B$ is the Boolean algebra of the compact and clopen subsets of $E^0$, and given $\alpha\in\Labe$ we define the  $\theta_\alpha(A):=d(r^{-1}(A)\cap V_\alpha)$ for every $A\in \B$, then $(\B,\Labe,\theta)$ is a Boolean dynamical system.
\end{lem}
\begin{proof} It is straightforward to check that $\theta_\alpha$ is an action on $\B$ with compact range $\Ra:=d(V_\alpha)$ and compact domain $\Do$. 
\end{proof}
 
\begin{rema} Observe that if $E$ is $0$-dimensional topological graph, then we can construct a Boolean dynamical system. However, it is not unique, because it could exist several $\{V_\alpha\}_{\alpha\in \Labe}\subseteq E^1$  satisfying the above conditions. We will see that, despite of the choice of the above pairs of sets, the $C^*$-algebras of the associated Boolean dynamical systems are isomorphic. 
\end{rema}

\begin{lem}\label{vertices}
Let $E$ be a $0$-dimensional topological graph and let $\{V_\alpha\}_{\alpha\in \Labe}$ be a family of subsets of $E^1$ satisfying conditions of the Definition \ref{compactsup}. Then if  $(\B,\Labe,\theta)$ is the associated Boolean dynamical system defined in Lemma \ref{boolean_dynamical} then  given $A\in \B$ we have that
\begin{enumerate}
\item $A\subseteq  E^0_{sce}$ if and only if $\lambda_A=0$.
\item $A\subseteq E^0_{fin}$ if and only if  $\lambda_A<\infty$.
\item  $A\subseteq  E^0_{rg}$ if and only if $A\in \Breg$.
\end{enumerate}
\end{lem}
\begin{proof}
(1) We have that $A\subseteq E^0_{sce}$, this means that 
$$\emptyset= A \cap r(E^1)=A\cap r(\bigcup\limits_{\alpha\in \Labe}V_\alpha)=\bigcup\limits_{\alpha\in \Labe} A\cap r(V_\alpha)\,,$$
so $A\cap r(V_\alpha)=\emptyset$ for every $\alpha\in\Labe$, but it is equivalent to $r^{-1}(A)\cap V_\alpha=\emptyset$ for every $\alpha\in\Labe$. Then by definition $\theta_\alpha(A)=\emptyset$ for every $\alpha\in \Labe$, whence $\lambda_A=0$.

(2) Let $A\subseteq E^0_{fin}$, by definition $r^{-1}(A)$ must be compact. Then since $\bigcup\limits_{\alpha\in \Labe}(r^{-1}(A)\cap V_\alpha)$ is an disjoint  open covering of $r^{-1}(A)$, only a finite number of $r^{-1}(A)\cap V_\alpha$ can be non-empty. But this is equivalent to say that only a finite number of $\theta_\alpha(A)=d(r^{-1}(A)\cap V_\alpha)$ is non-empty, whence $\lambda_A<\infty$.

(3) This is clear using (1) and (3)
\end{proof}

\begin{prop}\label{proposition_5} Let $E$ be a $0$-dimensional topological graph and let $\{V_\alpha\}_{\alpha\in \Labe}$ be a family of subsets of $E^1$ satisfying conditions of the Definition \ref{compactsup}. Then if  $(\B,\Labe,\theta)$ is the associated Boolean dynamical system defined in Lemma \ref{boolean_dynamical}, given any Cuntz-Krieger $E$-representation $(T^0,T^1)$ on $\mathcal{A}$, the  family of elements of $\mathcal{A}$ defined by 
$$P_A   :=T^0(\chi_{A})\text{ and } S_\alpha:=T^1(\chi_{V_\alpha})\,.$$
for every $A\in \B$ and $\alpha\in \Labe$, is a representation of $(\B,\Labe,\theta)$ on $\mathcal{A}$, i.e.,

\begin{enumerate}
\item If $A,B\in \B$ then $P_AP_B=P_{A\cap B}$ and $P_{A\cup B}=P_A+P_B-P_{A\cap B}$, where $P_\emptyset=0$.
\item If $\alpha\in \Labe$ and $A\in \B$ then $P_AS_\alpha=S_\alpha P_{\theta_\alpha(A)}$.
\item If $\alpha,\beta\in \Labe$ then $S^*_\alpha S_\alpha=P_{\Ra}$, and $S^*_\alpha S_\beta=0$ unless $\alpha=\beta$.
\item For $A\in \Breg$, we have 
$$P_A=\sum_{\alpha\in\Delta_A}S_\alpha P_{\theta_\alpha(A)} S^*_\alpha\,.$$   
\end{enumerate} 
\end{prop}
\begin{proof} For $(1)$, observe that $\{P_A\}_{A\in\B}$ is a family of commuting projections. Then, $P_{A\cap B}=P_AP_B$ and $P_{A\cup B}=P_A+P_B-P_{A\cap B}$ for every $A,B\in \B$ follows from the fact that $T^0$ is a homomorphism. For $(2)$, given $A\in\B$ and $\alpha\in \Labe$, we have that 
\begin{align*}P_AS_\alpha & =T^0(\chi_{A})T^1(\chi_{V_\alpha})=T^1(\pi_r(\chi_{A})\chi_{V_\alpha}) =T^1((\chi_{A}\circ r)\chi_{V_\alpha}) \\ & =T^1(\chi_{r^{-1}(A)}\chi_{V_\alpha})=T^1(\chi_{r^{-1}(A)\cap V_\alpha}) \\ & =T^1(\chi_{V_\alpha})T^0(\chi_{d(r^{-1}(A)\cap V_\alpha})=T^1(\chi_{V_\alpha})T^0(\chi_{\theta_\alpha(A)})=S_\alpha P_{\theta_\alpha(A)}\,.\end{align*}
For $(3)$, we look at the equality 
$$S_\alpha^*S_\beta=T^1(\chi_{V_\alpha})^*T^1(\chi_{V_\beta})=T^0(\langle\chi_{V_\alpha}|\chi_{V_\beta}\rangle)\,.$$
By the definition,
$$\langle \chi_{V_\alpha}|\chi_{V_\beta}\rangle (v) = \sum_{e\in d^{-1}(v)}\overline{\chi_{V_\alpha}(e)}\chi_{V_\beta}(e)\,,$$
for any $v\in E^0$. Since $V_\alpha\cap V_\beta=\emptyset$  whenever $\alpha\neq \beta$, we get that this expression will sum $0$ if $\alpha\neq \beta$. Now, since $d_{|V_\alpha}$ is a homeomorphism it follows that 
$$\sum_{e\in d^{-1}(v)}|\chi_{V_\alpha}(e)|^2=|\{e\in V_\alpha:d(e)=v\}|=\chi_{d(V_\alpha)}(v)=\chi_{\Ra}(v)\,.$$
For $(4)$, we will use the Cuntz-Krieger relation 
$$T^0(f)=\Phi(\pi_r(f)),$$
which holds whenever $f\in C_0(E^0_{rg})$. Since $A\in \Breg$, by the Lemma  \ref{vertices} we have that $A\subseteq E^0_{rg}$. So, it is  enough to show that 
$$\pi_r(\chi_{A})=\sum_{\alpha\in \Delta_A}\theta_{\chi_{V_\alpha},\chi_{V_\alpha}\cdot \chi_{\theta_{\alpha}(A)}}\,.$$ 
Evaluating at $\xi\in C_d(E^1)$ and $e\in E^1$, we have that
$$\sum_{\alpha\in \Delta_A}\theta_{\chi_{V_\alpha},\chi_{V_\alpha}\cdot \chi_{\theta_{\alpha}(A)}}(\xi)(e)=$$
$$\sum_{\alpha\in \Delta_A}\chi_{V_\alpha}(e)\langle\chi_{V_\alpha}\cdot \chi_{\theta_{\alpha}(A)}|\xi\rangle (d(e))=$$
$$\sum_{\alpha\in \Delta_A}\chi_{V_\alpha}(e)\left(\sum_{d(e')=d(e)} \chi_{V_\alpha}(e')\chi_{\theta_{\alpha}(A)}(d(e'))\xi(e')\right)\,.$$
Whenever $e,e'\in V_\alpha$ for some $\alpha\in \Labe$, since $d(e)=d(e')$ if and only if $e=e'$, this reduces to 
$$\sum_{\alpha\in \Delta_A}\chi_{\mathcal{R}_\alpha}(e)\chi_{\theta_\alpha(A)}(d(e))\xi(e)=\left\lbrace  \begin{array}{ll} \chi_{\theta_{\alpha}(A)}(d(e))\xi(e) & \text{whenever }e\in V_\alpha\text{ for }\alpha\in \Delta_A \\ 0 & \text{otherwise}\end{array}\right. \,.$$
In addition, $\theta_\alpha(A)=\emptyset$ when $\alpha\notin \Delta_A$. Thus, we can omit the case clause.  What remains is $\chi_{\theta_\alpha(A)}(d(e))\xi(e)$ when $e\in V_\alpha$ for any $\alpha\in \Labe$. On the other hand,
$$(\pi_r(\chi_{A})	\xi)(e)=\chi_{A}(r(e))\xi(e)\,.$$
Now, when $e\in V_\alpha$ for some $\alpha\in \Labe$, we get that $\chi_{A}(r(e))=\chi_{d(r^{-1}(A)\cap V_\alpha)}(d(e))=\chi_{\theta_\alpha(A)}(d(e))$, so we are done.   
\end{proof}

\section{A faithful representation of  $(\B,\Labe,\theta)$.}\label{section_4}

Now, given a Boolean dynamical system $(\B,\Labe,\theta)$, we will construct a faithful representation of $(\B,\Labe,\theta)$ in $\mathcal{O}(E)$, where $E$ is a   $0$-dimensional topological graph.

Let $(\B,\Labe,\theta)$ be a Boolean dynamical system.  We define $E^0$ to be  the Stone's spectrum $\widehat{\B}$ of $\B$, and $E^1$ to be the disjoint union 
$$E^1=\bigsqcup\limits_{\alpha\in \Labe} \widehat{\Id_{\Ra}}\,,$$
of Stone's spectrums of the principal ideals of $\B$ generated by the range  $\mathcal{R}_\alpha$ of the actions $\theta_\alpha$.
Since $\widehat{\B}$ and each $\widehat{\Id_{\Ra}}$ have a basis of clopen sets, they are $0$-dimensional spaces, and since they are totally disconnected spaces they are locally compact Hausdorff spaces too. These properties are transfered to arbitrary unions of such spaces, so $E^0$ and $E^1$ are also locally compact Hausdorff $0$-dimensional spaces. Also observe that, given any $\alpha\in \Labe$, then $\widehat{\Id_{\Ra}}$ is a clopen and compact subset of $\widehat{\B}$. 

\begin{nota} To distinguish the edge and the vertex space of the topological graph $E$, we will denote 
$$E^0=\{v_\Cu:\Cu\in \widehat{\B}\}\qquad \text{ and }\qquad E^1=\bigsqcup\limits_{\alpha\in\Labe}E^1_\alpha\,,$$
where $E^1_\alpha=\{e^{\alpha}_{\Cu}:\Cu\in \widehat{\Id_{\Ra}}\}$. 
Given $\alpha\in \Labe$ and $A,B\in \B$ with  $B\subseteq \mathcal{R}_\alpha$, we define the clopen and compact subsets 
$$\Ne_A:=\{v_\Cu:A\in \Cu\}\subseteq E^0\qquad\text{and}\qquad \Me^\alpha_B:=\{e^\alpha_\Cu:B\in \Cu\}\subseteq E_\alpha^1\,.$$
\end{nota}

\begin{prop}\label{prop_top} Let $(\B,\Labe,\theta)$ be a Boolean dynamical system, and let $E^0=\widehat{\B}$ and $E^{1}=E^1=\bigsqcup\limits_{\alpha\in \Labe} \widehat{\Id_{\Ra}}$. If we define the maps $d,r:E^1\rightarrow E^0$ as 
$$d(e^\alpha_\Cu)=v_\Cu\qquad \text{and}\qquad r(e^\alpha_\Cu)=v_{\widehat{\theta_\alpha}(\Cu)}\,,$$
for every $\alpha\in\Labe$ and $\Cu\in \widehat{\Id_{\Ra}}$, then $(E^0,E^1,d,r)$ is a topological graph.
\end{prop}
\begin{proof} First, by the above arguments, we have that $E^0$ and $E^1$ are locally compact Hausdorff spaces. Let   $d:E^1\longrightarrow E^0$ be the map defined by $d(e^\alpha_\Cu)=v_\Cu$ for some $e^\alpha_\Cu\in E^1_\alpha$. Every point of $E^1$ belongs to a component $E^1_\alpha$ for some $\alpha\in \Labe$, and clearly we have that $d_{|E^1_\alpha}$ is an homeomorphism. Thus, $d$ is a local homeomorphism. 

Let $\widehat{\theta\alpha}:\widehat{\Id_{\Ra}}\longrightarrow\widehat{\Id_{\Do}}$ be the induced map, that is continuous by Lemma \ref{induced_map}. Thus, $(E^0,E^1,d,r)$ is a $0$-dimensional topological graph. 
\end{proof}

\begin{corol}\label{faithful_rep} Let $(\B,\Labe,\theta)$ be a Boolean dynamical system, let $E$ be the associated topological graph defined in Proposition \ref{prop_top}, and let $(t^0,t^1)$ the universal Cuntz-Krieger $E$-pair. Then,
$$p_A:=t^0(\chi_{\Ne_A})\qquad\text{and}\qquad s_\alpha:=t^1(\chi_{E^1_\alpha})$$
 for $A\in \B$ and $\alpha\in \Labe$, defines a  faithful representation of $(\B,\Labe,\theta)$ in $\mathcal{O}(E)$.
\end{corol}
\begin{proof} Let $E=(E^0,E^1,d,r)$ be the topological graph defined in Proposition \ref{prop_top}.  Observe that  $\{E_\alpha^1\}_{\alpha\in \Labe}$ compactly supports $E$. It is straightforward to check that the  Boolean dynamical system associated to $E$ defined in Lemma \ref{boolean_dynamical}  is $(\B,\Labe,\theta)$ again. Now, using Proposition \ref{proposition_5} with the universal faithful representation $(t^0,t^1)$ of $\mathcal{O}(E)$, we conclude the proof. 
\end{proof}

Our next step is to prove that the faithful representation constructed in Corollary \ref{faithful_rep} is the universal one. To do that, we first have to look closer at the topological graph $E$ associated to a Boolean dynamical system.

The following lemma will be useful in the sequel.

\begin{lem}\label{lemma_81} Let $(\B,\Labe,\theta)$ be a Boolean dynamical system, and let $\alpha\in \Labe$ and $\Cu\in \widehat{\Id_{\Do}}$. Then, given any $\Cu'\in \widehat{\Id_{\Ra}}$ such that $\theta_{\alpha}(A)\in \Cu'$ for every $A\in \Cu$, we have that $\Cu=\{B\in \Id_{\Do}:\theta_\alpha(B)\in \Cu'\}$.
\end{lem}
\begin{proof} The first inclusion is clear because $\Cu'$ contains  $\theta_\alpha(A)$ for every $A\in \Cu$. Now, let $B\in \B$ such that $\theta_\alpha(B)\in \Cu'$. Then, given any $A\in \Cu$ we have that $\theta_\alpha(A)\in \Cu'$. So, we have that 
$$\emptyset\neq \theta_\alpha(A)\cap \theta_\alpha(B)=\theta_\alpha(A\cap B)\in \Cu'\,.$$
Thus, $A\cap B\neq \emptyset$. Then, $A=(A\cap B)\cup (A\setminus (A\cap B))$, but by condition $\textbf{F3}$ it follows that either $A\cap B$ or $A\setminus (A\cap B)$ belongs to $\Cu$. Observe that   $A\setminus (A\cap B)$ cannot belong to $\Cu$, as otherwise $$\theta_\alpha(A\cap B)\cap \theta_\alpha(A\setminus (A\cap B))=\emptyset,$$
contradicting condition $\textbf{F2}$ of the ultrafilter $\Cu'$. Therefore, $A\cap B\in \Cu$, whence so does $B$ by condition $\textbf{F1}$. 
\end{proof}

%\begin{lem}\label{lemma_83} Let $(\B,\Labe,\theta)$ be a Boolean dynamical system. If $\alpha\in \Labe$, then $\varphi_\alpha:\D(\mathcal{I}_{\mathcal{R}_\alpha})\longrightarrow \D(\mathcal{I}_{\mathcal{D}_\alpha})$ is injective if and only if given any $B\in \mathcal{I}_{\mathcal{R}_\alpha} $ there  exists $A\in \mathcal{I}_{\mathcal{D}_\alpha}$ such that $\theta_\alpha(A)=B$.
%\end{lem}
%\begin{proof} Observe that given $\alpha\in \Labe$, the action $\theta_\alpha$ induces a $*$-homomorphism $\widehat{\theta}_\alpha:C^*(\mathcal{I}_{\mathcal{D}_\alpha})\rightarrow C^*(\mathcal{I}_{\mathcal{R}_\alpha})$ defined by $\chi_A\mapsto \chi_{\theta_{\alpha}(A)}$ for every $A\in \B$. Then, using the Stone's Representation Theorem, we have that $\widehat{\theta}_\alpha:C(\mathcal{I}_{\mathcal{D}_\alpha})\rightarrow C(\mathcal{I}_{\mathcal{R}_\alpha})$ is defined by $f\mapsto f\circ \varphi_\alpha$ for every $C(\mathcal{I}_{\mathcal{D}_\alpha})$. Thus, $\theta_\alpha$ is surjective if and only if $\varphi_\alpha$ is injective. But if $\widehat{\theta}_\alpha$ is surjective then, given any $\chi_B\in C(\mathcal{I}_{\mathcal{R}_\alpha})$, there exists $A\in \mathcal{D}_\alpha$ such that $\widehat{\theta}_\alpha(\chi_A)=\chi_B$, and hence $B=\theta_\alpha(A)$, as desired.
%\end{proof}

\begin{lem}\label{lemma_82} Let $(\B,\Labe,\theta)$ be a Boolean dynamical system, and let $E$ be the topological graph defined in Proposition \ref{prop_top}. Then, given $e\in E^1_\alpha$, the following statements are equivalent:
\begin{enumerate}
\item  $r(e)\in \Ne_A$. 
\item $d(e)\in \Ne_{\theta_\alpha(A)}$. 
\item $e\in \Me^\alpha_{\theta_\alpha(A)}$.
\end{enumerate}
\end{lem}
\begin{proof} $(2)\Leftrightarrow (3)$ is clear by definition. Now, let $e=e^\alpha_\Cu$ for some $\alpha\in \Labe$ and $\Cu\in \widehat{\Id_{\Ra}}$. Suppose that $v_{\Cu'}=r(e^\alpha_\Cu)\in \Ne_A$, where $\Cu'=\{B\in \Id_{\Do}:\theta_\alpha(B)\in \Cu\}$, whence $v_{\Cu}\in \Ne_{\theta_{\alpha}(B)}$ for every $B\in \Cu'$. Since $A\in \Cu'$, it follows that $v_\Cu\in \Ne_{\theta_\alpha(A)}$, as desired. Now, let us suppose that $d(e^\alpha_\Cu)=v_\Cu\in \Ne_{\theta_{\alpha}(A)}$, so that $\theta_{\alpha}(A)\in \Cu$. Since $r(e^\alpha_\Cu)=v_{\Cu'}$, where $\Cu'=\{B\in\Id_{\Do}:\theta_{\alpha}(B)\in \Cu\}$, it follows that $A\in \Cu'$. Thus, $v_{\Cu'}\in \Ne_A$, as desired. 
\end{proof}

\begin{exem} \label{example_3} Let  $X=\N\cup \{w\}$, and let  $\B$ be the minimal Boolean space generated by the subsets $\{F\subseteq \N: F \text{ finite }\}\cup \{\N\setminus F : F \text{ finite }\}\cup\{w\}$. We have that $\widehat{\B}$ is the compact space  $\{v_{C_i}:i=1,2,\ldots,\infty\}\cup\{v_{\Cu_w}\}$, where $\Cu_w=\{A\in \B:w\in A\}$. Let $\Labe=\{\alpha\}$, and define
$$\theta_\alpha(A)=\left\lbrace \begin{array}{ll} \N & \text{if }A=\{w\} \\ \emptyset & \text{otherwise}\end{array}\right.,$$
that is an action on the Boolean space $\B$. Therefore, $(\B,\Labe,\theta)$ is a Boolean dynamical system, and let $E$ be its associated topological graph. Thus, $E^0=\{v_{C_i}:i=1,2,\ldots,\infty\}\cup\{v_{\Cu_w}\}$ and $E^1=\{e^\alpha_{\Cu_i}:i=1,\ldots,\infty\}$. Then, $d(e^\alpha_{\Cu_i})=v_{\Cu_i}$ and $r(e^\alpha_{\Cu_i})=v_{\Cu_w}$ for every $i=1,2,\ldots,\infty$. A picture of this topological graph will be as follows:
$$
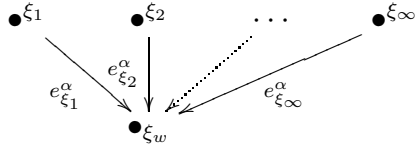
\xymatrix{\bullet^{\Cu_1}\ar[dr]_{e^\alpha_{\Cu_1}} & \bullet^{\Cu_2}\ar[d]_{e^\alpha_{\Cu_2}} & {\cdots}\ar@{.>}[dl] & \bullet^{\Cu_\infty}\ar[dll]^{e^\alpha_{\Cu_\infty}}\\ & {\bullet_{\Cu_w}}& }$$
\end{exem}

\begin{exem}\label{example_4} Let $\B$ be the minimal Boolean algebra generated by $$\{F: F\subseteq \Z\text{ finite }\}\cup \{\Z\setminus F:F\subseteq \Z \text{ finite}\}\,.$$ 
Let  $\theta_a$, $\theta_b$ and $\theta_c$ be actions on $\B$ given by the following graph
$$
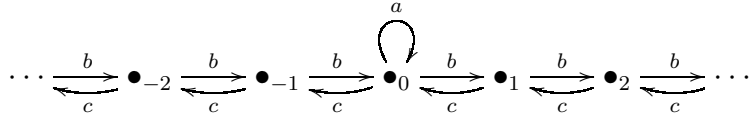
\xymatrix{{\cdots}\ar[r]^b & {\bullet}_{-2}\ar	@/^6pt/ [l]^c\ar[r]^b & {\bullet}_{-1}\ar@/^6pt/ [l]^c\ar[r]^b & {\bullet}_0\ar@/^6pt/ [l]^c\ar[r]^b\uloopr{}^a & {\bullet}_1\ar@/^6pt/ [l]^c\ar[r]^b & {\bullet}_2\ar@/^6pt/ [l]^c\ar[r]^b & {\cdots}\ar@/^6pt/ [l]^c }$$
We have that $\widehat{\B}=\{\Cu_{n}:n\in \Z\}\cup \{\Cu_\infty\}$ where $\Cu_n=\{A\in \B:n\in A\}$ and $\Cu_\infty=\{\Z\setminus F: F\subseteq \Z \text{ finite}\}$. 

Let us consider its associated topological graph $E$, where  $E^0=\{v_{\Cu_n}:n\in \Z\}\cup\{v_{\Cu_\infty}\}$ is the one point compactification of $\Z$, $E^1_a=\{e^a_{\Cu_0}\}$, $E^1_b=\{e^b_{\Cu_n}:n\in \Z\}\cup \{e^b_{\Cu_\infty}\}$ and $E^1_c=\{e^c_{\Cu_n}:n\in \Z\}\cup \{e^c_{\Cu_\infty}\}$. Hence, 
$$E^1=E^1_a\sqcup E^1_b \sqcup E^1_c$$
 is a compact space because $E^1_a$, $E^1_b$ and $E^1_c$ are compact by Corollary \ref{corollary_1}. Then, we have that $d(e^a_{\Cu_0})=v_{\Cu_0}$ and $r(e^a_{\Cu_0})=v_{\Cu_0}$. Given $n\in \Z$, we have that $d(e^b_{\Cu_n})=v_{\Cu_n}$ and $r(e^b_{\Cu_n})=v_{\Cu_{n-1}}$, and  $d(e^c_{\Cu_n})=v_{\Cu_n}$ and $r(e^c_{\Cu_n})=v_{\Cu_{n-1}}$. Finally, $d(e^b_{\Cu_\infty})=d(e^c_{\Cu_\infty})=v_{\Cu_\infty}$ and $r(e^b_{\Cu_\infty})=r(e^c_{\Cu_\infty})=v_{\Cu_\infty}$.
\end{exem}

%Given a Boolean dynamical system $(\B,\Labe,\theta)$ we have defined a compactly supported and $0$-dimensional topological graph $E$, and describe in terms of $(\B,\Labe,\theta)$.  

Now, using Lemma \ref{vertices} we can characterize the following sets: given $A\in\B$
\begin{enumerate}
\item $\Ne_A\subseteq E^0_{sce}$ if and only if $\lambda_A=0$,
\item $\Ne_A\subseteq E^0_{fin}$ if and only if $\lambda_A<\infty$,
\item $\Ne_A\subseteq E^0_{reg}$ if and only if for every  $\emptyset\neq B\subseteq A$ we have that $0<\lambda_A<\infty$,
\item $\Ne_A\subseteq E^0_{sg}$ if and only if there exists $\emptyset\neq B\subseteq A$ such that $\lambda_A\in\{0,\infty\}$.
\end{enumerate}

\begin{theor}\label{proposition_51} Let $(\B,\Labe,\theta)$ be a Boolean dynamical system, and let $E$ be the associated topological graph defined in Proposition \ref{prop_top}. Then, the faithful representation constructed in Corollary \ref{faithful_rep} is universal. Therefore, $\Clab\cong\mathcal{O}(E)$.
\end{theor}
\begin{proof} Our strategy will be to prove that any representation $\{P_A,S_\alpha\}$ of $(\B,\Labe,\theta)$ induces a representation $(T^0,T^1)$ of the associated topological graph $E$ constructed in Proposition \ref{prop_top}, such that $T^0(\chi_{\Ne_A})=P_A$ and $T^1(\chi_{E^1_\alpha})=S_\alpha$. Then the universality of $(t^0,t^1)$ will induce the map $\eta:\mathcal{O}(E)\rightarrow C^*(P_A,S_\alpha)$ with $p_A=t^0(\chi_{\Ne_A})\mapsto T^0(\chi_{\Ne_A})=P_A$ and $s_\alpha=t^1(\chi_{E^1_\alpha})\mapsto T^1(\chi_{E^1_\alpha})=S_\alpha$.

 First, we claim that the families $\{\chi_{\Ne_A}:A\in \B\}$ and $\{\chi_{\Me^\alpha_A}:\alpha\in \Labe,\,A\in \Id_{\Ra}\}$ generate $C_0(E^0)$ and $C_d(E^1)$ respectively.  Recall that the definition for $\xi\in C_0(E^1)$ to be in $C_d(E^1)$ is that
$$\sum_{e\in d^{-1}(v)}|\xi(e)|^2<\infty$$
for all $v\in E^0$. Since $d$ is injective on a given $E^1_\alpha$, we just show that $\{\chi_{\Me^\alpha_A}:A\in \Id_{\Ra}\}$ generates $C(E^1_\alpha)$ for each $\alpha\in \Labe$. Then, Proposition \ref{proposition_2} proves the claim.

Therefore, we define $T^0:C_0(E^0)\longrightarrow \mathcal{A}$ by $\chi_{\Ne_A}\longmapsto P_A$ for every $A\in \B$, and $T^1:C_d(E^1)\longrightarrow \mathcal{A}$ by $\chi_{\Me^\alpha_A}\longmapsto S_\alpha P_A$ for every $A\in \Id_{\Ra}$ and $\alpha\in \Labe$.  $T^0$ is an $*$-homomorphism by \cite[Lemma B.1]{Toke1}, and $T^1$ is a well-defined linear map since it decreases the  norm. Given $\alpha,\beta\in \Labe$ , $A\in \Id_{\Ra}$ and $B\in \Id_{\mathcal{R}_\beta}$,
$$T^1(\chi_{\Me^\alpha_A})^*T^1(\chi_{\Me^\beta_B})=(S_\alpha P_A)^*S_\beta P_B=\delta_{\alpha,\beta} P_{A\cap B}.$$ 
Observe that, given $e\neq e'$ with $d(e)=d(e')=v$, if $e\in E^1_\alpha$ for some $\alpha\in \Labe$ then $e'\notin E^1_\alpha$. Indeed, let $e=e^\alpha_{\Cu}$ and $e'=e^\beta_{\Cu'}$ for some $\alpha,\beta\in \Labe$, $\Cu\in\widehat{\Id_{\Ra}}$ and $\Cu'\in \widehat{\Id_{\mathcal{R}_\beta}}$. By hypothesis  $v_{\Cu}=d(e^\alpha_{\Cu})=d(e^\beta_{\Cu'})=v_{\Cu'}$, so $\Cu=\Cu'$. But since $e^\alpha_{\Cu}\neq e^\beta_{\Cu}$, it implies that $\alpha\neq \beta$.

Therefore, 
\begin{align*}\langle\chi_{\Me^\alpha_A}|\chi_{\Me^\beta_B}\rangle(v) & =\sum_{d(e)=v}\overline{\chi_{\Me^\alpha_A}(e)}\chi_{\Me^\beta_B}(e) \\ & =\delta_{\alpha,\beta}\chi_{\Ne_A}\chi_{\Ne_B}(v)=\delta_{\alpha,\beta}\chi_{\Ne_A\cap\Ne_B}(v)\,,\end{align*}
and hence 
$$T^0(\langle\chi_{\Me^\alpha_A}|\chi_{\Me^\beta_B}\rangle)=\delta_{\alpha,\beta}P_{A\cap B}=T^1(\chi_{\Me^\alpha_A})^*T^1(\chi_{\Me^\beta_B})\,,$$
as desired.
Now let $\alpha\in \Labe$, $A\in \B$ and $B\in \Id_{\Ra}$. Then, 
$$T^0(\chi_{\Ne_A})T^1(\chi_{\Me^\alpha_B})=P_AS_\alpha P_B=S_{\alpha}P_{\theta_\alpha(A)}P_{B}=S_\alpha P_{\theta_\alpha(A)\cap B}\,.$$

Thus, given $e\in E^1$, and Lemma \ref{lemma_82}, we have that 
\begin{align*}\pi_r(\chi_{\Ne_A})(\chi_{\Me^\alpha_B})(e) & =\chi_{\Ne_A}(r(e))\chi_{\Me^\alpha_B}(e) \\ & =\chi_{\Me^\alpha_{\theta_\alpha(A)}}(e)\chi_{\Me^\alpha_B}(e) \\  & =\chi_{\Me^\alpha_{\theta_\alpha(A)}\cap\Ne_B}(e)\,.\end{align*}
Hence, 
$$T^0(\chi_{\Ne_A})T^1(\chi_{\Me^\alpha_B})=S_\alpha P_{\theta_\alpha(A)\cap B}=T^1(\pi_r(\chi_{\Ne_A})(\chi_{\Me^\alpha_B}))\,,$$
whence $(T^0,T^1)$ is a Toeplitz $E$-pair.

Finally, let $f\in C_0(E^0_{rg})$. We need to prove that $T^0(f)=\Phi(\pi_r(f))$, where $\Phi: \mathcal{K}(C_d(E^1))\longrightarrow B$ is the associated $*$-homomorphism associated to $(T^0,T^1)$. Given $\varepsilon>0$, we will construct $f'\in C_0(E^0_{rg})$ such that $\|f-f'\|<\varepsilon$ and such that $\Phi(\pi_r(f'))=T^0(f')$.  Let $K$ be a compact subset of $E^0_{rg}$ such that $\|f_{E^0_{rg}\setminus K}\|<\varepsilon$. Given $v\in K$, we define the open subset $Z_v:=\{w\in E^0_{rg}:\|f(v)-f(w)\|<\varepsilon\}$. Then, we can find $A_v\in \Breg$ such that $v\in \Ne_{A_v}\subseteq Z_v$. Therefore, we have that $K\subseteq \bigcup\limits_{v\in K}\Ne_{A_v}\subseteq E^0_{rg}$, but since $K$ is compact, there exist $v_1,\ldots,v_n\in K$ such that $K\subseteq \bigcup\limits^n_{i=1}\Ne_{A_{v_i}}$.  Observe that we also can assume that $\Ne_{A_{v_i}}\cap \Ne_{A_{v_j}}=\emptyset$ for $i\neq j$, and that $K=\bigcup\limits^n_{i=1}\Ne_{A_{v_i}}$. 
Then, we define $f':=\sum_{i=1}^n f(v_i)\chi_{\Ne_{A_{v_i}}}\in C_0(E^0_{rg})$. Clearly, $\|f-f'\|<\varepsilon$. We claim that 
$$\pi_r(f')=\sum_{i=1}^nf(v_i)\left(\sum_{\alpha\in \Delta_{A_{v_i}}} \theta_{	\chi_{\Me^\alpha_{\theta_\alpha(A_{v_i})}},\chi_{\Me^\alpha_{\theta_\alpha(A_{v_i})}}}\right)\,.$$
Indeed, let $\xi\in C_d(E^1)$ and $e\in E^1_\alpha$. Observe that $0<|\Delta_{A_{v_i}}|=\lambda_{A_{v_i}}<\infty$ for every $i=1,\ldots,n$. Then  we have that 
$$\left(\sum_{i=1}^nf(v_i)\left(\sum_{\beta\in \Delta_{A_{v_i}}} \theta_{	\chi_{\Me^\beta_{\theta_\beta(A_{v_i})}},\chi_{\Me^\beta_{\theta_\beta(A_{v_i})}}}\right)\right)(\xi)(e)=$$
$$\sum_{i=1}^nf(v_i)\left(\sum_{\beta\in \Delta_{A_{v_i}}} \chi_{\Me^\beta_{\theta_\beta(A_{v_i})}}(e)\langle \chi_{\Me^\beta_{\theta_\beta(A_{v_i})}}|\xi\rangle(d(e))\right)=$$
$$\sum_{i=1}^nf(v_i)\left(\sum_{\beta\in \Delta_{A_{v_i}}} \chi_{\Me^\beta_{\theta_\beta(A_{v_i})}}(e)\left(\sum_{d(e')=d(e)}\overline{ \chi_{\Me^\beta_{\theta_\beta(A_{v_i})}}(e')}\xi(e')\right)\right)=$$
$$\sum_{i=1}^nf(v_i)\left(\sum_{\beta\in \Delta_{A_{v_i}}} \chi_{\Me^\beta_{\theta_\beta(A_{v_i})}}(e)\xi(e)\right)\,.$$

Observe that, by Lemma \ref{lemma_82} and the fact that $\Ne_{A_{v_i}}\cap \Ne_{A_{v_j}}=\emptyset$ for $i\neq j$, we have that $r(e)\in K$ if and only if there exists a unique $1\leq k\leq n$ such that  $e\in \Me^\alpha_{\theta_\alpha(A_{v_k})}$.
Then, 
$$\sum_{i=1}^nf(v_i)\left(\sum_{\beta\in \Delta_{A_{v_i}}} \chi_{\Me^\beta_{\theta_\beta(A_{v_i})}}(e)\xi(e)\right)=\sum_{i=1}^n f(v_i)\chi_{\Ne_{A_{v_i}}}(r(e))\xi(e)=\pi_r(f')\xi(e)\,,$$
as desired.

Finally, since $\{P_A,S_\alpha\}$ is a representation of $(\B,\Labe,\theta)$, we have that
$$T^0(f')=\sum_{i=1}^n f(v_i) T^0\left(\chi_{\Ne_{A_{v_i}}}\right)=\sum_{i=1}^n f(v_i)P_{A_{v_i}}=\sum_{i=1}^n f(v_i)\sum_{\alpha\in \Delta_{A_{v_i}}}S_\alpha P_{\theta_{\alpha}(A_{v_i})}S_\alpha^*\,,$$
because $A_{v_i}\in \Breg$.
But 
$$\Phi(\pi_r(f'))=\Phi\left(\sum_{i=1}^nf(v_i)\left(\sum_{\alpha\in \Delta_{A_{v_i}}} \theta_{	\chi_{\Me^\alpha_{\theta_\alpha(A_{v_i})}},\chi_{\Me^\alpha_{\theta_\alpha(A_{v_i})}}}\right)\right)=$$
$$\sum_{i=1}^nf(v_i)\left(\sum_{\alpha\in \Delta_{A_{v_i}}} T^1(\chi_{\Me^\alpha_{\theta_\alpha(A_{v_i})}})T^1(\chi_{\Me^\alpha_{\theta_\alpha(A_{v_i})}})^*\right)=\sum_{i=1}^nf(v_i)\sum_{\alpha\in \Delta_{A_{v_i}}} S_\alpha P_{\theta_\alpha(A_{v_i})}(S_\alpha P_{\theta_\alpha(A_{v_i})} )^*=$$
$$\sum_{i=1}^n f(v_i)\sum_{\alpha\in \Delta_{A_{v_i}}}S_\alpha P_{\theta_{\alpha}(A_{v_i})}S_\alpha^*=T^0(f')\,.$$
Thus, $(T^0,T^1)$ is a Cuntz-Krieger $E$-pair, as desired.
\end{proof}

We can use the characterization of $\Clab$ as a topological graph to deduce the following results:

\begin{corol}\label{CK_unital}\cite[Section 6]{KatsuraI} Let $(\B,\Labe,\theta)$ be a Boolean dynamical system. 
\begin{enumerate}
\item $\Clab$ is nuclear,
\item if $\B$ is a unital Boolean algebra then $\Clab$ is unital,
\item if $\B$  and $\Labe$ are countable then $\Clab$ satisfies the Universal Coefficients Theorem.
\end{enumerate}
\end{corol}

Our intention now is to state a gauge invariant theorem for $C^*(\B,\Labe,\theta)$. By the universality of $\mathcal{O}(E)$, there exists a gauge action $\beta':\mathbb{T}\curvearrowright \text{Aut }(\mathcal{O}(E))$ defined by $\beta'_z(t^0(f))=t^0(f)$ and $\beta'_z(t^1(\xi))=zt^1(\xi)$ for $f\in C_0(E^0)$, $\xi\in C_d(E^1)$ and $z\in \mathbb{T}$. Moreover, the map $\Phi:C^*(\B,\Labe,\theta)\longrightarrow \mathcal{O}(E)$, defined by $p_A\longmapsto t^0(\chi_{\Ne_A})$ and $s_\alpha\longmapsto t^1(\chi_{\Me_\alpha})$ for $A\in\B$ and $\alpha\in \Labe$, is an isomorphism. Then, it is clear that $\beta'_z\circ \Psi=\Psi\circ \beta_z$ for $z\in \mathbb{T}$, where $\beta$ is the gauge action of $C^*(\B,\Labe,\theta)$ defined in Section \ref{section_2}. Therefore, using the above isomorphism $\Psi$, we will not make distinction between $C^*(\B,\Labe,\theta)$ and $\mathcal{O}(E)$, and between their respective gauge actions $\beta$ and $\beta'$.

\begin{theor}\label{theorem_2} Let $(\B,\Labe,\theta)$ be a Boolean dynamical system, and let $\{P_A,S_\alpha\}$ be a Cuntz-Krieger representation of $(\B,\Labe,\theta)$ in $\mathcal{A}$. Suppose that $P_A\neq 0$ whenever $A\neq\emptyset$, and that there is a strongly continuous action $\gamma$ of $\mathbb{T}$ on $C^*(P_A,S_\alpha)\subseteq \mathcal{A}$, such that for all $z\in \mathbb{T}$ we have that $\gamma_z\circ \pi_{S,P}=\pi_{S,P}\circ \beta_z$. Then, $\pi_{S,P}$ is injective.
\end{theor}
\begin{proof} The result follows  by Theorem \ref{proposition_51}, the above comment and \cite[Theorem 4.5]{KatsuraI}.
\end{proof}

Finally we will compute the $K$-Theory of Cuntz-Krieger Boolean algebras. To do that, we will use the above characterization as topological graph $C^*$-algebra, and then we will use the results of Katsura \cite[Section 6]{KatsuraI} to give a $6$-term exact sequence that allows to compute the $K$-Theory of the Cuntz-Krieger Boolean algebra. The peculiarity of the space, that is $0$-dimensional, implies that this computation reduces to computing the kernel and cokernel of a map between the $K$-groups of certain subspaces of the vertex spaces.

First recall that, given a topological graph $E$, there is a $6$-term exact sequence
$$\xymatrix{K_0(C_0(E^0_{rg})) \ar[rr]^{\iota_*-[\pi_r]} & & K_0(C_0(E^0))\ar[rr]&  &K_0(\mathcal{O}(E))\ar[d] \\ K_1(\mathcal{O}(E)) \ar[u] &  &K_1(C_0(E_{rg}))\ar[ll]& & K_1(C_0(E^0_{rg}))\ar[ll]_{\iota_*-[\pi_r] } }$$
where $\iota:C_0(E^0_{rg})\rightarrow C_0(E^0)$ is the natural  map, and $\pi_r:C_0(E^0_{rg})\rightarrow \mathcal{K}(C_d(E^1))$.

Let $(\B,\Labe,\theta)$ be a Boolean dynamical system, and let $E$ be the associated topological graph. Recall that  $E^0=\widehat{\B}$, that $E^0_{rg}=\widehat{\Breg}$, and that by the Stone's Representation Theorem we have that
$$\xymatrix{K_0(C^*(\Breg)) \ar[rr]^{\iota_*-[\pi_r]} & & K_0(C^*(\B))\ar[rr]&  &K_0(\mathcal{O}(E))\ar[d] \\ K_1(\mathcal{O}(E)) \ar[u] &  & 0\ar[ll]& & 0\ar[ll] }\,.$$
Observe that, since $E^0$ is a $0$-dimensional space,  we have that 
$$K_0(C^*(\B))=K_0(C_0(E^0))=C_0(E^0,\Z)=C(\B,\Z)\,,$$ 
where $C(\B,\Z)$ is the $\Z$-linear span of the functions  defined on $\B$ by
$$\chi_A(B)=\left\lbrace \begin{array}{ll} 1 & \text{if } A\cap B\neq \emptyset \\ 0 & \text{otherwise }\end{array}\right. $$
for $A,B\in \B$. 

Now, given $A\in \Breg$, we have that the characteristic function $\chi_{\Ne_A}\in C_0(E^0_{rg})$, and hence $\pi_r(\chi_{\Ne_A})=\sum_{\alpha\in \Delta_A}\Theta_{\chi_{\Me^\alpha_{\theta_{\alpha}(A)}},\chi_{\Me^\alpha_{\theta_{\alpha}(A)}}}$. Therefore, the map $[\pi_r]:C(\Breg,\Z)\rightarrow C(\B,\Z)$ is given by $\chi_A\mapsto \sum_{\alpha\in \Delta_A}\chi_{\theta_{\alpha}(A)}$ for every $A\in \Breg$.

\begin{prop}[{cf. \cite[Proposition 6.9]{KatsuraI}}]\label{K-theory} Let $(\B,\Labe,\theta)$ be a Boolean dynamical system. Then, $K_0(\Clab)\cong \text{Ker }(Id-[\pi_r])$ and $K_1(\Clab)\cong \text{Coker }(Id-[\pi_r])$, where $Id-[\pi_r]:C(\Breg,\Z)\rightarrow C(\B,\Z)$ is given by $\chi_A\mapsto \chi_A-\sum_{\alpha\in \Delta_A}\chi_{\theta_{\alpha(A)}}$ for $A\in \Breg$. 
\end{prop}

\begin{rema}  We would like to remark that Corollary \ref{CK_unital} is a generalization
of \cite[Corollary 3.11]{BPIII}, that Theorem \ref{theorem_2} is a generalization of \cite[Corollary 3.10]{BPIII},  and that Proposition \ref{K-theory}  is a generalization of \cite[Theorem 4.4]{BPIII}.
\end{rema}

\section{An $*$-inverse semigroup}

In this section we will associate to $\Clab$ an $\ast$-inverse semigroup, which will help us to construct the groupoid used to represent the above algebra as a groupoid $C^*$-algebra. In order to attain our goal we will first associate to $\Clab$ a suitable $*$-inverse semigroup. 
  
\begin{defi}
$$T=T_{(\B,\Labe,\theta)}:=\{s_\alpha p_A s^*_\beta: \alpha,\beta\in \Labe^*\,, A\in \B\,, A\subseteq \Ri_\alpha\cap\Ri_\beta\neq \emptyset\}\cup\{0\}\subseteq \Clab\,.$$
%Recall that given $A\in \B$ we have that $s_A:=p_A$.
\end{defi}

%Clearly, $T\subseteq \Clab$. Now,

\begin{prop}\label{prop1}
$T$ is an $*$-inverse semigroup.
\end{prop}
\begin{proof}
First notice that, given $\alpha, \beta\in \Labe^*$ and $A\in \B$,
$$s_\alpha p_A s^*_\beta=s_\alpha p_{A\cap\Ri_\alpha\cap\Ri_\beta} s^*_\beta\,,$$
so the assumption implies that $s_\alpha p_A s^*_\beta\neq 0$.

 Now given $s_\alpha p_A s^*_\beta$, $s_\gamma p_B s^*_\delta$, we have that 
$$s_\alpha p_A s^*_\beta\cdot s_\gamma p_B s^*_\delta=\left\{
\begin{array}{ll}
s_{\alpha\gamma'}p_{\theta_{\gamma'}(A)\cap B} s_\delta^* & \text{if } \gamma=\beta\gamma'\text {and } \Ri_{\alpha\gamma'}\cap \Ri_\delta\neq \emptyset \\
s_\alpha p_{A\cap \theta_{\beta'}(B)} s^*_{\delta\beta'}& \text{if }\beta=\gamma\beta'\text{ and } \Ri_\alpha\cap \Ri_{\delta\beta'}\neq\emptyset \\
s_\alpha p_{A\cap B} s_\delta& \text{if }\gamma=\beta\text{ and } \Ri_\alpha\cap \Ri_\gamma\neq \emptyset \\
0 & \text{otherwise}
\end{array}
\right.$$
So $T$ is closed under multiplication. Moreover, 
$$(s_\alpha p_A s^*_\beta)^*=s_\beta p_As^*_\alpha$$
for every $\alpha,\beta\in \Labe^*$ $A\in \B$  with $A\subseteq \Ri_\alpha\cap \Ri_\beta\neq \emptyset$. Thus, $T$ is an $*$-semigroup with $0$.\vspace{.2truecm}

Next, notice that for any $s=s_\alpha p_A s^*_\beta\in T$, we have that $s=ss^*s$: 
$$ss^*s=(s_\alpha p_A s^*_\beta\cdot s_\beta p_A s^*_\alpha)\cdot s_\alpha p_A s^*_\beta=s_\alpha p_A s^*_\alpha\cdot s_\alpha p_A s^*_\beta=s_\alpha p_A s^*_\beta=s\,.$$
Thus, every $s\in T$ is a partial isometry.\vspace{.2truecm}

Finally, notice that the idempotents $ss^*$, for $s\in T$, have the form $s_\alpha p_A s^*_\alpha$. Hence,
$$s_\alpha p_A s^*_\alpha \cdot s_\beta p_B s^*_\beta=\left\{
\begin{array}{ll}
s_{\beta}p_{\theta_{\beta'}(A)\cap B} s^*_\beta & \text{if } \beta=\alpha\beta' \\
 s_\alpha p_{A\cap \theta_{\alpha'}(B)} s^*_{\alpha}& \text{if }\alpha=\beta\alpha'\\
s_\alpha p_{A\cap B} s^*_\alpha & \text{if }\alpha=\beta \\
0 & \text{otherwise}
\end{array}
\right.$$
and it is straightforward to check that these projections pairwise commute. Thus, $T$ is an $\ast$-inverse semigroup by \cite[Theorem 1.1.3]{Lawson}.
\end{proof}

\begin{corol}\label{corol1}
$\Clab=\overline{\text{span}}\{x:x\in T\}$
\end{corol}

\begin{defi} 
We will define $\ET$ to be the set of idempotents of $T$.
\end{defi}

In order to go forward, we want to keep control of the natural ordering of $\ET$.

\begin{lem}\label{lem2} Let $\alpha,\beta\in \Labe^*$, $A\in \B$. Then:
\begin{enumerate}
\item If either $\alpha\ne \emptyset$ or $\alpha=\beta=\emptyset$, then $s_\alpha p_A s^*_\alpha\leq s_\beta p_B s^*_\beta$ if and only if $\alpha=\beta\alpha'$ and $A\subseteq \theta_{\alpha'}(B)$.
\item If $\alpha=\emptyset$ and $\beta \ne \emptyset$, then $s_\alpha p_A s^*_\alpha\leq s_\beta p_B s^*_\beta$ if and only if: (i) $\Delta_A=\{\beta\}$ and (ii) $\theta_\beta(A)\subseteq B$.
\end{enumerate} 
\end{lem}
\begin{proof}
$(1)$ $s_\alpha p_A s^*_\alpha\leq s_\beta p_B s^*_\beta$ if and only if $s_\alpha p_A s^*_\alpha=s_\alpha p_A s^*_\alpha\cdot s_\beta p_B s^*_\beta$ if and only if $\alpha=\beta\alpha'$ and $A\subseteq \theta_{\alpha'}(B)$ by Proposition \ref{prop1}.

$(2)$ If $\alpha=\emptyset$, then $s_\alpha p_A s^*_\alpha=p_A$. Hence, if $p_A\leq s_\beta p_B s^*_\beta$, then $p_A=p_A\cdot s_\beta p_B s^*_\beta=s_\beta p_{\theta_\beta(A)\cap B} s^*_\beta$. Multiplying on the right side by $s_\beta$ we have that $p_A s_\beta=s_\beta p_{\theta_\beta(A)\cap B}$, and multiplying on the left side by $s^*_\beta$ we have that $s^*_\beta p_A s_\beta=p_{\theta_\beta(A)\cap B}$. Since $s^*_\beta p_As_\beta=p_{\theta_{\beta}(A)}$, we have $\theta_\beta(A)\subseteq B$. Moreover, $p_A=s_\beta p_{\theta_\beta(A)} s^*_\beta$ means that $\Delta_A=\{\beta\}$.

Conversely, if $\Delta_A=\{\beta\}$ and $\theta_\beta(A)\subseteq B$, then $p_A=s_\alpha p_{\theta_\beta(A)} s^*_\beta=s_\beta p_{\theta_\beta(A)\cap B} s^*_\beta=p_A\cdot s_\beta p_B s^*_\beta$, whence $p_A\leq s_\beta p_B s^*_\beta$.  
\end{proof}

In order to prove the next property of $T$, we need a technical result.

\begin{lem}\label{lem3}
If $\emptyset \ne\alpha\in \Labe^*$ and $A\in \B$ with $A\subseteq \Ri_\alpha$, then $p_A\ne p_A s^*_\alpha$.
\end{lem}
\begin{proof}
Suppose that $p_A=p_A s^*_\alpha$. Since $p_A$ is a projection, we have that $$p_A=p_As^*_\alpha=(p_As^*_\alpha)^*=s_\alpha p_A\,,$$
whence $p_A=s_\alpha p_A s^*_\alpha$, which only occurs if $\Delta_A=\{\alpha\}$ and $\theta_\alpha(A)=A$. Now, given any $\emptyset\neq B\subseteq A$, it also follows that $0\neq p_B=p_Bp_A=p_Bp_As^*_\alpha=p_Bs^*_\alpha$, so  $\theta_\alpha(B)=B$ by the above argument.

Now, consider the ideal $\Id_A$
with unique action $\theta_\alpha$. Then $(\theta_\alpha)_{|\tilde{A}}=\text{id}$, whence $C^*(\Id_A,\alpha,\theta_\alpha)\cong C(\widehat{\Id_A},\mathbb{T})$. Since $C^*(\Id_A,\alpha,\theta_\alpha)$ has a faithful representation and any representation of 
$(\Id_A,\alpha,\theta_\alpha)$ induces a representation of $(\B,\Labe,\theta)$, we get a contradiction.
\end{proof}

\begin{defi} A $*$-inverse semigroup  $T$ is $E^*$-unitary if for every $s\in T$, $e\in \mathcal{E}(T)$, if $e\leq s$ then $s\in \mathcal{E}(T)$.
\end{defi}
\begin{prop}\label{prop2} 
$T$ is a $E^*$-unitary inverse semigroup.
\end{prop}
\begin{proof} We need to check the $6$ possible cases:
\begin{enumerate}
\item $s_\gamma p_B s^*_\gamma\leq s_\alpha p_A s^*_\beta$ if and only if 
$$s_\gamma p_B s^*_\gamma=s_\alpha p_A s^*_\beta\cdot s_\gamma p_B s^*_\gamma=s_\alpha s^*_\beta s_\beta s_\delta  p_B s^*_\gamma=\qquad (\gamma=\beta\delta) $$
$$=s_\alpha p_A p_{\Ra_\beta} s_\delta p_B s^*_\gamma=s_\alpha p_A s_\delta p_B s^*_\gamma=s_{\alpha\delta} p_{\theta_\delta(A)\cap B} s^*_\gamma$$
if and only if $\alpha\delta=\gamma=\beta \delta$, whence $\alpha=\beta$ and then $s_\alpha p_A s^*_\alpha\in \ET$.
\item $s_\gamma p_Bs^*_\gamma\leq s_\alpha p_A$ if and only if
$$s_\gamma p_Bs^*_\gamma=s_\alpha p_A\cdot s_\gamma p_Bs^*_\gamma=s_{\alpha\gamma}p_{\theta_\gamma(A)\cap B} s^*_\gamma$$ if and only if $\alpha\gamma=\gamma$, i.e., $\alpha=\emptyset$, whence $s_\alpha p_A=p_A\in \ET$.
\item  $s_\gamma p_Bs^*_\gamma\leq  p_A s^*_\alpha$, this case is analog to $(2)$.
\item $p_B\leq s_\alpha p_A s^*_\beta$ if and only if 
$$p_B=s_\alpha p_A s^*_\beta \cdot p_B=s_\alpha p_{A\cap \theta_\beta(B)} s^*_\beta=s_\beta p_{A\cap \theta_\beta(B)} s^*_\alpha\,.$$
Thus, 
$$p_{A\cap \theta_\beta(B)}=s^*_\alpha p_B s_\beta=s^*_\alpha s_\beta p_{\theta_\beta(B)}\,.$$
By Lemma \ref{lem3} , the only possibility is that $\alpha=\beta$, whence $s_\alpha p_A s^*_\alpha\in \ET$.
\item $p_B\leq s_\alpha p_A$ if and only if $p_B=s_\alpha p_A\cdot p_B= s_\alpha p_{A\cap B}$. Thus, by Lemma \ref{lem3} $\alpha=\emptyset$, whence $ s_\alpha p_A\in \ET$.
\item $p_B\leq p_A s^*_\alpha$, this case is analog to case $(5)$.
\end{enumerate}
\end{proof}

Proposition \ref{prop2} will play an important role in the sequel. We also need to determine the orthogonality of idempotents.
\begin{lem} 
$s_\alpha p_A s^*_\alpha\cdot s_\beta p_B s^*_\beta=0$ if and only if either
\begin{enumerate}
\item  $\alpha\nleq \beta$ and $\beta\nleq \alpha$, or
\item $\beta=\alpha\beta'$ and $\theta_{\beta'}(A)\cap B=\emptyset$, or 
\item $\alpha=\beta \alpha'$ and $\theta_{\alpha'}(B)\cap A=\emptyset$.
\end{enumerate}
\end{lem}
\begin{proof} It is a simple computation, according Proposition \ref{prop1}. 
\end{proof}

\section{Tight representations of T.}

This intermediate step will help us to connect $\Clab$ with a universal $C^*$-algebra for a suitable family of representations of $T$. Concretely, the goal of this section is to prove that the map
$$\iota:T\longrightarrow \Clab^{\sim}$$
is the universal tight representation of $T$. Here, given a $C^*$-algebra $A$, $A^{\sim}$ will denote the (minimal) unitization of $A$, with the convention that $A^{\sim}=A$ in case $A$ already has a unit.\vspace{.2truecm}

First we recall some definitions from  \cite{Exel}.

\begin{defi} Set $\E=\E(T)$. Then:
\begin{enumerate}
\item Given $X,Y\subseteq \E$ finite subsets,
$$\E^{X,Y}:=\{z\in \E:z\leq x\, \text{ for all } x\in X\text{ and }z\bot y\,\text{ for all }  y\in Y\}\,.$$
\item Given any $F\subseteq \E$, we say that $Z\subseteq F$ is a cover for $F$ if for every $0\neq x\in F$ there exists $z\in Z$ such that $zx\neq 0$. $Z$ is cover for $y\in \E$ if it is a cover for $F=\{x\in \E: x\leq y\}$.
\item A representation $\varphi$ of $\E$ is tight if for every $X,Y\subseteq \E$ finite subsets, and for every finite cover $Z\subseteq \E^{X,Y}$,
$$\bigvee_{z\in Z}\varphi(z)\geq \bigwedge_{x\in X}\varphi(x)\wedge \bigwedge_{y\in Y}\neg\varphi(y)\hspace{.5truecm} (\dag),$$
where ``$\bigvee$'' refers to the operation of taking supremum of a commuting set of projections.
\end{enumerate} 
\end{defi}

\begin{rema}\label{Rem:Unitization}
{\rm
In terms of the algebra, identity $(\dag)$ above becomes

$$\bigvee_{z\in Z}\varphi(z)= \prod_{x\in X}\varphi(x) \prod_{y\in Y}(1-\varphi(y)),$$
and so, when looking for the tightness of a map, we shall assume that we are working with unital algebras, by using the unitization of an algebra when necessary.}
\end{rema}

Next result will help us to determine when a representation $\varphi$ is tight.

\begin{prop}[{\cite[Prop. 11.8]{Exel}}]\label{prop4} If $\varphi$ is a representation of $\E$ which satisfies :
\begin{enumerate}
\item $\E$ contains $X\subseteq \E$ finite such that $\bigvee\limits_{x\in X}\varphi(x)=1$, or 
\item $\E$ admits no finite cover,
\end{enumerate}
then $\varphi$ is tight if and only if for every $x\in \E$ and for every finite cover $Z\subseteq \E$ for $x$,
$$\bigvee_{z\in Z}\varphi(z)\geq \varphi(x)\,.$$
%$\hspace{\fill} \Box$
\end{prop}
\vspace{.2truecm}

In order to apply Proposition \ref{prop4} to our case, first observe that $\Clab$ is unital if and only if $\B$ is a unital Boolean algebra, with  suprema $\mathbf{1}$, and in this case $p_\mathbf{1}$ will be a finite cover for $T$. If $\Clab$ is not unital, then we have that $\{p_A\}_{A\in \B}$ is an approximate unit of projections. In particular, given a finite set $Y$ of elements of $\E$, there exists $A$ such that $p_Bep_B=e$ for every $e\in Y$ and $B\in \B$ with $A\subseteq B$.

Now, let $X\subseteq \E$ be a finite cover. Then, $X$ is of the form
$$\{p_A\}\cup \{s_{\alpha_i}p_{B_i}s^*_{\alpha_i}\}_{i=1}^n\,.$$
Let us define $C:=A\cup \bigcup\limits_{i=1}^n\D_{\alpha_i}\in \B$. Since $\Clab$ is not unital, and hence $\B$ has not suprema, there exists $\emptyset\neq D\in \B$ with  $C\cap D=\emptyset$. Therefore 
$$p_D\cap p_A=\emptyset\qquad \text{and}\qquad p_D\cdot s_{\alpha_i}p_{B_i}s^*_{\alpha_i}=0\qquad \forall i\in \{1,\ldots,n\}.$$

Then,

\begin{corol}\label{prop5} 
Proposition \ref{prop4} apply to $\ET$ for every $(\B,\Labe,\theta)$.
\end{corol}

Next step is to identify finite covers for $\Sigma_x=\{y\in \ET: y\leq x\}$, $x\in \ET$. But first a (probably well known) result.

\begin{lem}\label{lem_cover}
Let $T$ be any $*$-inverse semigroup, and let $\ET$ its semilattice of idempotents. Let $x\in E$ and $s\in S$ such that $x\leq s^*s$. Then $\{e_1\ldots,e_n\}$ is a finite cover for $\Sigma_x$ if and only if $\{se_1s^*,\ldots,se_n s^*\}$ is a finite cover for $\Sigma_{sxs^*}$. 
\end{lem}

Now, we need to fix a concept.

\begin{defi}
Given  $\emptyset\neq A\in \B$, we define an expansion of $A$ to be a finite set $\{\alpha_1,\ldots,\alpha_n\}\subseteq \Labe^*$ such that $\theta_{\alpha_i}(A)\neq \emptyset$ for every $1\leq i\leq n$. Moreover, we say that an expansion of $A$ is complete if $\alpha_i\nleq \alpha_j$ and $\alpha_j\nleq \alpha_i$ whenever $i\neq j$, and for every $\beta\in \Labe^*$ with $\theta_\beta(A)\neq\emptyset$ there exists $i$ such that either $\alpha_i\leq\beta$ or $\beta\leq \alpha_i$. Equivalently, $\{\alpha_1,\ldots,\alpha_n\}$ is a complete expansion for $A$ if $p_A=\sum_{i=1}^ns_{\alpha_i} p_{\theta_{\alpha_i}(A)}s^*_{\alpha_i}$.
\end{defi}

\begin{defi}
Given $\emptyset\neq A\in \B$, and $n\in \N$, we define 
$$\Delta^n_A:=\{\alpha\in \Labe^n:\theta_\alpha(A)\neq \emptyset\}\,,$$
and $\Delta_A^{\leq n}=\bigcup\limits_{k=1}^n \Delta^k_A$.
\end{defi}

\begin{defi}
Given a cover $Z$ of $\Sigma$, we say that $\hat{Z}$ is a refinement of $Z$ if $\hat{Z}$ is a cover of $\Sigma$, and for every element $x\in \hat{Z}$ there exists $y\in Z$ with $x\leq y$.

\end{defi}

\begin{noname}\label{Expansion}
{\rm 
Now we will analyse how look like the finite covers of $\Sigma_x$ for $x=p_A$ and $x=s_\alpha p_A s^*_\alpha$. By Lemma \ref{lem_cover} it will be enough to look at $x=p_A$. Then a finite cover for $\Sigma_x$ has the form 
$$Z=\{p_{B_i}\}_{i=1}^n\cup \{s_{\gamma_j}p_{C_j}s^*_{\gamma_j}\}_{j=1}^m\subseteq \Sigma_x\,.$$
Observe that we can joint all the idempotents $\{p_{B_i}\}_{i=1}^n$ in a single idempotent $p_B$ where $B:=\bigcup\limits_{i=1}^nB_i$, so 
$$Z=\{p_{B}\}\cup \{s_{\gamma_j}p_{C_j}s^*_{\gamma_j}\}_{j=1}^m\subseteq \Sigma_x\,.$$

Now, if $A\setminus B=A\setminus (A\cap B)\notin \Breg$, it means that there exists $C\subseteq A\setminus B$ with either $\lambda_C=0$ or $\lambda_C=\infty$. If $\lambda_C=0$ then we have that 
$$p_C\cdot s_{\gamma_j}p_{C_j}s^*_{\gamma_j}=s_{\gamma_j}p_{\theta_{\gamma_j}(C)\cap C_j}s^*_{\gamma_j}=s_{\gamma_j}p_{\emptyset}s^*_{\gamma_j}=0\qquad \forall j\in\{1,\ldots,m\}\,,$$
contradicting the fact that $Z$ is a cover of $p_A$. If $\lambda_C=\infty$, there exists $\beta\in \Labe$ such that $\beta\nleq \gamma_i$ for $1\leq i\leq m$. Thus, if we consider the element $s_\beta p_{\theta_\beta(C)} s^*_\beta$, then 
$$s_\beta p_{\theta_\beta(C)} s^*_\beta\cdot s_{\gamma_j}p_{C_j}s^*_{\gamma_j}=0\qquad \forall j\in\{1,\ldots,m\}\,,$$
and moreover, since
$$p_A\cdot s_\beta p_{\theta_\beta(C)} s^*_\beta=p_A\cdot p_C s_\beta s^*_\beta=0,$$
this contradicts that $Z$ is a cover for $\Sigma_x$. Therefore, $A\setminus B$ must be in $\Breg$ for $Z$ to be a cover.

Notice that $p_B$ covers all the elements of $\Sigma_x$ that are dominated by $p_{A\cap B}$. Thus, without loss of generality, we can assume that $$Z=\{s_{\gamma_i}p_{C_i}s^*_{\gamma_i}\}_{i=1}^n\,,$$
since $Z\subseteq \Sigma_x$ with $\theta_{\gamma_i}(A)\neq\emptyset$ for every $1\leq i\leq n$, where $x=p_A$ with $A\in \Breg$, and that $\gamma_i\neq \gamma_j$ whenever $i\neq j$.
\vspace{.2truecm}

Next, we see that $\{\gamma_i\}_{i=1}^n$ must contain a complete expansion for $A$. Otherwise, there exists $\beta\in \Labe^*$ with $\theta_\beta(A)\neq \emptyset$ with $\alpha_i\nleq \beta$ and $\beta\nleq \alpha_i$ for every $1\leq i\leq n$, and then $s_\beta p_{\theta_\beta(A)} s^*_{\beta}\leq p_A$ and $s_\beta p_{\theta_\beta(A)} s^*_{\beta}\leq p_A\cdot s_{\gamma_i}p_{C_i}s^*_{\gamma_i}=0$ for every $1\leq i\leq n$, contradicting that $Z$ is a cover for $p_A$. We relabel the complete expansion as $\gamma_1,\ldots,\gamma_l$ for some $1\leq l \leq n$. We can also take it minimal, so for every $k\geq l$ there exists $1\leq i\leq l$ with $\gamma_i\leq \gamma_k$.

Another important observation is that  $D_i:=\theta_{\gamma_i}(A) \setminus C_i\in \Breg$ whenever $1\leq i \leq l$. Indeed, let us first suppose that $\lambda_{D_i}=0$. Then, $0\neq s_{\gamma_i}p_{D_i} s^*_{\gamma_i}$ is the element that leads to contradiction with $Z$ being a cover of $p_A$. Now suppose that there exists $E_i\subseteq D_i$ with $\lambda_{E_i}=\infty$. Then, there exists $\beta\in \Delta_{E_i}$ such that $\gamma_i\beta\nleq \gamma_i$ for every $\gamma_j$ with $l+1\leq j \leq n$. Thus, the element $s_{\gamma_i\beta} p_{\theta_\beta(E_i)}s^*_{\gamma_i\beta}$ is the element that leads to contradiction with $Z$ being a cover of $p_A$.

We also have that, given $\gamma_i$ with $1\leq i\leq l$ such that $\gamma_i\nleq \gamma_j$ for every $j\geq l+1$, it must be $\theta_{\gamma_i}(A)\subseteq C_i$. Otherwise, the element $s_{\gamma_i}p_{\theta_{\gamma_i}(A)\setminus C_i} s^*_{\gamma_i}$ is the element that leads to contradiction with $Z$ being a cover of $p_A$. 

Now, we define $A_i:=\theta_{\gamma_i}(A)\setminus C_i$ for those $i\leq l$ such that $A_i\neq \emptyset$. So, there exist $\gamma_{i_1},\ldots,\gamma_{i_{k(i)}}$ with $\gamma_i\lneq \gamma_{i_{j}}$ for $1\leq j\leq k(i)$, and we define $E_{i,j}:=C_{i_j}$ for $1\leq j\leq k(i)$. We can relabel the $A_i$s as $A_1,\ldots, A_m$, and if we define $\beta_{i,j}:=\gamma_{i_j}\setminus \gamma_i$ for $1\leq j\leq k(i)$, then the sets $Z_i:=\{s_{\beta_{i,j}}p_{E_{i,j}}s_{\beta_{i,j}}^*\}$ are finite covers of $p_{A_i}$ for $1\leq i\leq m$. 

Now, must proceed as above with this new covers as many time as we need, and since they are finite covers, each step will have less elements than the previous. So, in a finite number of steps,  there will be a refinement of the cover that will contain a complete expansion $\{\gamma_i\}$ of $A$ with $C_i=\theta_{\gamma_i}(A)$.
  }
\end{noname}
 
Summarizing

\begin{lem}\label{lem5} If $Z\subseteq\Sigma_x$ is a finite cover for $x\in \ET$, there exists a refinement of $\hat{Z}$ of $Z$ such that:
\begin{enumerate}
\item $\hat{Z}\subseteq \Sigma_x$ is a finite cover,
\item The elements in $\hat{Z}$ are pairwise orthogonal,
\item $\bigvee\limits_{z\in Z}\rho(z)=\sum\limits_{\hat{z}\in \hat{Z}}\rho(\hat{z})$ for every representation $\rho$ of $\ET$. 	
\end{enumerate}
\end{lem}

We are ready to prove the main  result of this section. Notice that, because of Remark \ref{Rem:Unitization}, we need to require to the universal algebra for tight representations of $T$ being unital. Hence, we have the following

\begin{theor}
\label{theor1}
The representation $\iota:T\longrightarrow \Clab^{\sim}$ is the universal tight representation of $T$.
\end{theor}
\begin{proof}
First notice that, because of Corollary \ref{prop5} and Lemma \ref{lem5}, the representation $\iota: T\rightarrow \Clab^{\sim}$ is tight.

Now, let $A$ be any unital $C^*$-algebra, and suppose that $\rho: T\rightarrow A$ is a tight representation. Consider $\hat{s}_a:=\rho (s_a)$ for every $a\in \mathcal{L}$, and $\hat{p}_A:=\rho(p_A)$ for every $A\in\mathcal{B}$. Then, $\{\hat{s}_a : a\in \mathcal{L}\}\cup\{\hat{p}_A :A\in \mathcal{B}\}\subset A$, and clearly:
\begin{enumerate}
\item $\{\hat{p}_A :A\in \mathcal{B}\}$ is a set of projections in $A$.
\item $\{\hat{s}_a : a\in \mathcal{L}\}$ is a set of partial isometries in $A$.
\end{enumerate}

Since $\rho$ is a $\ast$-homomorphism of semigroups, we clearly have that:
\begin{enumerate}
\item $\hat{p}_A\hat{p}_B=\hat{p}_{A\cap B}$ for every $A,B\in \mathcal{B}$.
\item $\hat{p}_A\hat{s}_{a}=\hat{s}_{a}\hat{p}_{\theta_{a}(A)}$ for every $a \in \mathcal{L}$ and $A\in \mathcal{B}$.
\item $\hat{s}^*_a\hat{s}_b=\delta_{a,b}\hat{p}_{\mathcal{R}_a}$ for every $a,b\in \mathcal{L}$.
\end{enumerate}

In order to prove the two remaining identities, we will use the fact that $\rho$ is tight:
\begin{enumerate}
\item Take $A,B\in \mathcal{B}$. Then, it is clear that $\{p_{A\setminus B}, p_{A\cap B}\}$ is a finite orthogonal cover of $p_A$, and so does $\{p_{B\setminus A}, p_{A\cap B}\}$ of $p_B$. Hence, $\hat{p}_A=\hat{p}_{A\setminus B}+\hat{p}_{A\cap B}$ and $\hat{p}_B=\hat{p}_{B\setminus A}+\hat{p}_{A\cap B}$, whence $\hat{p}_A+\hat{p}_B-\hat{p}_{A\cap B}=\hat{p}_{A\setminus B}+\hat{p}_{B\setminus A}+\hat{p}_{A\cap B}$. Since $\{p_{A\setminus B}, p_{B\setminus A}, p_{A\cap B}\}$ is an orthogonal finite cover of $p_{A\cup B}$, we conclude that $\hat{p}_{A\cup B}=\hat{p}_{A\setminus B}+\hat{p}_{B\setminus A}+\hat{p}_{A\cap B}=\hat{p}_A+\hat{p}_B-\hat{p}_{A\cap B}$, as desired.
\item If $A\in \Breg$, then $\{s_ap_{\theta_a(A)}s^*_a : a\in \Delta_A\}$ is an orthogonal finite cover of $p_A$. Hence,
$$\hat{p}_A=\rho(p_A)=\bigvee\limits_{a\in \Delta_A}\rho (s_ap_{\theta_a(A)}s^*_a)=\bigvee\limits_{a\in \Delta_A}\hat{s}_a\hat{p}_{\theta_a(A)}\hat{s}^*_a=\sum\limits_{a\in \Delta_A}\hat{s}_a\hat{p}_{\theta_a(A)}\hat{s}^*_a,$$
so we are done.
\end{enumerate}

Thus, by the Universal Property of $\Clab$, there exists a unique $\ast$-homomorphism

$$
\begin{array}{crcc}
\psi : &  \Clab^{\sim} & \rightarrow  & A  \\
 & s_a & \mapsto  & \hat{s}_a \\ 
 & p_A & \mapsto   &  \hat{p}_A\\
 &  1 & \mapsto& 1
\end{array}.
$$

Since $\psi\circ \iota=\rho$, the universality of $\iota$ is proved.
\end{proof}

%Recall $\pi:S\rightarrow T$ is an onto $\ast$-semigroup homomorphism. By Lemma \ref{Densitat} and (\ref{injectivity_of_pi}), the lack of injectivity of $\pi$ is linked to the existence of dense pairs of idempotents in $\Sa$. By \cite[Proposition 2.11]{ExelTight}, it is then immediate to conclude

\begin{corol}\label{corol3} $\Clab^{\sim}\cong C^*_{\text{tight}}(T)$. 
\end{corol}

\section{The tight groupoid of $T$}

In this section we will benefit of the previous work to construct a groupoid $\mathcal{G}$ such that $\Clab\cong C^*(\mathcal{G})$. Now, we proceed to recall the construction of $\mathcal{G}_{\text{tight}}(T)$. Let us recall  the construction in a generic form (see e.g. \cite{ExPa}):

\begin{itemize}
\item If $T$ is an inverse semigroup, then $\E=\ET=\{\text{idempotents of }T \}$ is a semilattice with ordering $e\leq f$ if and only if $ef=e$, and $e\wedge f=ef$. It extends to an order in $S$, $s\leq t$ if and only if $s=ts^*s=ss^*t$. We denote by $e\bot f$ if and only if $ef=0$, and $e\Cap f$ if and only if $ef\neq 0$.
\item A character on $\E$ is a nonzero map $\phi:\E\rightarrow\{0,1\}$ with $\phi(0)=0$, and $\phi(ef)=\phi(e)\phi(f)$ for every $e,f\in \E$. We denote the set of characters by $\widehat{\E}_0$. This is a topological space when equipped with the product topology inherited from $\{0,1\}^{\E}$. Since the zero map does not belong to  $\widehat{\E}_0$, it is a locally compact space and totally disconnected Hausdorff space.
\item A filter in $\E$ is a nonempty subset $\eta\subseteq \E$ such that:
\begin{enumerate}
\item $0\notin \eta$,
\item closed under $\wedge$,
\item $f\geq e\in \eta$ implies $f\in \eta$.
\end{enumerate}
\item Given a filter $\eta$,
$$\begin{array}{rl} \phi_\eta:\E & \longrightarrow \{0,1\} \\ e & \longrightarrow [e\in \eta]\end{array}$$
is a character. Conversely, if $\phi\in \widehat{\E}_0$, then $\eta_\phi=\{e\in \E| \phi(e)=1\}$ is a filter. These correspondences  are mutually inverses.
\item A filter $\eta$ is a ultrafilter if it is not properly contained in another filter. We denote $\widehat{\E}_\infty\subseteq \widehat{\E}_0$ the space of ultrafilters.
\item Tight filters are defined in analogy with tight representations. The set of tight filters (tight spectrum) is a closed subspace $\widehat{\E}_{\text{tight}}$ of $\widehat{\E}_0$, containing $\widehat{\E}_\infty$ as a dense subspace. 
\item We can define a standard action of $T$ on $\widehat{\E}_0$ as follows:
\begin{enumerate}
\item For each $e\in \E$, $D^\beta_e=\{\phi\in \widehat{\E}_0: \phi(e)=1\}$,
\item given $s\in T$,
$$\begin{array}{rl} \beta_s:D^\beta_{s^*s}& \longrightarrow D^\beta_{ss^*} \\ \phi & \longrightarrow \beta_s(\phi)(e)=\phi(s^*es)\end{array}$$
When working with filters, $D^\beta_e=\{\eta\in  \widehat{\E}_0| e\in \eta\}$ while $\beta_s(\eta)=\{f\in \E: f\geq ses^*\text{ for every }e\in \eta\}$.

\end{enumerate}
\item $\beta$ restricts to an action of $T$ on ultrafilters and on tight filters.
\end{itemize} 

\begin{defi}
Consider the set $\Omega=\{(s,x)\in T\times  \widehat{\E}_{\text{tight}}: x\in D^\beta_{s^*s}\}$ and define $(s,x)\sim (t,y)$ if and only if $x=y$ and exists $e\in \E$ such that $x\in D^\beta_e$ and $se=te$. 

Define $\mathcal{G}_{\text{tight}}(S)=\Omega/\sim $, with:
\begin{enumerate}
\item $d([s,x])=x$ and $r([s,x])=\beta_s(x)$,
\item $[s,z]\cdot [t,x]=[st,x]$ if and only if $z=\beta_t(x)$,
\item $[s,x]^{-1}=[s^*,\beta_s(x)]$,
\item $\mathcal{G}^{(0)}_{\text{tight}}=\{[e,x]: e\in \E\}\cong \widehat{\E}_{\text{tight}}$
\end{enumerate} 
$\mathcal{G}_{\text{tight}}(T)$ is the tight groupoid of the inverse semigroup $T$.
\end{defi}

Then, we have

\begin{lem}\label{lem6} 
$\mathcal{G}_{\text{tight}}(T)$ is Hausdorff. 
\end{lem}
\begin{proof} 
By Proposition \ref{prop2} and \cite[Corollary 3.17]{ExPa}
\end{proof}

Moreover, if we restrict our attention to the case of the inverse semigroup $T$ being countable (which corresponds to the requirement that both $\B$ and $\mathcal{L}$ are countable), then we can prove the following facts

\begin{theor}\label{theor2} 
If $\B$ and $\mathcal{L}$ are countable, then $\Clab\cong C^*(\mathcal{G}_{\text{tight}}(T))$.
\end{theor}
\begin{proof} 
Since $T$ is countable, the result holds by Definition \ref{DefAlgebra}, Corollary \ref{corol3} and \cite[Theorem 2.4]{Reconst}, because $C^*(\mathcal{G}_{\text{tight}}(T))$ is the closed $\ast$-subalgebra of $C^*(\mathcal{G}_{\text{tight}}(T))^{\sim}$ generated by $\{1_s : s\in T\}$, and this algebra is isomorphic to $\Clab$ because of Corollary \ref{corol1}.
\end{proof}

and

\begin{lem}\label{lem_nuclear}
If $\B$ and $\mathcal{L}$ are countable, then $\mathcal{G}_{\text{tight}}(T)$ is amenable. 
\end{lem}
\begin{proof}
Since $C^*(\Gti(T))\cong \Clab$ is nuclear, then $C_{red}^*(\Gti(T))=C^*(\Gti(T))$, and thus $C_{red}^*(\Gti(T))$ is nuclear. Hence, the result holds by \cite[Theorem 5.6.18]{BrownOzawa}.
\end{proof}

Suppose that $\B$ and $\mathcal{L}$ are countable. Then, since $\Gti(T)$ is the tight groupoid of an countable $\ast$-inverse semigroup, $\Gti(T)$ is an \'etale, second countable, topological groupoid \cite{Exel}. Hence, because of \cite[Lemma 3.3 \& Proposition 10.7]{Tu}, Lemma \ref{lem6} and Lemma \ref{lem_nuclear}, we conclude

\begin{lem}\label{UCT_groupoid}
If $\B$ and $\mathcal{L}$ are countable, then $C^*(\mathcal{G}_{\text{tight}}(T))$ is in the UCT class.
\end{lem}

Notice that Lemma \ref{UCT_groupoid} proves Corollary \ref{CK_unital}(3) using groupoids instead of topological graphs.

\section{Simplicity of $\Clab$}

In this section we will characterise when $\Clab$ is simple, using information from $\Gti(T)$. To this end, we use a result of \cite{Clarketal}.

\begin{theor}[{\cite[Theorem 5.1]{Clarketal}}]\label{theor3}
Let $\mathcal{G}$ be an \'etale, Hausdorff, second countable, topological groupoid. If $\mathcal{G}$ is (elementary) amenable, then the following are equivalent:
\begin{enumerate}
\item $\mathcal{G}$ is minimal and essentially principal,
\item $C^*(\mathcal{G})$ is simple.
\end{enumerate} 
\end{theor}   

If $\B$ and $\mathcal{L}$ are countable then, since $\Gti(T)$ is the tight groupoid of an countable $\ast$-inverse semigroup, $\Gti(T)$ is an \'etale, second countable, topological groupoid \cite{Exel}. We know that $\Gti(T)$ is Hausdorff and amenable. Hence, we need only to take care of $\Gti(T)$ being essentially principal and minimal. As $\Gti(T)$ is the tight groupoid of an inverse semigroup, we can benefit of the results of \cite{ExPa} for this task. 

\subsection{Essentially principal groupoids}

In this subsection we take care of the essential principal property.  For this and related properties we refer to \cite[Section 4]{ExPa}. In particular, we skip the definitions.
\vspace{.2truecm}

Recall the following facts. 

\begin{theor}[{\cite[Theorem 4.7]{ExPa}}]\label{theor4} $\Gti(T)$ is essentially principal if and only if $\beta:T\acts  \widehat{\mathcal{E}}_{\text{tight}}$ is topologically free.
\end{theor}   

\begin{defi}[{\cite[Definition 4.8]{ExPa}}]  Let $s\in T$, $e\in \ET$ such that $e\leq ss^*$. Then, we say that:
\begin{enumerate}
\item $e$ is fixed under $s$ if $se=e$.
\item $e$ is weakly fixed under $s$, if $sfs^*\Cap f$ for every $f\in \ET\setminus\{0\}$ and $f\leq e$.
\end{enumerate}
\end{defi}

\begin{theor}[{\cite[Theorem 4.10]{ExPa}}]\label{theor5} Since $ {\mathcal{G}}_{\text{tight}}(T)$ is Hausdorff, the following statements are equivalent:
\begin{enumerate}
\item $\beta:T \acts  \widehat{\mathcal{E}}_{\text{tight}}$ is topologically free.
\item for every $s\in T$ and every $e\in \ET$ weakly fixed under $s$, there exists $F\subseteq\Sigma_e$ finite cover consisting of  fixed elements.
\end{enumerate}
\end{theor}

%\begin{defi}  $\mbox{ }$
%\begin{enumerate}
%\item We say that $\alpha=\alpha_1\cdots\alpha_n\in \Labe^*$ is a cycle without exits if for any $\emptyset\neq A\in \B$ such that %$\alpha_1\in \Delta_A$ we have that $\theta_{\alpha_1\cdots\alpha_t}(A)\in\Breg$ with  %$\Delta_{\theta_{\alpha_1\cdots\alpha_t}(A)}=\{\alpha_{t+1}\}$ for $t<n$ and $\theta_\alpha(A)\in\Breg$ with %$\Delta_{\theta_\alpha(A)}=\{\alpha_1\}$. 
%\item We say that $(\B,\Labe,\theta)$ satisfies Condition $(L_\B)$  if given $\alpha\in \Labe^*$ a cycle without exits and %$\emptyset\neq A\subseteq \Ri_\alpha$, there exist $k\in\N\cup\{0\}$ and $\emptyset\neq B\subseteq \theta_{\alpha^k}(A)$ with $B\cap %\theta_{\alpha}(B)= \emptyset$. 
%\end{enumerate}
%\end{defi}

\begin{defi}  $\mbox{ }$ Let $(\B,\Labe,\theta)$ be a Boolean dynamical system.
\begin{enumerate}
\item We say that the pair $(\alpha,A)$ with  $\alpha=\alpha_1\cdots\alpha_n\in \Labe^n$, $n\geq 1$, and $\emptyset\neq A\in \B$ with $A\subseteq\Ri_\alpha$, is a cycle if given $k\in\N\cup\{0\}$ we have that $\theta_{\alpha^k}(A)\neq \emptyset$ and for every $\emptyset\neq B\subseteq \theta_{\alpha^k}(A)$ we have that $B\cap \theta_{\alpha}(B)\neq \emptyset$.

\item A cycle $(\alpha,A)$ has no exits if  given any $k\in\N\cup\{0\}$ we have that $\theta_{\alpha^k\alpha_1\cdots\alpha_t}(A)\in\Breg$ with  $\Delta_{\theta_{\alpha^k\alpha_1\cdots\alpha_t}(A)}=\{\alpha_{t+1}\}$ for $t<n$ and $\theta_{\alpha^{k+1}}(A)\in\Breg$ with $\Delta_{\theta_{\alpha^{k+1}}(A)}=\{\alpha_1\}$. 
\item We say that $(\B,\Labe,\theta)$ satisfies condition $(L_\B)$ if there is no cycle without exits.
\end{enumerate}
\end{defi}

The following result justifies the above definitions in comparison with the definitions given in \cite[Definition 6.5]{KatsuraIII}.
\begin{prop}\label{cycles_top}  Let $(\B,\Labe,\theta)$ be a Boolean dynamical system, and let $E$ be the associated topological graph defined in Proposition \ref{prop_top}. Let $(\alpha,A)$ be a cycle, then $\Ne_A$ is an open subset of $E^0$ such that every point $x\in \Ne_A$ is a base for a loop. Moreover, if $(\alpha,A)$ is a cycle without exits then every point $x\in \Ne_A$ is a base for a loop without entrances. 
\end{prop}
\begin{proof}
Let $(\alpha,A)$ be a cycle, then given $k\in \N\cup\{0\}$ and $\emptyset\neq B\subseteq\theta_{\alpha^k}(A)$ we have that $B\cap\theta_{\alpha}(B)\neq \emptyset$. Without lost of generality we can suppose that $\alpha\in \Labe^1$.
We claim that $\theta_\alpha(A)=A$. Indeed, first suppose that $\emptyset \neq A\setminus\theta_\alpha(A)\subseteq A$. Then 
$$(A\setminus\theta_\alpha(A))\cap \theta_\alpha(A\setminus\theta_\alpha(A))=(A\setminus\theta_\alpha(A))\cap (\theta_\alpha(A)\setminus \theta_{\alpha^2}(A))=\emptyset\,,$$
that contradicts the hypothesis. Then $A\subseteq\theta_\alpha(A)$. No suppose that $\emptyset\neq \theta_\alpha(A)\setminus A$, then 
 $$(\theta_\alpha(A)\setminus A)\cap\theta_\alpha(\theta_\alpha(A)\setminus A)=(\theta_\alpha(A)\setminus A)\cap(\theta_{\alpha^2}(A)\setminus \theta_\alpha(A))=\emptyset\,,$$
that contradicts the hypothesis. Thus, $\theta_\alpha(A)=A$. In particular observe that $\theta_\alpha(B)=B$ for every $B\subseteq A$, so $(\theta_\alpha)_{|\Id_A}=\text{Id}$.

Then using the definition of $E$ in Proposition \ref{prop_top} it follows that every point in $\Ne_A$ is a basis for a loop. Moreover, if $(\alpha,A)$ is a cycle without exits then the loops with base point in $\Ne_A$ have no entrances.
\end{proof}

Then we can visualise condition $(L_\B)$ in terms of the groupoid.
\begin{theor}\label{theor6} The following are equivalent:
\begin{enumerate}
\item $(\B,\Labe,\theta)$ satisfies condition $(L_\B)$,
\item $\beta:T  \acts  \widehat{\mathcal{E}}_{\text{tight}}$ is topologically free,
\item $\Gti(T)$ is essentially principal.
\end{enumerate}
\end{theor}
\begin{proof} $(2)\Leftrightarrow (3)$  by Theorem \ref{theor4}.

For $(1)\Leftrightarrow (2)$  First, let $s\in T$. If it satisfies condition $(2)$ in Theorem \ref{theor5}, then for any idempotent $g\in F$ we have that $sg=g$, equivalently $g\leq s$. So, when $s=h\in \E(T)$, condition $(2)$ in Theorem \ref{theor5} always holds (just take $F=\{ e\}$). But since $T$ is a $E^*$-unitary by Proposition \ref{prop2}, no element in $T\setminus \E(T)$ can satisfy condition $(2)$ in Theorem \ref{theor5}, as $g\leq s$ imply that $s\in \E(T)$. Hence, condition $(2)$ in Theorem \ref{theor5} is equivalent to the statement: 
$$\forall s\in T \setminus\E(T)\text{ and }\forall 0\neq e\in \E(T) \text{ with }e\leq s^*s\,, \text{ there exists }0\neq f \leq e\text{ such that }sfs^*\cdot f=0\,.$$
We will separate $3$ cases:
\begin{enumerate}
\item \underline{Case $s=s_\alpha p_A$} with $A\subseteq \Ri_\alpha$: Then, let $e\leq s^*s= p_A$. Thus, without loss of generality, we can assume $e=p_A$ and $f\leq p_A$. By Lemma \ref{lem2}, $f=s_\beta p_B s^*_\beta$ with $\emptyset\neq B\subseteq \theta_\beta(A)\subseteq\Ri_{\alpha\beta}$. Without loss of generality, we can assume that $|\alpha|<|\beta|$. Then 
$$0\neq sfs^*\cdot f=s_\alpha p_A s_{\beta}p_Bs^*_{\beta}p_As^*_\alpha s_\beta p_B s^*_\beta=s_{\alpha\beta}p_Bs^*_{\alpha\beta} s_\beta p_B s^*_\beta$$
implies $\beta=\alpha\hat{\beta}$. Assuming $|\alpha|<|\hat{\beta}|$, we have that $\hat{\beta}=\alpha\beta'$ and by recurrence 
$$\beta=\alpha\beta_1=\alpha^2\beta_2=\cdots=\alpha^n\beta_n=\cdots$$

Since $\vert \beta\vert < \infty$,  $\beta$ must be $\alpha^k$ for some $k\in \N$, and thus $0\ne sfs^*\cdot f=s_{\alpha^{k+1}}p_{B\cap \theta_{\alpha}(B)}s^*_{\alpha^{k+1}}$ is equivalent to $\beta=\alpha^k$ and $B\cap \theta_{\alpha}(B)\ne \emptyset$. But this is equivalent to say that $(\alpha,A)$ is a cycle without exits. This prove the equivalence for this case.

%Since $\vert \beta\vert < \infty$,  $\beta$ must be $\alpha^k$ for some $k\in \N$, and thus $0\ne sfs^*\cdot %f=s_{\alpha^{k+1}}p_{B\cap \theta_{\alpha}(B)}s^*_{\alpha^{k+1}}$ is equivalent to $\beta=\alpha^k$ and $B\cap \theta_{\alpha}(B)\ne %\emptyset$. So, $sfs^*\cdot f=0$ for a suitable nonzero idempotent $f$ occurs exactly when one of the following two situations hold:
%\begin{enumerate}
%\item There exists $\beta\in \Labe^*$ such that $\theta_\beta(A)\neq \emptyset$, $\beta\nleq \alpha$ and $\alpha\nleq \beta$; in %particular, this is the case if $\alpha$ is a cycle with an exit. 
%\item The path $\alpha$ is a cycle without an exit, and $\beta=\alpha^k$ for some $k\in \N$. Then, $sfs^*\cdot %f=s_{\alpha^{k+1}}p_{B\cap \theta_{\alpha}(B)}s^*_{\alpha^{k+1}}=0$ if and only if there exists $\emptyset\neq B\in \B$ such that %$\theta_{\alpha}(B)\neq \emptyset $ and $B\cap \theta_{\alpha}(B)=\emptyset$.
%\end{enumerate}
%This prove the equivalence for this case.
\item \underline{Case $s=p_As^*_\alpha$}  with $A\subseteq \Ri_\alpha$: Then, let $e\leq s^*s=s_\alpha p_{A} s^*_\alpha$.
Replacing $e$ for $ses^*$ we reduce to the case $(1)$.
\item \underline{Case $s=s_\gamma p_A s^*_\alpha$} with $A\subseteq \Ri_\alpha\cap \Ri_\gamma $: Then, let $e\leq s^*s=s_\alpha p_A s^*_\alpha$. Again replacing $e$ for $ses^*$ we reduce to the case $(1)$.
\end{enumerate}
\end{proof}

%This picture provides an analog of the Cuntz-Krieger Uniqueness Theorem for labelled graph $C^*$-algebras \cite[Theorem 5.5]{BPII} in our context, as a direct consequence of Lemma \ref{lem6}, Lemma \ref{lem_nuclear} and \cite[Theorem 4.4]{N_Haus}.

%%%%%%%%%%%%%%%%%%%%%%%%%%%%%%%%%%%%%%%%%%%%%%%%%%%%%%%%
%AQUESTA ES LA NOVA VERSIO DE C-KUTH
%%%%%%%%%%%%%%%%%%%%%%%%%%%%%%%%%%%%%%%%%%%%%%%%%%%%%%%%

If $\B$ and $\mathcal{L}$ are countable, this picture allows to prove an analog of the Cuntz-Krieger Uniqueness Theorem for labelled graph $C^*$-algebras \cite[Theorem 5.5]{BPII} in our context. In order to prove such a theorem, we need to recall some facts:

\begin{rema}\label{RemCKUT} Suppose that $\B$ and $\mathcal{L}$ are countable. Then:
\begin{enumerate}
\item By \cite[Proposition 2.5]{ExPa}, the set $\{D_e : e\in \mathcal{E}(T)\}$ is a basis of $\widehat{\mathcal{E}}_{\text{tight}}(T)$ by clopen compact sets.
\item For any $s\in T$, the set $\Theta (s, D_{s^*s}):=\{[s, \eta] : \eta\in D_{s^*s}\}$ is a open bisection of $\Gti(T)$ \cite[Proposition 4.18]{Exel}. Moreover, the isomorphism $\Clab\cong C^*(\mathcal{G}_{\text{tight}}(T))$ sends each $s\in T\subset \Clab$ to the characteristic function $1_{\Theta (s, D_{s^*s})}\in C^*(\mathcal{G}_{\text{tight}}(T))$.
\item By \cite[Proposition 4.15]{Exel} and point (1) above, $\Theta (s, D_{s^*s})$ is open and compact for every $s\in T$.
\item By point (1) above and \cite[Proposition 3.8]{ExPa}, the set $\{ \Theta (s, D_{s^*s}) : s\in T\}$ is a basis of the topology of $\mathcal{G}_{\text{tight}}(T)$. In particular, since $\mathcal{G}^{(0)}_{\text{tight}}=\{[e,x]: e\in \E\}\cong \widehat{\E}_{\text{tight}}$, the set $\{ \Theta (e, D_{e}) : e\in \mathcal{E}(T)\}$ is a basis of the topology of $\mathcal{G}_{\text{tight}}(T)^{(0)}$.
\end{enumerate}
\end{rema}

Now, we are ready to prove our theorem.

\begin{theor}[{Cuntz-Krieger Uniqueness Theorem for $\Clab$}]\label{Thm:CKUTh}
Let $(\B, \Labe, \theta)$ be a Boolean dynamical system such that $\B$ and $\mathcal{L}$ are countable, satisfying condition $(L_\B)$, and let $\Clab$ be its associated $C^*$-algebra. Then, for any $\ast$-homomorphism $\pi: \Clab\rightarrow B$, the following are equivalent:
\begin{enumerate}
\item $\pi (s_{\alpha}P_As^*_{\alpha})\ne 0$ for every $\emptyset\ne A\in \mathcal{B}$ with $A\subseteq \mathcal{R}_{\alpha}$.
\item $\pi$ is injective.
\end{enumerate}
%Moreover, if $J$ is a nonzero ideal of $\Clab$, then $J\cap C_0({\Gti(T)}^0)$ is nonzero.
\end{theor}
\begin{proof}
By Lemma \ref{lem6}, Lemma \ref{lem_nuclear} and Theorem \ref{theor6}, be can apply \cite[Theorem 4.4]{N_Haus} to $C^*(\Gti(T))$. Thus, in order to conclude our result,  it is enough to prove that $\pi\vert_{C_0({\Gti(T)}^{(0)})}$ is injective if and only if $\pi (s_{\alpha}P_As^*_{\alpha})\ne 0$ for every $\emptyset\ne A\in \mathcal{B}$ with $A\subseteq \mathcal{R}_{\alpha}$.

By Remark \ref{RemCKUT}(2), if $\pi\vert_{C_0({\Gti(T)}^{(0)})}$ is injective then $\pi (s_{\alpha}P_As^*_{\alpha})\ne 0$ for every $\emptyset\ne A\in \mathcal{B}$ with $A\subseteq \mathcal{R}_{\alpha}$. 

Conversely, suppose that $\pi (s_{\alpha}P_As^*_{\alpha})\ne 0$ for every $\emptyset\ne A\in \mathcal{B}$ with $A\subseteq \mathcal{R}_{\alpha}$. If there exists $0\ne f\in C_0({\Gti(T)}^{(0)})$ such that $\pi(f)=0$, then by Remark \ref{RemCKUT}(4)
there exists $e\in \mathcal{E}(T)$ such that $\Theta (e,D_e)\subseteq \mbox{supp}(f)$, whence $\pi(e)=0$, contradicting the assumption. So we are done.\end{proof}

%%%%%%%%%%%%%%%%%%%%%%%%%%%%%%%%%%%%%%%%%%%%%%%%%%%%%%%%

%\begin{theor}[{Cuntz-Krieger Uniqueness Theorem for $\Clab$}]\label{Thm:CKUTh}
%Let $(\B, \Labe, \theta)$ be a Boolean dynamical system satisfying condition $(L_\B)$, let $\Gti(T)$ be its tight groupoid, and let $\Clab$ be its associated $C^*$-algebra. Then, for any $\ast$-homomorphism $\pi: \Clab\rightarrow B$, the following are equivalent:
%\begin{enumerate}
%\item $\pi\vert_{C_0({\Gti(T)}^{(0)})}$ is injective.
%\item $\pi$ is injective.
%\end{enumerate}
%Moreover, if $J$ is a nonzero ideal of $\Clab$, then $J\cap C_0({\Gti(T)}^0)$ is nonzero.
%\end{theor}
%
%Notice that Theorem \ref{Thm:CKUTh} requires to check the injectivity of $\pi$ on the unit space of the tight groupoid in order to guarantee the injectivity of $\pi$ on $\Clab$. In several cases (see e.g. graph $C^*$-algebras or $C^*$-algebras of self-similar graphs \cite{ExPa}), there is a homeomorphism between this space and a suitable space of infinite paths associated to the projections of the algebra, and thus the above result corresponds to the classical Cuntz-Krieger Uniqueness Theorem. In our case, this would depend of whether our unit space enjoys the same parallelism with the projections of $\Clab$. This is a quite involved question, because you need to clearly identify who are the tight filters associated to $\B$; even in the case of labelled graphs, this is a complex problem, as stated in \cite{BCM}.

Now we are going to prove that condition $(L_\B)$ is also a necessary condition to apply the Cuntz-Krieger uniqueness theorem. 
\begin{prop} Let $(\B, \Labe, \theta)$ be a Boolean dynamical system that does not satisfy condition $(L_\B)$. Then there exists a faithful representation $\{P_A,S_\alpha\}$ of $(\B, \Labe, \theta)$ that is not isomorphic to $\Clab$.
\end{prop}
\begin{proof}
Let $E$ be the associated topological graph defined in Proposition \ref{prop_top}. 
Since $(\B, \Labe, \theta)$ does not satisfy condition $(L_\B)$, there exists a cycle without exits $(\alpha,A)$. Then by Proposition  \ref{cycles_top} we have that $E$ is not a topologically free graph (see \cite[Definition 6.6]{KatsuraIII}).  Then identifying $\Clab$ with the topological graph $C^*$-algebra $\mathcal{O}(E)$ (Theorem \ref{proposition_51}) and using  \cite[Theorem 6.14]{KatsuraIII}, it follows the result.
\end{proof}

\subsection{Minimal groupoids}

In this subsection we deal with the question of minimality of the groupoid. As in the previous subsection, we refer \cite[Section 5]{ExPa} for definitions and results. We will use the following

\begin{theor}[{\cite[Theorem 5.5]{ExPa}}]\label{theor7} The following statements are equivalent:
\begin{enumerate}
\item $\beta:T  \acts  \widehat{\mathcal{E}}_{\text{tight}}$ is irreducible,
\item $\Gti(T)$ is minimal,
\item for every $0\neq e,f\in \ET$ there exists $s_1,\ldots,s_n\in S$ such that $\{s_i fs^*_i\}_{i=1}^n$ is an outer cover for $e$.
\end{enumerate} 
\end{theor}

By analogy with the case of graph $C^*$-algebras, we propose the following definition:
\begin{defi}
We say that $(\B,\Labe,\theta)$ is cofinal if for every $\emptyset\neq A\in \B$ and for every $\zeta\in 	\widehat{\mathcal{E}}_{\text{tight}}$ there exist $\alpha,\beta\in \Labe^*$ such that $s_\alpha p_{\theta_{\beta}(A)}s^*_\alpha\in \zeta$. 
\end{defi}

Recall that given $e\in \E$, we define the cylinder set of $e$ in $\widehat{\mathcal{E}}_{\text{tight}}$ as 
$$Z(e):=\{\zeta\in\widehat{\mathcal{E}}_{\text{tight}}:e\in \zeta\}\,.$$
For every $e\in \mathcal{E}$, $Z(e)$ is a compact open subset of $\widehat{\mathcal{E}}_{\text{tight}}$.\vspace{.2truecm}

Then, we have

\begin{prop}\label{prop7} The following statements are equivalent:
\begin{enumerate}
\item $(\B,\Labe,\theta)$ is cofinal.
\item $\Gti(T)$ is minimal.
\end{enumerate}
\end{prop}
\begin{proof}
First, we will prove that cofinality implies condition (3) in Theorem \ref{theor7}. For this end, suppose that $e=s_{\alpha}p_As^*_{\alpha}$and $f=s_{\beta}p_Bs^*_{\beta}$. Since $A\subseteq \mathcal{R}_{\alpha}$ we have
$$s_{\alpha}p_As^*_{\alpha}\leq s_{\alpha}p_{\mathcal{R}_{\alpha}}s^*_{\alpha}=s_{\alpha}s^*_{\alpha}\leq p_{\mathcal{D}_{\alpha}}.$$
As every cover of $p_{\mathcal{D}_{\alpha}}$ is a cover of $s_{\alpha}p_As^*_{\alpha}$, we can assume without loss of generality that $e=p_A$ for some $A\in \mathcal{B}$. Since $p_B=s^*_{\beta}fs_{\beta}$, we can assume without loss of generality that $f=p_B$ for some $B\in \mathcal{B}$.

Given $\xi\in Z(p_A)$, cofinality implies that there exist $\alpha_{\xi}, \beta_{\xi}\in \mathcal{L}^*$ such that 
$$s_{\alpha_{\xi}}p_{\theta_{\beta_{\xi}}(B)}s^*_{\alpha_{\xi}}\in \xi.$$
Hence,
$$Z(p_A)\subseteq \bigcup\limits_{\xi\in Z(p_A)}Z(s_{\alpha_{\xi}}p_{\theta_{\beta_{\xi}}(B)}s^*_{\alpha_{\xi}}).$$
Since $Z(p_A)$ is compact, there exist ${\alpha_{\xi}}_1, \dots, {\alpha_{\xi}}_n, {\beta_{\xi}}_1, \dots ,{\beta_{\xi}}_n$ such that
$$Z(p_A)\subseteq \bigcup\limits_{i=1}^nZ(s_{{\alpha_{\xi}}_i}p_{\theta_{{\beta_{\xi}}_i}(B)}s^*_{{\alpha_{\xi}}_i}).$$
By \cite[Proposition 3.7]{ExPa}, this is equivalent to say that $\{s_{{\alpha_{\xi}}_i}p_{\theta_{{\beta_{\xi}}_i}(B)}s^*_{{\alpha_{\xi}}_i}\}_{i=1}^n$ is an outer cover for $p_A$. Notice that $s_{{\alpha_{\xi}}_i}p_{\theta_{{\beta_{\xi}}_i}(B)}s^*_{{\alpha_{\xi}}_i}=(s_{{\alpha_{\xi}}_i}s^*_{{\beta_{\xi}}_i})p_B(s_{{\alpha_{\xi}}_i}s^*_{{\beta_{\xi}}_i})^*$. Thus, the result holds for $s_i:=(s_{{\alpha_{\xi}}_i}s^*_{{\beta_{\xi}}_i})$.

Now, we will prove that condition (3) in Theorem \ref{theor7} implies cofinality. For this end, take any $\emptyset \ne A\in \B$ and any $\xi\in \widehat{\E}_{\text{tight}}$. By the argument at the start of this proof, there exists $\emptyset \ne B\in \B$ such that $p_B\in \xi$. By condition (3) in Theorem \ref{theor7}, there exists $s_i:=s_{\alpha_i}p_{C_i}s^*_{\beta_i}$ for $1\leq i\leq n$ such that $\{s_ip_As^*_i\}_{i=1}^n$ is an outer cover for $p_B$. Without loss of generality, we can assume that $\theta_{\beta_i}(A)\subseteq C_i$ for every $1\leq i\leq n$, so that $s_ip_As^*_i=s_{\alpha_i}p_{\theta_{\beta_i}(A)}s^*_{\alpha_i}$ for every $1\leq i\leq n$. By multiplying by $p_B$, we conclude that $\{s_{\alpha_i}p_{(\theta_{\alpha_i}(B)\cap \theta_{\beta_i}(A))}s^*_{\alpha_i}\}_{i=1}^n$ is a finite cover for $p_B$. Since $\xi$ is tight and $p_B\in\xi$, then there exists $1\leq j\leq n$ such that $$p_B\cdot s_{\alpha_j}p_{\theta_{\beta_j}(A)}s^*_{\alpha_j}=s_{\alpha_j}p_{(\theta_{\alpha_j}(B)\cap \theta_{\beta_j}(A))}s^*_{\alpha_j}\in \xi$$ by \cite[(2.10)]{ExPa}. As $\xi$ is a filter and $p_B\cdot s_{\alpha_j}p_{\theta_{\beta_j}(A)}s^*_{\alpha_j}\leq s_{\alpha_j}p_{\theta_{\beta_j}(A)}s^*_{\alpha_j}$, we conclude that $s_{\alpha_j}p_{\theta_{\beta_j}(A)}s^*_{\alpha_j}\in \xi$, as desired.
\end{proof}

Our next goal is to give a characterization of the cofinality of $(\B,\Labe,\theta)$ in terms of the elements in $\B$ and the actions $\theta$. First we need the following definitions.

\begin{defi}\label{def_her_sat}
We say that an ideal $\Id$ of $\B$ is hereditary if given $A\in \Id$ and $\alpha\in \Labe$ then $\theta_\alpha(A)\in \Id$. We also say that $\Id$ is saturated if given $A\in \Breg$  with $\theta_{\alpha}(A)\in \Id$ for every $\alpha\in \Delta_A$ then $A\in \Id$.
\end{defi} 

%Given $n\in \N$ and $A\in \B$ we define 
%$$\Delta_A^n:=\{\alpha\in \Labe^n: \theta_\alpha(A)\neq \emptyset\}\,.$$
%Observe that if $\B=\Breg$ then $\Delta_A^n$ is a complete expansion of $A$ consisting of words of length $n$.

Given a collection $\Id$ of elements of $\B$ we define the hereditary expansion of $\Id$ as 
$$\He(\Id):=\{B\in \B: B\subseteq \bigcup\limits_{i=1}^n \theta_{\alpha_i}(A_i)\text{ where }A_i\in \Id \text{ and }\alpha_i\in \Labe^*\}\,.$$
Clearly, $\He(\Id)$ is  the minimal hereditary ideal of $\B$ containing $\Id$. Also, we define the saturation of $\Id$, denoted by $\Sa(\Id)$, to be the minimal ideal of $\B$ generated by the set
$$\bigcup\limits_{n=0}^\infty\Sa^{[n]}(\Id),$$
defined by recurrence on $n\in\Z^+$ as follows:
\begin{enumerate}
\item $\Sa^{[0]}(\Id):=\Id$
\item For every $n\in \N$, $\Sa^{[n]}(\Id):=\{B\in \Breg:\theta_\alpha(B)\in\Sa^{[n-1]}(\Id)\text{ for every }\alpha\in \Delta_B\}\,$.
\end{enumerate}  
Observe that if $\Id$ is hereditary, then $\Sa(\Id)$ is also hereditary. Therefore, given a collection $\Id$ of elements of $\B$, $\Sa(\He(\Id))$ is the minimal hereditary and saturated ideal of $\B$ containing $\Id$.\vspace{.2truecm}

%In particular, when $\B=\Breg$ and taking $\Id=\{A\}$ we have that
%$$\Sa({\He(A)})=\{C\in \B:\exists\beta_1,\ldots,\beta_m\in \Labe^*\text{ and }n\in \N\text{ such that }\theta_{\alpha}(C)\subseteq \bigcup_{i=1}^m\theta_{\beta_i}(A)\text{ for all }\alpha\in \Delta_C^n\}\,.$$

We set $\Labe^\infty:=\prod_{n=1}^\infty \Labe$. Given $\alpha\in \Labe^\infty$ and $k\in \N$, we define $\alpha_{[1,k]}=\alpha_1\cdots\alpha_k\in \Labe^k$.

\begin{theor}\label{prop8}
Let $(\B,\Labe,\theta)$ a Boolean dynamical system. Then the following statements are equivalent:
\begin{enumerate}
\item  The only hereditary and saturated ideals of $\B$ are $\emptyset$ and $\B$,
\item Given $A,B\in \B$, there exists $C\in \Breg\cup\{\emptyset\}$ such that 
\begin{enumerate}
\item $B\setminus C\in \He(A)$, and
\item For every $\alpha\in \Labe^\infty$ there exists $k\in \N$ such that $\theta_{\alpha_{[1,k]}}(C)\in \He(A)$.
\end{enumerate}
\item For every $0\neq e,f\in \ET$, there exist $s_1,\ldots,s_n\in S$ such that $\{s_i fs^*_i\}_{i=1}^n$ is an outer cover for $e$.
\item $(\B,\Labe,\theta)$ is cofinal.
\item $\Gti(T)$ is minimal.
\end{enumerate} 
\end{theor}

\begin{proof}
First observe that $(3)\Leftrightarrow (4) \Leftrightarrow (5)$ follows from Theorem \ref{theor7} and Proposition \ref{prop7}.

$(1)\Rightarrow (2)$. Suppose that the only hereditary and saturated are $\emptyset$ and $\B$. Then, given $A\neq \emptyset$ we have that  
$\Sa({\He(A)})=\B$. By definition, 
$$\He(A)=\{C\in \B:\exists\beta_1,\ldots,\beta_m\in \Labe^*\text{ and }n\in \N\text{ such that }C\subseteq \bigcup\limits_{i=1}^m\theta_{\beta_i}(A)\}\,.$$
Since  $\Sa({\He(A)})=\B$, by definition of saturation we have that $\B=\{C\cup D: C\in \He(A) \text{ and } D\in \Breg\}$. Thus, given any $B\in \B$, there exists $D\in \He(A)$ such that $C:=B\setminus D\in \Breg$, and there exists $n\in \N$ such that $C\in \Sa^{[n]}(\He(A))$. Therefore, for every $\alpha\in \Labe^\infty$, we have that $\theta_{\alpha_{[1,n]}}(C)\in \He(A)$. 

$(2)\Rightarrow (3)$.  Without loss of generality, we can assume that $f=p_A$ and $e=p_B$ for some $\emptyset\neq A,B\in \B$.  By hypothesis, there exists $C\in \Breg\cup\{\emptyset\}$ such that $B\setminus C\in \He(A)$. So, there exist $\beta_1,\ldots,\beta_m\in \Labe^*$ such that $B\setminus C\subseteq \bigcup\limits_{i=1}^m\theta_{\beta_i}(A)$. Thus, if we define $s_i:=s^*_{\beta_i}$ for $1\leq i\leq m$, then $s_ifs^*_i=p_{\theta_{\beta_i}(A)}$. Hence, since $\bigvee_{i=1}^m p_{\theta_{\beta_i}(A)}=p_{ \bigcup\limits_{i=1}^m\theta_{\beta_i}(A)}$, we can reduce the proof to the case that $e=p_C$. Now, if $\theta_{\gamma}(C)\in \He(A)$ for every $\gamma\in \Delta_C$ and $C\in \Breg$, we have that $C\in\Sa^{[1]}(\He(A))$, whence we can find a finite cover for $p_C$. Otherwise, there exists $\gamma_1\in \Delta_C$ such that $\theta_{\gamma_1}(C)\notin \He(A)$. Now, we repeat the argument  to find a finite cover for $p_{\theta_{\gamma_1}(C)}$. By recurrence, we either construct a finite path $\gamma=\gamma_1\cdots\gamma_m$ such that $\theta_\gamma(C)\in \He(A)$, or we construct an infinite path $\alpha\in \Labe^\infty$ such that $\alpha_{[1,k]}(C)\notin\He(A)$ for every $k\in \N$. In the first case we obtain a finite cover for $p_C$. In the second case we get an infinite path, contradicting the hypothesis. So we are done.

$(3)\Rightarrow (1)$.  Let $\emptyset\neq A\in \B$. We want to prove that $\Sa(\He(A))= \B$. If we take $\emptyset\neq B\in \B$ then, by hypothesis, there exist $s_1,\ldots,s_n$ such that $\{s_ifs_i^*\}_{i=1}^n=\{s_{\alpha_i} p_{\theta_{\beta_i}(A)} s^*_{\alpha_i}\}_{i=1}^n$ is an outer cover for $p_B$. So, 
$$p_B\leq \bigvee_{i=1}^ns_{\alpha_i} p_{\theta_{\beta_i}(A)} s^*_{\alpha_i}\,.$$
We set $N_1:=\text{max }\{|\alpha_i|:i=1,\ldots,n\}$. Since only regular sets can have finite covers, it must exists $C\in \Breg$ such that 
$$B\setminus C\subseteq \bigcup\limits_{\alpha_i=\emptyset}\theta_{\beta_i}(A)\in \He(A)\,.$$
So we have that 
$$p_C\leq \bigvee_{i=1,\alpha_i\neq\emptyset}^ns_{\alpha_i} p_{\theta_{\beta_i}(A)} s^*_{\alpha_i}\,,$$
and $C\in \Breg$. Thus, we can assume that $B\in \Breg$ and $\alpha_i\neq\emptyset$ with  
$$p_B\leq \bigvee_{i=1}^ns_{\alpha_i} p_{\theta_{\beta_i}(A)} s^*_{\alpha_i}\,.$$

Now, we label $\Delta_B=\{\gamma_1,\ldots,\gamma_m\}$, and relabel $\{\alpha_i\}$ so that there exist $0=j_0< j_1<j_2<\cdots< j_m= n$ with $\gamma_k\leq \alpha_i$ for every $j_{i-1}< k\leq j_i$ and $\gamma_k\nleq \alpha_i$ otherwise. Then,  we have that 
$$s_{\gamma_i}p_{\theta_{\gamma_i}(B)}s^*_{\gamma_i}\leq \bigvee_{k=j_{i-1}+1}^{j_i}s_{\alpha_k} p_{\theta_{\beta_k}(A)}  s^*_{\alpha_k}\qquad\text{for every }i=1,\ldots,m\,,$$
or equivalently
$$p_{\theta_{\gamma_i}(B)}\leq \bigvee_{k=j_{i-1}+1}^{j_i}s_{\alpha_k\setminus \gamma_i} p_{\theta_{\beta_k}(A)}  s^*_{\alpha_k\setminus \gamma_i}\qquad\text{for every }i=1,\ldots,m\,.$$
Observe that we have $|\alpha_k\setminus \gamma_i|<|\alpha_k|$. Thus, we can assume that
$$p_{\theta_{\gamma}(B)}\leq \bigvee_{i=1}^{n}s_{\alpha_i} p_{\theta_{\beta_i}(A)}  s^*_{\alpha_i}\qquad\text{for every }\gamma\in \Delta_B$$
with $N_2:=\text{max }\{|\alpha_i|:i=1,\ldots,n\}=N_1-1<N_1$. By hypothesis, we can also assume that $\theta_{\gamma}(A)\in \Breg$ for every $\gamma\in \Delta_B$.

Therefore, after repeating this process $N_1$ times, we prove that $p_{\theta_\gamma(B)}\in \Breg$ for every $\gamma\in \Delta^{\leq N_1-1}_B$, and ${\theta_\gamma(B)}\in \He(A)$ for every $\gamma\in \Delta^{N_1}_B$. Thus, $B\in \Sa^{[N_1]}(\He(A))$, and hence $B\in \Sa(\He(A))$.

\end{proof}

\subsection{The simplicity result}

Now, we are ready to state a result, giving a characterization of simplicity for $\Clab$ in terms of properties enjoyed by $(\B, \Labe, \theta)$.

\begin{theor}\label{theor_simple}
Let $(\B, \Labe, \theta)$ be a Boolean dynamical system such that $\B$ and $\mathcal{L}$ are countable, and let $\Clab$ be its associated $C^*$-algebra. Then, the following statements are equivalent:
\begin{enumerate}
\item $\Clab$ is simple.
\item The following properties hold:
\begin{enumerate}
\item $(\B,\Labe,\theta)$ satisfies condition $(L_\B)$, and
\item The only hereditary and saturated ideals of $\B$ are $\emptyset$ and $\B$.
\end{enumerate}
\end{enumerate}
\end{theor}
\begin{proof}
By Theorem \ref{theor3}, $\Clab\cong C^*(\Gti (T))$. By Lemma \ref{lem6} and Lemma \ref{lem_nuclear}, $\Gti (T)$ is Hausdorff and amenable. Then, the result holds by Theorem \ref{theor6}, Theorem \ref{prop8} and Theorem \ref{theor3}.
\end{proof}

Theorem \ref{theor_simple} generalizes \cite[Theorem 6.4]{BPII} (where only sufficient conditions are given) and \cite[Theorem 3.8, 3.14 \& 3.16]{JeKi} (which provided an equivalence, and solved a problem in Bates and Pask's result) in our context, the point being the use of a completely different approach to fix the conditions equivalent to simplicity, that are stated in terms of both the groupoid properties and the Boolean dynamical system.

\section{Gauge invariant ideals}
Now,  using the characterization of the Cuntz-Krieger Boolean $C^*$-algebras as topological graph $C^*$-algebras explained in Section \ref{section_4}, we will use the work of Katsura \cite{KatsuraIII} to determine the gauge invariant ideals of the Cuntz-Krieger Boolean $C^*$-algebras. We will restrict for simplicity, to the class of locally finite Boolean dynamical systems (see definition \ref{loc_finite}).

Given a Boolean dynamical system $(\B,\Labe,\theta)$, we will denote by $E_{(\B,\Labe,\theta)}$ the associated topological graph defined in Proposition \ref{prop_top}. If there is no confusion, we will just write $E$.

\begin{defi} Let $E=(E^0,E^1,d,r)$ be a topological graph. A subset $X^0$ of $E^0$ is said to be \emph{positively invariant} if $d(e)\in X^0$ implies $r(e)\in X^0$ for each $e\in E^1$, and to be \emph{negatively invariant} if for every $v\in X^0\cap E^0_{rg}$ there exists $e\in E^1$ with $r(e)=v$ and $d(e)\in X^0$. A subset $X^0$ of $E^0$ is called \emph{invariant} if $X^0$ is both positively and negatively invariant.  
\end{defi}

%We define the \emph{singular vertices} as $E^0_{sg}=E^0\setminus E^0_{rg}$.
%$$\E^0_{sg}=\E^0\setminus\E^0_{rg}=\{v_{\Cu}:\text{ min }\{\lambda_A:A\in\Cu\}\in\{0,\infty\}\}\,.$$
%$$E^0_{sg}=E^0\setminus E^0_{rg}=\{v_{\Cu}:\forall A\in \Cu \text{  }\exists B\in \B\text{ with }B\subseteq A \text{ with }\lambda_B\in\{0,\infty\}\}\,.$$

\begin{defi} Let $E=(E^0,E^1,d,r)$ be a topological graph. A subset $Y$ of $E^0$ is said to be \emph{hereditary} if $r(e)\in Y$ implies $d(e)\in Y$, and \emph{saturated} if $v\in E^0_{rg}$ with $d(r^{-1}(v))\subseteq Y$ implies $v\in Y$.
\end{defi}

Observe that a subset $X^0$ of $E^0$ is positively invariant if and only if $E^0\setminus X^0$ is hereditary, and it is negatively invariant if and only if $E^0\setminus X^0$ is saturated.
 
\begin{lem}\label{lemma_11} Let $(\B,\Labe,\theta)$ be a Boolean dynamical system, and let $E=E_{(\B,\Labe,\theta)}$ be the associated topological graph. If $\He$ is an ideal of $\B$, then $\He$ is  hereditary (definition \ref{def_her_sat}) if and only if $Y:=\bigcup\limits_{A\in \He} \Ne_A$ is a hereditary subset of $E^0$. 
\end{lem}
\begin{proof} Suppose that $\He$ is a hereditary ideal of $\B$. Let $v_\Cu\in Y$, so there exists $A\in \He$ such that $v_\Cu\in \Ne_A$, and suppose that there exists $\alpha\in \Labe$ such that $v_{\Cu}\in r(E^1_\alpha)$. Let $\Cu'\in \widehat{\Id_{\Ra}}$ such that $r(e^\alpha_{\Cu'})=v_\Cu$, so that $\Cu=\widehat{\theta_\alpha}(\Cu')=\{B\in \Id_{\Do}:\theta_\alpha(B)\in \Cu'\}$. Since $A\in \Cu$, we have that $\theta_\alpha(A)\in \Cu'$, so $v_{\Cu'}\in \Ne_{\theta_\alpha(A)}$. As $A\in \He$, by hypothesis $\theta_\alpha(A)\in \He$, and therefore $v_{\Cu'}\in \Ne_{\theta_\alpha(A)}\subseteq Y$. Thus, $d(e^{\alpha}_{\Cu})=v_{\Cu'}\in Y$, as desired.

Conversely, suppose that $Y:=\bigcup\limits_{A\in \He} \Ne_A$ is a hereditary subset of $E^0$, and suppose that there exists $A\in \He$ such that $\theta_\alpha(A)\notin \He$. We claim that there exists an ultrafilter $\Cu$ of $\B$ such that $A\in \Cu$ and $\theta_\alpha(B)\notin \He$ for every $B\in \Cu$. Indeed, let us consider  the set $\Gamma$ of all the filters $\Cu$ of $\B$ such that $A\in \Cu$ and $\theta_\alpha(B)\notin \He$ for every $B\in \Cu$. $\Gamma$  is a partially ordered set with the inclusion.  

 First observe that $\Gamma\neq \emptyset$, because the minimal filter containing $A$ belongs to $\Gamma$. Now, let $\{\Cu_n\}_{n\in \N}$ be an ascending sequence of filters of $\Gamma$. $\Cu=\bigcup\limits_{n\in \N}\Cu_n$ is clearly a filter from $\Gamma$ with $\Cu_n\subseteq \Cu$ for every $n\in \N$. Then, by  Zorn's Lemma, there exist maximal elements in $\Gamma$. If $\Cu$ is a maximal element of $\Gamma$, we claim that $\Cu$ is an ultrafilter of $\B$. Indeed, we only have to check condition $\textbf{F3}$. Let $B\in \Cu$, and let $C,C'\in \B\setminus\{\emptyset\}$ such that $B=C\cup C'$ and $C\cap C'=\emptyset$. Suppose that $C,C'\notin \Cu$. Then, $C\cap D, C'\cap D\neq \emptyset$ for every $D\in \Cu$; otherwise, if there exists $D\in \Cu$ such that $C\cap D=\emptyset$, then  $\Cu\ni(B\cap D)\subseteq C'$ by condition $\textbf{F2}$. Thus, $C'\in \Cu$ by condition $\textbf{F1}$, a contradiction, whence $C\cap D\neq \emptyset$ for every $D\in \Cu$. By the same argument $C'\cap D\neq \emptyset$ for every $D\in \Cu$. Now, suppose that there exists $D\in \Cu$ such that $\theta_\alpha(C\cap D)\in \He$. Then, for every $D'\in \Cu$ with $D'\subseteq D$,  we have that $\theta_\alpha(C\cap D')\subseteq \theta_\alpha(C\cap D)$. So, $\theta_\alpha(C\cap D')\in \He$ too, since $\He$ is an ideal. Now, suppose that $\theta_\alpha(C\cap G)\in \He$ for some $G\in \Cu$. By the same argument as above, $\theta_\alpha(C\cap G')\in \He$ for every $G'\in \Cu$ with $G'\subseteq G$. Thus, $B\cap D\cap G\in \Cu$ and 
$$\theta_\alpha(B\cap D\cap G\cap C)\cup \theta_\alpha(B\cap D\cap G\cap C')=\theta_\alpha(B\cap D\cap G)\notin \He\,.$$
But by the above arguments, we have that $\theta_\alpha(B\cap D\cap G)\in \He$ because $\He$ is an ideal, a contradiction. Therefore, we can assume that $\theta_\alpha(C\cap D)\notin \He$ for every $D\in \Cu$. Now, we construct the filter $\Cu'=\{B\in \B: C\cap D\subseteq B \text{ for some }D\in \Cu\}$. We clearly have that $\Cu'\in \Gamma$ with $\Cu\subsetneq \Cu'$, contradicting with the maximality of $\Cu$. Thus, $\Cu$ is an ultrafilter of $\B$, as desired.
 
Now, we claim that there exists an ultrafilter  $\Cu'$ of $\B$ such that $\theta_\alpha(B)\in \Cu'$ for every $B\in \Cu$ and $C\notin \He$ for every $C\in \Cu'$, where $\Cu$ is the ultrafilter constructed above. Let $\Gamma'$ be the set of all filters of $\B$ satisfying the above requirements. We have that $\Gamma'\neq \emptyset$ since the filter $\D=\{C:\in \B: \theta_\alpha(B)\subseteq C \text{ for some }B\in \Cu\}$ belongs to $\Gamma'$. Also,  $\Gamma'$ is a partially ordered set with the inclusion, and clearly every ascending sequence of filters of $\Gamma'$ has an upper-bound. By the Zorn's Lemma, $\Gamma'$ has maximal elements. Let $\Cu'$ be a maximal element. We claim that $\Cu'$ is an ultrafilter of $\B$. Indeed, we only have to check condition $\textbf{F3}$. Let $C\in \Cu'$ and let $D,D'\in \B\setminus\{\emptyset\}$ with $C=D\cap D'$ and $D\cap D'=\emptyset$ and $D,D'\notin \Cu'$. We have that $D\cap G,D'\cap G\neq \emptyset$ for every $G\in \Cu'$; otherwise, if there exists $G\in \Cu'$ such that $D\cap G=\emptyset$, then we have that $(C\cap G)\subseteq D'$. So, $D'\in \Cu$ by condition $\textbf{F1}$, a contradiction. Thus, $D\cap G\neq \emptyset$ for every $G\in \Cu'$. By the same argument we have that  $D'\cap G\neq \emptyset$ for every $G\in \Cu'$. Finally suppose that there exists $G,G'\in \Cu'$ such that $D\cap G, D'\cap G'\in \He$. Then, 
$$(C\cap G\cap G'\cap D)\cup (C\cap G\cap G'\cap D')=C\cap G\cap G'\notin \He\,,$$
but since $\He$ is an ideal, we have that $C\cap G\cap G'\in \He$, a contradiction. Therefore, suppose that $\emptyset\neq D\cap G\notin \He$ for every $G\in \Cu'$. Then, we can define the filter $\Cu''=\{C\in \B:D\cap G\subseteq C\text{ for some }G\in \Cu'\}$. We have that $\Cu''\in \Gamma'$ and $\Cu'\subsetneq \Cu''$, contradicting the maximality of $\Cu'$. Thus, $\Cu'$ is an ultrafilter, as desired.

Finally, since $\Cu'\in \widehat{\Id_{\Ra}}$, we can define  $e^\alpha_{\Cu'}\in E^1_\alpha$. But $v_{\Cu'}\notin Y$, since $B\notin \He$ for every $B\in \Cu'$. Observe that by Lemma \ref{lemma_81} we have that $\widehat{\theta_\alpha}(\Cu')=\Cu$, whence $r(e^\alpha_{\Cu'})=v_\Cu$.  Moreover, $v_{\Cu}\in \Ne_A\subseteq Y$, since $A\in \Cu$. But this contradicts that $Y$ is a hereditary set of $E^0$. Thus, $\theta_\alpha(A)\in \He$, as desired, whence $\He$ is a hereditary ideal of $\B$.  
\end{proof}

Observe that, if $A\in \Breg$, then given any $\Cu\in \widehat{\Id_A}$ we have that $v_\Cu\in E^0_{rg}$.

\begin{lem}\label{lemma_12} Let $(\B,\Labe,\theta)$ be a  Boolean dynamical system, and let $E=E_{(\B,\Labe,\theta)}$  be the associated topological graph. If $\He$ is an ideal of $\B$, then $\He$ is  saturated (definition \ref{def_her_sat}) if and only if $Y:=\bigcup\limits_{A\in \He} \Ne_A$ is a saturated subset of $E^0$. 
\end{lem}
\begin{proof} First,  suppose that $\He$ is a saturated subset of $\B$, and let $\Cu\in \widehat{\B}$ such that $v_\Cu\in E^0_{rg}$. Recall that 
$$r^{-1}(v_\Cu)=\{e^\alpha_{\Cu'}:\Cu'\in \widehat{\B} \text{ such that }\exists \alpha\in \Labe \text{ with }\Cu=\{A\in \B:\theta_\alpha(A)\in \Cu'\}\}\,.$$
Suppose that $d(e^\alpha_{\Cu'})=v_{\Cu'}\in Y$ for every $e^\alpha_{\Cu'}\in r^{-1}(v_\Cu)$. Hence, there exists $B_{\Cu'}\in \Cu'$ such that $B_{\Cu'}\in \He$. We claim that, for every $\alpha\in \Labe$  such that $\theta_\alpha(A)\neq \emptyset$ for every $A\in \Cu$, there exists $A\in \Cu$ such that $\theta_\alpha(A)\in \He$. Indeed, suppose that there exists $\alpha\in \Labe$ such that $\theta_\alpha(A)\notin \He$ for every $A\in \Cu$. Let $\Gamma$ the set of all filters $\F$ of $\B$ such that $\theta_\alpha(A)\in \F$ and $\theta_\alpha(A)\notin \He$ for every $A\in \Cu$. Then, $\F=\{B\in \B:\theta_\alpha(A)\subseteq B\text{ for some }A\in \Cu\}$ is a filter in $\Gamma$, whence $\Gamma\neq \emptyset$. We have that $\Gamma$ is a partially ordered set with the inclusion, and it is clear that $\Gamma$ contains an upper-bound for every ascending chain. Therefore, by the Zorn's Lemma, $\Gamma$ has maximal elements. Given any maximal element $\Cu'\in \Gamma$, we have that $\Cu'$ is an ultrafilter. Therefore, we have that $\Cu'\notin \widehat{\Id_B}$ for every $B\in \He$, and hence $v_{\Cu'}\notin Y$. Moreover, by Lemma \ref{lemma_81} we have that $r(e^\alpha_{\Cu'})=v_{\Cu}$. But this contradicts the hypothesis that $d(r^{-1}(v_\Cu))\subseteq Y$. Thus, there exists $A\in \Cu$ such that $\theta_\alpha(A)\in \He$. Then, given any $\alpha\in \Labe$ such that $\theta_\alpha(A)\neq \emptyset$ for every $A\in \Cu$, there exists $A_\alpha\in \Cu$ such that $\theta_\alpha(A_\alpha)\in \He$.

Now, since $v_{\Cu}\in E^0_{rg}$, there exists $A\in \Cu$ such that $\lambda_A<\infty$, and given any $B\in \B$ with $B\subseteq A$ then $\lambda_B\neq 0$. So, $A$ is a regular set of $\B$. If replace $A$ by $A\cap \left(\bigcap\limits_{\alpha\in \Delta_A} A_\alpha\right)\in \Cu$, we can suppose that $\theta_\alpha(A)\in \He$ for every $\alpha\in \Delta_A$. Then, since $\He$ is saturated, we have that $A\in \He$, and hence $v_\Cu\in \Ne_A\subseteq Y$. Thus, $Y$ is a saturated subset of $E^0$.

Conversely, suppose that $Y$ is a saturated subset of $E^0$, and let $\He$ be an ideal of $\B$. Let $A\in \He$ and regular such that $\{\theta_\alpha(A):\alpha\in \Labe\}\subseteq \He$. We claim that for every ultrafilter $\Cu\in \widehat{\Id_A}$ there exists $B_\Cu\in \He$ with $\Cu\in \widehat{\Id_{B_\Cu}}$. Indeed, since $A$ is regular, we have that $v_\Cu\in E^0_{rg}$. Moreover, since  $\{\theta_\alpha(A):\alpha\in \Labe\}\subseteq \He$,  we have that $d(r^{-1}(v_\Cu))\subseteq Y$. Therefore, since $Y$ is saturated, it follows that $v_\Cu\in Y$, so $B_\Cu\in \Cu$ for some $B_\Cu\in \He$, as desired. 

Let $\Cu\in \widehat{\Id_A}$. By the above claim, there exists $B_\Cu\in \He$ with $\Cu\in \widehat{\Id_{B_\Cu}}$, and then $A\cap B_\Cu\in \Cu\cap \He$ and  $\Ne_{A}\cap\Ne_{B_\Cu}=\Ne_{A\cap B_\Cu}$.  Therefore, $\Ne_{A}=\bigcup\limits_{\Cu\in \widehat{\Id_A}}\Ne_{A\cap B_\Cu}$. But since  $\widehat{\Id_A}$ is compact by Corollary \ref{corollary_1}, we have that $\Ne_{A}=\Ne_{A\cap B_{\Cu_1}}\cup\cdots\cup \Ne_{A\cap B_{\Cu_n}}$ for some $n\in \N$. Hence, it is easy to check that  $A=\bigcup\limits_{i=1}^n(A\cap B_{\Cu_i})$. As $A\cap B_{\Cu_i}\in \He$ for every $i=1,\ldots,n$, and $\He$ is an ideal, it follows that $A\in \He$, as desired.
\end{proof}

We have proved in the previous lemmas that, given a hereditary and saturated ideal $\He$ of $\B$, then $Y=\bigcup\limits_{A\in \He} \Ne_A$ is a hereditary and saturated subset of $E^0$. The converse is also true. Indeed, let $Y$ be a hereditary and saturated subset of $E^0$. Given $v\in Y$, pick $A_v\in \B$ such that $v\in \Ne_{A_v}$ and $\Ne_{A_v}\subseteq Y$. We define $\He$ to be the minimum ideal of $\B$ containing the $A_v$'s. Observe that since every $\Ne_{A_v}$ is compact by Corollary \ref{corollary_1}, and since $\He$ is an ideal, $\He$ is independent of the choice of the $A_v$'s. Now, following the proof of Lemmas \ref{lemma_11} \& \ref{lemma_12}, one can check that $\He$ is a hereditary and saturated ideal of $\B$. Thus, the following results follows:

\begin{prop}\label{proposition_6} Let $(\B,\Labe,\theta)$ be a  Boolean dynamical system, and let $E=E_{(\B,\Labe,\theta)}$  be the associated topological graph. Then, there is a bijection between the hereditary and saturated subsets of $\B$ and the invariant subsets of $E$.
\end{prop}

\begin{exem}\label{example_nonsimple} Let $(\B,\Labe,\theta)$ be the Boolean dynamical system of Example \ref{example_4}. Then, the only hereditary and saturated subset of $\B$ is the set $\He=\{F: F\subseteq \Z\text{ finite}\}$, the associated open hereditary and saturated subspace $Y=\bigcup\limits_{A\in \He}\Ne_A$ of $E^0$ is $\{v_{\Cu_n}:n\in \Z\}$, and let $X=E^0\setminus Y=\{\Cu_\infty\}$ is the associated invariant space.
\end{exem}

\begin{prop}\label{proposition_7} Let $(\B,\Labe,\theta)$ be the Boolean dynamical system, and let $\He$ be a hereditary ideal of $\B$. If for any $\alpha\in \Labe$ and any $[A]\in \B/\He$ we define $\theta_\alpha([A])=[\theta_\alpha(A)]$, then $(\B/\He,\Labe,\theta)$ is a Boolean dynamical system.
\end{prop}
\begin{proof} We only need to prove that, given $\alpha\in \Labe$, the map $\theta_\alpha:\B/\He\longrightarrow \B/\He$ is a well-defined map. But this clear because $\He$ is a hereditary ideal of $\B$. Also, the range and domain of $\theta_\alpha$ are $[\Ra]$ and $[\D_\alpha]$ respectively. \end{proof}

Let $X^0$ be an invariant space of $E^0$. If we define $X^1=\{e\in E^1: d(e)\in X^0\}$, then $(X^0,X^1,d,r)$ is also a topological graph. 

\begin{prop}\label{proposition_8}  Let $(\B,\Labe,\theta)$ be a  Boolean dynamical system, and let $E=E_{(\B,\Labe,\theta)}$  be the associated topological graph. Given a hereditary and saturated ideal $\He$ of $\B$, define  $X^0:=E^0\setminus \bigcup\limits_{A\in \He} \Ne_A$. Then, $E_{\He}:=E_{(\B/\He,\Labe,\theta)}=(X^0,X^1,d,r)$.
\end{prop}
\begin{proof} 
Since $E^0=\widehat{\B}$ and $\bigcup\limits_{A\in \He} \Ne_A=\widehat{\He}$, using Remark  \ref{rema_sub} we can identify $X^0$ with $ \widehat{\B/\He}=E^0_\He$ by $v_\Cu\mapsto v_{[\Cu]}$. By definition, $X^1=\bigsqcup\limits_{\alpha\in \Labe}\{e^\alpha_\Cu: \Cu\in \widehat{\Id_{\Ra}}\,\text{ and }\,[\Cu]\in \widehat{\B/\He}\}$. So, we can identify it with $E^1_\He=\bigsqcup\limits_{\alpha\in \Labe}\widehat{\Id_{[\Ra]}}$ by $e^\alpha_\Cu\mapsto e^\alpha_{[\Cu]}$. With these identifications, it is clear that the maps $d$ and $r$ are the corresponding ones.
\end{proof}

A topological graph $E=(E^0,E^1,d,r)$ is called row-finite if $r(E^1)=E^0_{rg}$.

\begin{lem}
Let $(\B,\Labe,\theta)$ be a locally finite Boolean dynamical system, then the associated topological graph $E$ is row-finite.
\end{lem}
\begin{proof} Recall that 
$$r(E^1)=\{v_\Cu\in E^0| \exists \alpha\in\Labe\,, \theta_\alpha(A)\neq \emptyset\text{ }\forall A\in\Cu\}$$
and
$$E^0_{rg}=\{v_\Cu\in E^0| \exists A\in\Cu\,, \text{ }\forall B\subseteq A\text{ we have that }0<\lambda_B<\infty \}\,.$$
The inclusion $E^0_{rg}\subseteq r(E^1)$ is always valid, and the converse is obvious by locally finiteness of the Boolean dynamical system.
\end{proof}

Given a Boolean dynamical system $(\B,\Labe,\theta)$ and a hereditary and saturated set $\He$ of $\B$, we define $I_\He$ as the ideal of $C^*(\B,\Labe,\theta)$ generated by the projections $\{p_A:A\in \He\}$.

Conversely, given an ideal $I$ of $C^*(\B,\Labe,\theta)$ let us define $\rho_I:C^*(\B,\Labe,\theta)\longrightarrow C^*(\B,\Labe,\theta)/I$ to be the quotient map, and $\He_I:=\{A\in \B: \rho_I(p_A)=0\}$. Clearly $\He_I$ is a hereditary and saturated set of $\B$.

Then using Proposition \ref{proposition_8} it follows: 
	
\begin{prop}[{cf. \cite[Proposition 3.15]{KatsuraIII}}] Let $(\B,\Labe,\theta)$ be a  Boolean dynamical system. If $I$ is an ideal of $C^*(\B,\Labe,\theta)$, then there exists a natural surjection $C^*(\B/\He_I,\Labe,\theta)\rightarrow C^*(\B,\Labe,\theta)/I$ which is injective in $C^*(\B/\He_I)$.
\end{prop}

\begin{prop}[{cf. \cite[Proposition 3.16]{KatsuraIII}}] Let $(\B,\Labe,\theta)$ be a locally finite Boolean dynamical system. For an ideal $I$ of $C^*(\B,\Labe,\theta)$, the following statements are equivalent:
\begin{enumerate}
\item $I$ is a gauge-invariant ideal,
\item The natural surjection $C^*(\B/\He_I,\Labe,\theta)\rightarrow C^*(\B,\Labe,\theta)/I$ is an isomorphism,
\item $I=I_{\He_I}$.
\end{enumerate}
\end{prop}

\begin{theor}[{cf. \cite[Corollary 3.25]{KatsuraIII}}]  Let $(\B,\Labe,\theta)$ be a locally Boolean dynamical system and let $E$ the associated topological graph. Then the maps $I\rightarrow \He_I$ and $\He\rightarrow I_\He$ define a one-to-one correspondence between the set of all gauge invariant ideals of $C^*(\B,\Labe,\theta)$ and the set of all hereditary and saturated sets of $(\B,\Labe,\theta)$. \end{theor}

\begin{exem} Let $(\B,\Labe,\theta)$  be the Boolean dynamical system from Example \ref{example_4}. By Example \ref{example_nonsimple} there exists only one non-trivial hereditary and saturated subset $\He$. Then, the only gauge invariant ideal of $C^*(\B,\Labe,\theta)$ is the ideal $I_\He$ generated by the projections $\{p_F:F\subseteq \Z\text{ finite}\}$. Then the quotients $C^*(\B,\Labe,\theta)/I_\He$ is isomorphic to $C^*(\B/\He,\Labe,\theta)$. Observe that  $\B/\He$ has only one non-empty element $[\Z]$, and $\theta_a([\Z])=[\emptyset]$ and $\theta_b([\Z])=\theta_c([\Z])=[\Z]$, thus $C^*(\B/\He,\Labe,\theta)$ is isomorphic to the Cuntz algebra $\mathcal{O}_2$.
\end{exem}

\section{Examples} 

Our motivation to define the Boolean Cuntz-Krieger algebras was to study the labelled graph $C^*$-algebras from a more general point of view.
Therefore, our first example will be how, given a Labelled graph, to construct a Boolean dynamical system.  Besides of that, as we showed that the Boolean Cuntz-Krieger algebras are $0$-dimensional topological graphs, the $C^*$-algebras that we can construct as Boolean Cuntz-Krieger algebras includes homeomorphism $C^*$-algebras over $0$-dimensional compact spaces, and graph $C^*$-algebras, among others \cite{KatsuraI}. Finally, we present the $C^*$-algebras associated to one-sided subshifts as Cuntz-Krieger Boolean algebras, and apply our result about simplicity.

\begin{exem}(Weakly left-resolving labelled graphs)
First, we refer the reader to \cite{BPI,BPII,BPIII} for the basic definitions and terminology about labelled graphs $C^*$-algebras. Let $(E,\Labe,\B)$ be a labelled graph, where $E$ is a directed graph, $\Labe:E^1\rightarrow \mathcal{A}$ is a labelling map over an alphabet $\mathcal{A}$, and $\B$ is an accommodating set of vertices $E^0$ \cite[Section 2]{BPIII} that contains $\{\{v\}:v\in E^0_{sink}\}$. We will suppose that $(E,\Labe,\B)$ is weakly left-resolving and that $\B$ is a Boolean algebra. 

 Then, given $A,B\in \B$ and $\alpha\in\Labe(E^1)$, we have that $r(A\cup B,\alpha)=r(A,\alpha)\cup r( B,\alpha)$ by definition, and $r(A\cap B,\alpha)=r(A,\alpha)\cap r(B,\alpha)$ since $(E,\Labe,\B)$ is weakly left-resolving. We claim that $r(A\setminus B,\alpha)=r(A,\alpha)\setminus r(B,\alpha)$. Indeed, observe that 
$$r(A\setminus B,\alpha)\cap r(B,\alpha)=r((A\setminus B)\cap B,\alpha)=r(\emptyset,\alpha)=\emptyset$$
 since $(E,\Labe,\B)$ is weakly left-resolving and $A\setminus B\in\B$. Thus, since 
$$r(A\setminus B,\alpha)\cup r(B,\alpha)=r(A\cup B,\alpha)=r(A,\alpha)\cup r(B,\alpha)=(r(A,\alpha)\setminus r(B,\alpha))\cup r(B,\alpha)\,,$$
it follows that $r(A\setminus B,\alpha)=r(A,\alpha)\setminus r(B,\alpha)$, as desired.

Since $\B$ is an accommodating set we have that $r(\alpha)\in\B$ for every $\alpha\in\Labe$, that we will call $\Ra=r(\alpha)$. Moreover, we will assume that there exists $\Do\in\B$ such that $r(\Do,\alpha)=r(\alpha)$ for every $\alpha\in\Labe$.

Then the  triple $(\B,\Labe(E^1),\theta)$,  where $\theta_\alpha:=r(-,\alpha)$ for every $\alpha\in\Labe$, is a Boolean dynamical system, and  $C^*(E,\Labe,\B)\cong C^*(\B,\Labe(E^1),\theta)$.

Now we are going assume that the graph $E$ has no sinks or sources, and that the labelled graph $(E,\Labe)$  is receiver set-finite, set-finite
and weakly left-resolving (see \cite{BPII}).   Let $\B$ be the Boolean algebra generated by $\{\Ra:\alpha\in\Labe^*\}$, and given $l\in\N$ we define $\Omega_l:=\{x_F\in\B: F\subseteq \Labe^{\leq l}\}$ where 
$$x_F:=\bigcap_{\alpha\in F} \Ra \setminus\left( \bigcup_{\beta\in \Labe^{\leq l}\setminus F }\mathcal{R}_\beta \right)\,,$$
 that is well-defined because \cite[Lemma 2.3]{BPII}. We set $\Omega:=\bigcup_{l=1}^\infty \Omega_l$. Given any $F\subseteq \Labe^{\leq l}$ we have that $x_F\in\B$, and that for every $A\in\B$ there exist $k\in \N$  and $x_1,\ldots, x_n\in\Omega_k$ such that $A=\bigcup_{i=1}^n x_i$ (c.f. \cite[Proposition 2.4]{BPII}). Thus, the Boolean algebra generated by $\Omega$ is $\B$.

Observe that given $F\subseteq \Labe^{\leq l}$ 
$$x_F=\{v\in E^0: v\in r(\alpha) \text{ for every  }\alpha\in F,\text{ but }v\notin r(\beta)\text{ if }\beta\notin F\}.$$
Two vertices $v,w\in x_F$ are called \emph{$l$-past equivalent}. The set of $l$-past equivalent vertices to $v$ is denoted by $[v]_l$. Thus, for every $x\in \Omega_l$ there exists $v\in E^0$ such that $[v]_l=x$.

We will determine the cycles of the above defined Boolean dynamical system $(\B,\Labe(E^1),\theta)$. Let $(\alpha,A)$ be a cycle, where $\alpha=\alpha_1\cdots\alpha_n\in\Labe^n$ and $\emptyset \neq A\subseteq \Ri_\alpha=r(\alpha)$. Then $\emptyset \neq r(A,\alpha^k)\subseteq r(\alpha^k)$ for every $k\in\N_0$, and given $\emptyset \neq B\subseteq r(A,\alpha^k)$ we have that $B\cap r(B,\alpha)\neq \emptyset$.

Given $l\in\{0,\ldots,n-1\}$ we define $A_l:=r(A,\alpha_1\cdots \alpha_l)$, where $A_0:=A$ and $\alpha_0:=\alpha$. Then since $(\alpha,A)$ is a cycle without exits we have that $(A_l,\alpha_{l+1}\cdots \alpha_n\alpha_1\cdots\alpha_l)$ is also cycle and $\Delta_{A_l}=\{\alpha_{l+1}\}$ for every $0\leq l\leq n-1$. Then, as it is shown in the proof of Proposition \ref{cycles_top}, given any $B\subseteq A_l$ we have that $r(B,\alpha_{l+1}\cdots \alpha_n\alpha_1\cdots\alpha_l)=B$ for every $0\leq l\leq n-1$. Therefore, that given $0\leq l\leq n-1$ and  $v\in A_l$ then  $v\in r((\alpha_{l+1}\cdots \alpha_n\alpha_1\cdots\alpha_l)^k)$ for every $k\in \N$. In particular, given any $v\in A_l$ and $k\in\N$ with $[v]_k\subseteq A_l$ we have that $r([v]_k,\alpha_{l+1}\cdots \alpha_n\alpha_1\cdots\alpha_l)=[v]_k$, and $r([v]_k,\beta)=\emptyset$ whenever $\beta\neq \alpha_{l+1}$.

Then $(A,\alpha)$ with $A\in\B$ and $\alpha\in \Labe^n$ is a cycle without exits if given $l\in\{0,\ldots,n-1\}$ and $v\in A_l$ there exists $e\in E^n$ with  $r(e)=v$ such that $v\sim_k s(e)$ for every $k\in \N$. Moreover this path $e\in E^n$ must satisfy $\Labe(e)=\alpha_{l+1}\cdots \alpha_n\alpha_1\cdots\alpha_l$.  Observe that if $(E,\Labe,\B)$ is left-resolving the above $e$ is unique.

Now easily follows that if there exists a cycle without exits then the label graph is not disagreeable (see \cite[Definition 5.1]{BPII}). Thus, if the labelled graph $(E,\Labe,\B)$ is disagreeable then the associated Boolean dynamical system $(\B,\Labe(E^1),\theta)$ has no cycle without exits, whence satisfies condition $(L_{\B})$. 

The authors cannot prove the converse, that is $(E,\Labe,\B)$ is disagreeable when $(\B,\Labe(E^1),\theta)$ satisfies condition $(L_{\B})$.

Finally to describe when $(\B,\Labe(E^1),\theta)$ is cofinal we use Proposition \ref{prop8} and the fact that $\B=\Breg$, whence $(\B,\Labe(E^1),\theta)$ is cofinal if and only if  given $A,B\in\B\setminus\{\emptyset\}$ and $\alpha\in \Labe^{\infty}$ there exist $n\in\N$ and $\lambda_1,\ldots,\lambda_k\in \Labe^*$ such that $r(A,\alpha_{[1,n]})\subseteq \bigcup_{i=1}^kr(B,\lambda_i)$. It is then clear that $(E,\Labe,\B)$ is cofinal (see \cite[Definition 6.1]{BPII}) when $(\B,\Labe(E^1),\theta)$ is cofinal.
\end{exem}

\begin{exem} Now, we will construct a unital Boolean Cuntz-Krieger algebra with infinitely generated K-theory. Let us define the Boolean algebra $$\B:=\{F\subseteq \Z: F\text{ finite}\}\cup\{\Z\setminus F:F\text{ finite}\}\,,$$ and let $\Labe:=\{\alpha_i\}_{i\in \Z}\cup\{\beta\}$. Then, given $A\in \B$, we define the actions
$$\theta_{\alpha_i}(A)=A+i=\{x+i:x\in A\}\qquad\text{for every }i\in \Z$$
$$\theta_\beta(A)=\left\{ \begin{array}{ll}\Z & \text{if }0\in A \\ \emptyset & \text{otherwise,}\end{array}\right.$$
and then $\mathcal{R}_{\alpha_i}=\mathcal{R}_\beta=\D_{\alpha_i}=\D_\beta=\Z\in \B$ for every $i\in \Z$. Thus, $(\B,\Labe,\theta)$ is a Boolean dynamical system. Then $\Clab$ is a unital $C^*$-algebra, and since $(\B,\Labe,\theta)$ satisfies condition $(L_\B)$ and there are non-trivial hereditary and saturated ideals $\Clab$ is simple by Theorem \ref{theor_simple}. Since $\Breg=\emptyset$, it follows from Theorem \ref{K-theory} that $$K_0(\Clab)=\left(\bigoplus_{i\in\Z}\Z\right)^1 \text{ and } K_1(\Clab)=0\,.$$
Therefore, since $\Clab$ is unital and has non-finitely generated $K$-theory.
\end{exem}

\begin{exem} Let $X$ be a Cantor set, and let $Y,Z\subseteq X$ be compact clopen subsets, and let 
$\varphi:Y\rightarrow Z$ be a homeomorphism. Let $\bar{\varphi}:C(Z)\rightarrow C(Y)$ the induced isomorphism. We define $\B$ as the Boolean algebra of the compact and clopens of $X$, and $\Labe=\{\alpha\}$ with the single action $\theta_\alpha:\B\rightarrow \B$ defined as $\theta_\alpha(A):=\varphi^{-1}(A)$ for every $A\in \B$. Whence $\theta_\alpha$ has compact range, with $\mathcal{R}_\alpha=Y$, and compact domain because $\theta_\alpha(Z)=Y$. Then $\Clab$ is generated by projections $\{p_A\}_{A\in \B}$ and a partial isometry $s_\alpha$ such that 
$$p_As_\alpha=s_\alpha p_{\varphi^{-1}(A)},\qquad s_\alpha^*s_\alpha=p_Y\qquad\text{and}\qquad s_\alpha s_\alpha^*=p_Z\,.$$
since $Z\in \Breg$. Then $\Clab$ is isomorphic to the partial automorphism crossed product $C^*(C(X),\bar{\varphi})$ (see \cite{Exel2}).

%Observe, that in this situation all the cycles are of the form $\alpha^n$ for $n\in\N$, and hence given $A\in\B$ we have that $\theta_{\alpha^n}(A)=\varphi^{-n}(A)$. Hence 

The Boolean dynamical system $(\B,\Labe,\theta)$ satisfies condition $(L_\B)$ if and only if for every $A\subseteq Y\cap Z$ such that $\varphi^{-n}(A)\neq\emptyset$, there exists $k\in \N$ and $\emptyset\neq B\subseteq\varphi^{-k}(A)$ such that $\varphi^{-1}(B)\cap B=\emptyset$, and it is cofinal if given $A,B\in\B\setminus \emptyset$ there exist  $n_1,\ldots,n_k\in\Z$ such that $A\subseteq \bigcup_{i=1}^k\varphi^{n_i}(B)$.

In particular, if $\varphi:X\rightarrow X$ is a homeomorphism then $\Clab\cong C(X)\times_{\bar{\varphi}} \Z$. Thus, $(\B,\Labe,\theta)$ satisfies condition $(L_\B)$ if for every $\emptyset\neq A\in \B$ there exists $\emptyset\neq B\subseteq A$ with $B\cap \varphi^{-1}(B)=\emptyset$.

Moreover, observe that if $\varphi$ is minimal, i.e. all the orbits are dense, then $(\B,\Labe,\theta)$ satisfies condition $(L_\B)$ and minimality. The converse is also obvious.
\end{exem}

\begin{exem}{(Algebras associated with one-sided subshifts)}
In this section we are going to study the $C^*$-algebra associated with a general one-sided subshift \cite{Car,CarSil}. Given a one-sided subshift $(X,\sigma)$ with a finite alphabet $\Sigma$, we define the subsets
$$C(\alpha)=\{ x\in X:\,\alpha x\in X\}$$
for $\alpha\in\Sigma^*$, where $C(\emptyset)=X$. Let $\B_X$ be the minimal Boolean subalgebra of $2^X$ generated by the subsets $\{C(\alpha):\,\alpha\in\Sigma^*\}$.

Given $l\in\N$ and $x,y\in X$, we say that $x$ and $y$ are $l$-past equivalent, written $x\sim_l y$, if given $z\in \Sigma ^{\leq l}$ we have that 
$$zx\in X\qquad\text{if and only if}\qquad zy\in X\,.$$
We denote by $[x]_l$ the set of all the point of $X$ that are $l$-past equivalent to $x$. Observe that 
$$[x]_l:=\bigcap_{\{\alpha\in \Sigma^{\leq l}:\, \alpha x\in X\}} C(\alpha)\setminus\left( \bigcup_{\{\beta\in \Sigma^{\leq l}:\, \beta x\notin X\}}C(\beta) \right)\in\B_X\,,$$
and that $\B_X$ is generated by $\{[x]_l:x\in X\,,l\in\N\}$.

 Now given $\alpha\in\Sigma$ we define $\theta_\alpha$ as the action that extends $\theta_\alpha(C(a))=C(a\alpha)$ for $a\in \Sigma^*$ to $\B_X$. Observe that $\Ri_\alpha=\theta_\alpha(C(\emptyset))=C(\alpha)$, so $\theta_\alpha$ is an action with compact range and domain. Then $C^*(\B_X,\Sigma,\theta)$ is the universal algebra generated by $\{p_A\}_{A\in\B_X}$ and $\{s_\alpha\}_{\alpha\in \Sigma}$ satisfying:
\begin{enumerate}
\item The map $C(\alpha)\rightarrow p_{C(\alpha)}$ for $\alpha\in \Labe^*$, extends to a map of Boolean algebras,
\item $s^*_\alpha s_\alpha=p_{C(\alpha)}$ for every $\alpha\in \Sigma$,
\item $p_As_\alpha=s_\alpha p_{\theta_\alpha(A)}$ for $A\in\B_X$,
\item given $A\in(\B_X)_{reg}$ then
$$p_A=\sum_{\alpha} s_\alpha p_{\theta_\alpha(A)}s^*_\alpha\,.$$ 
\end{enumerate} 

First observe that given $\alpha,\beta\in \Sigma$ using $(3)$ we have that $s^*_\alpha s_\alpha s_\beta=s_\beta s^*_\beta s^*_\alpha s_\alpha s_\beta$, and then $s^*_\alpha s_\alpha s_\beta s^*_\beta=s_\beta s^*_\beta s^*_\alpha s_\alpha s_\beta s^*_\beta=s_\beta s^*_\beta  s_\beta  s^*_\beta s^*_\alpha s_\alpha=s_\beta  s^*_\beta s^*_\alpha s_\alpha$. The converse is also clear, so condition $(3)$ is equivalent to $s^*_\alpha s_\alpha s_\beta s^*_\beta= s_\beta  s^*_\beta s^*_\alpha s_\alpha$.

Finally, since $|\Sigma|<\infty$ we have that $C(\emptyset)\in(\B_X)_{reg}$, and then by condition $(4)$ it follows
$$1=p_{C(\emptyset)}=\sum_\alpha s_\alpha p_{C(\alpha)} s^*_\alpha=\sum_\alpha s_\alpha s^*_\alpha\,.$$
From the observation that every nonempty element $C(\alpha)\in\B_X$ is regular, it follows that every $\emptyset\neq A\in\B_X$ is also regular. Thus, condition $(4)$ is equivalent to  
$$1=\sum_\alpha s_\alpha s^*_\alpha\,.$$

Therefore, we have that $C^*(\B_X,\Sigma,\theta)$ is isomorphic to the $C^*$-algebra $\mathcal{O}_X$ associated to a one-sided subshift. 

Our task will be first to determine the cycles of the Boolean system $(\B_X,\Sigma,\theta)$. Let $(\alpha,A)$ be a cycle, where $\alpha=\alpha_1\cdots\alpha_n\in\Sigma^n$ and $\emptyset \neq A\subseteq \Ri_\alpha=C(\alpha)$. Then $\emptyset \neq \theta_{\alpha^k}(A)\subseteq C(\alpha^k)$ for every $k\in\N_0$. This means that the cycle $\alpha^\infty\in X$ since $X$ is closed. Moreover, by the proof of Proposition \ref{cycles_top} we have that  $(\theta_\alpha)_{|A}=\text{Id}$, whence for every $x\in A$ we have that  $\alpha^kx\in X$. 

Now suppose that $(\alpha,A)$ is a cycle without exits,  this implies that if $\alpha x\in A$ then $x=\alpha^\infty$. In particular $A=\{\alpha^\infty\}$. Conversely, if $\{\alpha^\infty\}\in \B_X$ we have that $(\alpha,\{\alpha^\infty\})$ is a cycle without exits.

Observe that $\{\alpha^\infty\}\in\B_X$ is equivalent to say that there exists $l\in\N$ such that $[\alpha^\infty]_l=\{\alpha^{\infty}\}$. Then $\alpha^\infty$ is said to be \emph{isolated in past equivalence} \cite{Car0}.

%Now by Proposition \ref{prop7} and Theorem \ref{theor7} it follows that $(\B_X,\Sigma,\theta)$ is cofinal if and only if $\beta:T  \acts  \widehat{\mathcal{E}}_{\text{tight}}$ is irreducible, i.e., given $\xi\in \widehat{\mathcal{E}}_{\text{tight}}$ the set $\{\beta_s(\xi):s\in T\}$ is dense in $\widehat{\mathcal{E}}_{\text{tight}}$. 

Given $x\in X$ we define the map $\xi_x:\mathcal{E}(T)\rightarrow\{0,1\}$ such that $\xi_x(S_\alpha p_AS^*_\alpha )=1$ if $x=\alpha x'$ for some $x'\in A$, and $0$ otherwise.  It is clear that $\xi_x$ is an ultrafilter and that $\{\Cu_x:x\in X\}$  is dense in  $\widehat{\mathcal{E}}_{\infty}$. Then it is only necessary to check cofinality of $(\B_X,\Sigma,\theta)$ for the characters of the form $\xi_x$ for $x\in X$. Let us suppose that $(\B_X,\Sigma,\theta)$ is cofinal, then giving $x,y\in X$ and $l\in \N$ there exist $\alpha,\beta\in\Labe^*$ such that $s_\alpha p_{\theta_\beta([y]_l)}s^*_\alpha\in \xi_x$, that is equivalent to there exists $z\in X$ such that $z\sim_l y$ and that $\sigma^m(x)=\sigma^n(z)$ for some $m,n\in\N$. Then we say that $X$ is \emph{cofinal in past equivalence} \cite{Car0}. 

Therefore, we have that $C^*(\B_X,\Sigma,\theta)$ is simple if and only if there is no cyclic point isolated in past equivalence and $X$ is cofinal in past equivalence \cite{Car0}.
\end{exem}

\section*{Acknowledgments}

The collaboration that crystallised in this paper started at the conference "Classification of $C^*$-algebras, flow equivalence of shift spaces, and graph and Leavitt path algebras", hosted at the University of Louisiana at Lafayette from 11 to 15 May 2015. The authors are indebted to the organizers and to the NSF for providing this opportunity. Part of this work was done during a visit of the third author to the Institutt for Matematiske Fag, Norges Teknisk-Naturvitenskapelige Universitet (Trondheim, Norway). The third author thanks the center for its kind hospitality.

\end{document}